%% file: FAL_arxiv_v2.tex
\documentclass{siamltex}
\usepackage{amsmath,amssymb,amsfonts,graphicx,xcolor,enumerate}
\usepackage{multirow,fullpage,rotating,url,changepage,algorithm,algorithmic}

\input defs_garud_edit.tex

\def\alg#1{\textsc{Algorithm~#1}}

\title{A First-order Augmented Lagrangian Method\\ for Compressed Sensing}
\author{N. S. Aybat\thanks{IEOR Department, Columbia University.
    Email: {\tt nsa2106@columbia.edu}} \and G. Iyengar\thanks{IEOR
    Department, Columbia University.
    Email: {\tt gi10@columbia.edu}}}
\begin{document}
\maketitle
\begin{abstract}
We propose a first-order augmented Lagrangian algorithm (FAL) for solving the basis pursuit problem. FAL computes a solution to this problem
by inexactly solving a sequence of $\ell_1$-regularized least squares sub-problems. These
sub-problems are solved using an infinite memory proximal
gradient algorithm wherein each update reduces to ``shrinkage'' or constrained
``shrinkage''. We show that FAL converges to an optimal solution of the
basis pursuit problem whenever the solution is unique, which is the case with very high probability for compressed sensing problems.
We construct a parameter sequence such that the corresponding FAL
iterates are $\epsilon$-feasible and $\epsilon$-optimal for all $\epsilon>0$ within $\cO\left(\log\left(\epsilon^{-1}\right)\right)$ FAL iterations.
Moreover, FAL requires at most $\cO(\epsilon^{-1})$ matrix-vector multiplications of the form $Ax$ or $A^Ty$ to compute an
$\epsilon$-feasible, $\epsilon$-optimal solution. We show that FAL can be easily extended to solve the basis pursuit denoising problem when there is a non-trivial level of noise on the measurements. We report the results of numerical experiments
comparing FAL with the state-of-the-art solvers for both noisy and noiseless compressed sensing problems.
A striking property of FAL that we observed in the numerical experiments
with randomly generated instances when there is no measurement noise was that FAL {\em always}
correctly identifies the support of the target signal without any
thresholding or post-processing, for moderately small error tolerance
values.
\end{abstract}
\section{Introduction}
In this paper we propose a new first-order augmented Lagrangian algorithm
to solve the {\em basis pursuit} problem
\begin{align}
\min_{x\in\reals^n}\|x\|_1 \hbox{ subject to } Ax=b, \label{ch2_eq:l1_minimization}
\end{align}
where $\ell_1$-norm $\norm{x}_1 := \sum_{i=1}^n \abs{x_i}$, $x_i$ denotes
the $i$-th component of $x\in\reals^n$, $b\in \reals^m$,
$A \in \reals^{m\times n}$, with $m \ll n$, and  $\rank(A) = m$, i.e. $A$
has full row rank.
%This
%problem can be reformulated into a linear program~(LP) and therefore, can,
%in theory, be solved efficiently.
% It is clear that \eqref{ch2_eq:l1_minimization} can be reformulated as a
% linear program.
The basis pursuit problem
%form~\eqref{ch2_eq:l1_minimization} have recently attracted a lot
%of attention % since they serve as the basis for
%since they
appears in the context of
%a new signal processing paradigm known as
{\em compressed
  sensing}~(CS)~\cite{Can06_1J, Can06_2J, Can06_3J,Don06_1J} where the goal is to recover a sparse signal $x_*$ from a small set of
linear measurements or transform values $b = Ax_*$.
%The sparsest signal consistent with
%NP-hard $\ell_0$-minimization problem
%\begin{align}
%\min_{x\in\reals^n}\|x\|_0 \hbox{ subject to } Ax=b, \label{ch2_eq:l0_minimization}
%\end{align}
%where the $\ell_0$ norm $\norm{x}_0 = \sum_{i=1}^n \ones(x_i \neq 0)$ and
%$\ones(.)$ is the indicator function that
%outputs $1$ if its argument is true; and $0$ otherwise.
Candes, Romberg and Tao~\cite{Can06_1J, Can06_2J, Can06_3J} and
Donoho~\cite{Don06_1J} have shown that
when the target signal $x_*$  is  $s$-sparse, i.e. only $s$ of the $n$
components are non-zero, and the measurement
matrix $A\in\Re^{m\times n}$ satisfies some regularity conditions, the
sparse signal $x_*$ can be recovered by solving the basis pursuit
problem~\eqref{ch2_eq:l1_minimization}
with high probability provided that the number of measurements $m = \cO(s\log(n))$.
% $1-\cO(e^{-\gamma n})$ for
% some $\gamma>0$.
The basis pursuit problem is a linear program~(LP). Therefore, computing the sparsest solution to the set of linear
equations $Ax = b$, which is an NP-hard
problem for general $A$, can be done efficiently, in theory, by solving an LP.
% The results of Candes, Romberg, Tao
% \cite{Candes-Romberg-Tao-05}, Donoho \cite{Donoho-06}, and others
% assert that, under some reasonable conditions on $A$  regarding
% its size and incoherence relative to the basis of $u$, the % sparsest
% signal $u$ %  that
% % satisfies $b=Au$
% can be recovered by solving
%  %-- i.e., $u$ that solves the combinatorial problem $\min \{\|u\|_0
%  %: Au=b \}$, where $\|u\|_0 \equiv |\{i: u_i \neq 0\}|$ --
% %the convex $\ell_1$-minimization problem (also, referred to as the basis
% %pursuit problem
% the convex $\ell_1$-minimization problem (\ref{ch2_BP}) (the so-called basis
% pursuit problem (BP)

However, in typical CS applications the signal dimension $n$ is large, e.g.
$n \approx 10^6$, and  the LP~\eqref{ch2_eq:l1_minimization} is often
ill-conditioned. Consequently, general purpose simplex-based LP
solvers  are unable to solve the LP. Moreover, the constraint matrix $A$ is typically dense. Therefore, general purpose
interior point methods that require factorization of $A^TA$
are not practical for solving  LPs
arising in CS applications.

%
% In this paper we are interested in computing sparse solutions for a
% system of equations $A x = b$
% where $b \in \reals^m$, $x \in \reals^n$, $A \in \reals^{m\times n}$ and
% the number of equations $m \ll n$.
%  The sparsest solution $x$ satisfying the
% equation $Ax = b$ can be recovered by solving the following
% Define $\ell_0$-norm $\norm{x}_0$ of
% the vector $x$ as
% \eq
% \norm{x}_0 =,
% \en
% where $x(i)$ denotes the $i$-th component of the vector $x$, and the
% indicator function $\ones(x(i) \neq 0)$ takes the value $1$ if $x(i) \neq
% 0$, and $0$ otherwise. Then
% It was known that optimal solution to
% (\ref{ch2_eq:l0_minimization}) can often be recovered by solving the
% This problem can be reformulated into a linear program and therefore, can,
% in theory, be solved efficiently.
% This result has given given rise to a new signal compression methodology
% known as {\em compressed sensing}~(CS). In the CS methodology a high
% dimensional $s$-sparse signal $x$ is compressed into $m = \cO(s\log(n))$
% dimensional code $b = Ax$ by taking $m$ linear measurements. Then as
% stated in \cite{Can06_1J, Can06_2J, Can06_3J}, the  original
% signal $x$ can be exactly recovered with an overwhelming probability by solving an LP.
%
% \subsection*{Previous results}
% In practice, solving the decoding LP~(\ref{ch2_eq:l1_minimization}) is
% hard. This is because the constraint matrix $A$ is large and dense, and
% the LPs are often ill-conditioned. Thus, general purpose LP solvers are
% not able to efficiently solve the CS LPs.

On the other hand, in CS applications the $A$, although dense, still has a lot of
structure. In many applications, $A$ is a partial transform matrix, e.g.
%can be exploited by special
%purpose algorithms. In many applications $A$
partial discrete cosine transform~(DCT), a partial wavelet, or a partial
pseudo-polar Fourier  matrix. Therefore, the  matrix-vector product $Ax$
and $A^Ty$ can be computed in $\cO(n\log(n))$ time using either the
Fast Fourier Transform~(FFT) or forward and backward Wavelet transforms.
This fact has been recently exploited by a number of \emph{first-order}
algorithms. In this paper, we propose a new first-order augmented Lagrangian
algorithm for the basis pursuit problem. Since the basic steps in a first-order algorithm
are the matrix-vector multiplications in the form of $ Ax $ and $ A^T y $, we will report complexity in terms of the
number of such matrix-vector multiplications required to solve the problem.

\subsection{Previous work on first-order algorithms for compressed sensing}
When the measurement data $b$ contains a non-trivial level of noise, one can solve
\begin{align}
\min_{x\in\reals^n}\bar{\lambda}\|x\|_1+\|Ax-b\|_2^2, \label{ch2_eq:l1_minimization_normsq}
\end{align}
for an appropriately chosen $\bar{\lambda}>0$ depending on the noise level to recover the sparse target signal $x_*$ with some error proportional to the noise on $b$~\cite{Can05_1J}. On the other hand, when there is no noise on the measurements, $b$, or when the noise level is low, one can solve \eqref{ch2_eq:l1_minimization_normsq} for a fixed small $\bar{\lambda}>0$, which can be viewed as a penalty approximation to \eqref{ch2_eq:l1_minimization}.

In~\cite{Wri07_1J} Figueiredo, Nowak and Wright proposed the GPSR algorithm
that uses gradient projection method with
Barzilai-Borwein steps to solve \eqref{ch2_eq:l1_minimization_normsq}.
% Since
% $x\in\reals^n$ can be written as $x=x^+-x^-$, where $x^+, \x^- \in
% \reals^n_+$ solving \eqref{ch2_eq:l1_minimization_normsq} is equivalent to solving a bound
% constrained quadratic problem. Thus, the convex set onto which iterate
% gradients are projected is the positive orthant.
% Hence, the projections
% that GPSR uses are trivial and gives the search direction in closed
% form.
% The algorithm proceeds by computing an approximately optimal solution
% $x^*(\lambda_k)$  starting from the approximately optimal solution
% $x^\ast(\gamma_{k-1})$ for the previous
% value of $\gamma$.
% To solve \eqref{ch2_eq:l1_minimization_normsq},
Hale, Yin and Zhang~\cite{Yin07_1R,Yin08_1J} proposed to solve
\eqref{ch2_eq:l1_minimization_normsq} via the fixed point
continuation~(FPC) algorithm that embeds the soft-thresholding (IST) algorithm~\cite{Dau04_2J} in a continuation scheme on $\lambda$, i.e. FPC begins with $\lambda>\bar{\lambda}$ and gradually decreases it to $\bar{\lambda}$, to recover the sparse solution of \eqref{ch2_eq:l1_minimization_normsq}.
% They use operator-splitting to show that the optimal solution of
% \eqref{ch2_eq:l1_minimization_normsq} satisfies a fixed-point equation
% $x^\ast(\gamma) = F(x^\ast(\gamma))$, where
% $F$ is a composition of a gradient descent step operator and shrinkage
% operator and is a contraction mapping.
% ; therefore, the mapping $x_{k+1}=F(x_k)$ converges to the
% optimal solution.
% In \cite{Yin07_1R,Yin08_1J}, continuation strategy is
% also implemented, instead of solving just one problem with a target value
% of $\gamma$.
Wen, Yin, Goldfarb and Zhang~\cite{Wen09_1R} improved the performance of
FPC by adding an active set~(AS) step. Please note that GPSR, FPC and FPC-AS only converge to the optimal solution of
\eqref{ch2_eq:l1_minimization_normsq}, \emph{not} to the optimal solution of \eqref{ch2_eq:l1_minimization}. Hence, when there is no noise or when it is low, the solutions produced by these algorithms are only good approximations to $x_*$.

% In the FPC-AS
% algorithm, once the
% shrinkage iterations produce a ``candidate'' solution $x_k$,
% %  to determine
% % the active set, second phase of the algorithm begins. Active set here
% % corresponds the variables which are zero or would-be zero at the optimal
% % solution. On the active set,
% the non-smooth  objective function $\|x\|_1$  is
% replaced by $\mbox{sign}(x_k)^Tx$ and the constraints $\mbox{sign}(x_k)(i)x(i)\geq
% 0$ are added. Since the objective function on the active set is smooth,
% one can use conjugate gradients or quasi-Newton
% methods to minimize the objective.
% The  multiplier $\gamma$ is updated
% in a continuation scheme, and fixed
% point iterations are done again to determine the new active set. Until the
% desired stopping condition is met, this two step method is repeated.
Yin, Osher, Goldfarb and Darbon~\cite{Yin08_2J}
solve~\eqref{ch2_eq:l1_minimization} using a Bregman iterative
regularization scheme that involves a sequence of problems of the form
$
\min_{x\in\reals^n}\bar{\lambda}\|x\|_1+\frac{1}{2}\|Ax-b^{(k)}\|_2^2, \label{ch2_eq:bregman}
$
%for a fixed penalty multiplier $\bar{\lambda}>0$,
where $b^{(k)}$ are obtained by suitably updating the measurement vector
$b$, and each sub-problem is solved using FPC. For the basis pursuit problem, the so-called Bregman iterative
regularization procedure is nothing but the classic augmented Lagrangian method.
%for solving
%the unconstrained subproblems~(\ref{ch2_eq:bregman}).
% Typically, one is only required to
% solve a very e, and for each
% outer iteration FPC~\cite{Yin08_1J} is called to solve subproblem
% (\ref{ch2_eq:bregman}).
The algorithm YALL1 developed by Yang and Zhang~\cite{Yang09}, which is an alternating direction algorithm, is able to
solve the basis pursuit problem
\eqref{ch2_eq:l1_minimization},  the penalty formulation
\eqref{ch2_eq:l1_minimization_normsq}, and the basis pursuit denoising problem
\begin{equation}
\label{ch2_eq:relaxed_l1}
\begin{array}{rl}
    \mbox{min} &  \|x\|_1,\\
    \mbox{s.t.} & \norm{Ax-b}_2 \leq \delta.
  \end{array}
\end{equation}
Bregman iteration based  methods~\cite{Yin08_2J} and YALL1~\cite{Yang09}
provably converge to the
optimal solution of the basis pursuit
problem~(\ref{ch2_eq:l1_minimization}); however, their convergence rates
are
unknown.

Other algorithms for $\ell_1$-regularized least squares problem \eqref{ch2_eq:l1_minimization_normsq} include an iterative interior-point solver~\cite{Boy07_1J},
% by Kim, Koh, Lustig, Boyd and Gorinevsky,
and an accelerated projected gradient method~\cite{Dau08_1J}.
% by Daubechies, Fornasier and Loris
Van den Berg and Friedlander~\cite{Ber08_1J} proposed SPGL1 to solve the penalty formulation~\eqref{ch2_eq:relaxed_l1} by solving a sequence of LASSO sub-problems $\Psi(t) = \{\|Ax-b\|_2^2: \norm{x}_1
\leq t\}$ %  by non-monotone
% spectral projected gradient algorithm and updating the
where parameter $t$ is updated by a
Newton step. This algorithm
% \begin{equation}
%   \label{ch2_eq:l1_minimization_relaxed}
%   \begin{array}{rl}
%     \mbox{min} &  \|x\|_1,\\
%     \mbox{s.t.} & \norm{Ax-b}_2 \leq \epsilon.
%   \end{array}
% \end{equation}
%The algorithm in~\cite{Ber08_1J}
provably converges to the optimal
solution of~\eqref{ch2_eq:relaxed_l1};
however, the convergence rate is again unknown.
% Moreover, in
% \cite{Ber08_1J}, Van den Berg and Friedlander adapted the nonmonotone
% spectral projected gradient algorithm to solve the LASSO problem
% \eq
% \min_{\{x: \norm{x}_1 \leq t\}} \|Ax-b\|_2^2.
% \en

Aybat and Iyengar~\cite{AybatI09:SPA} have proposed a first-order smoothed
penalty algorithm~(SPA) to solve the basis pursuit problem.
%SPA employs Nesterov's
%optimal gradient method for non-smooth convex
%optimization~\cite{Nesterov05} to solve the penalty sub-problems.
SPA iterates $\{x^{(k)}\}_{k\in\integers_+}$ are computed by inexactly solving a sequence of
smoothed penalty problems of the form
\eq
\min_{\norm{x}_2\leq \etak}\big\{\lk \pmk(x) + \fnk (x)\big\},
\en
where $\pmk(x)$ is a smooth approximation of $\norm{x}_1$, $\fnk(x)$
is a smooth approximation of $\norm{Ax-b}_2$ and $\etak$ is a suitably
chosen bound on the $\ell_2$-norm of an
optimal solution of the $k$-th sub-problem.
%(In~\cite{AybatI09:SPA} it is shown that the constraint $\norm{x}_1 \leq
%\eta$ is not needed for establishing the convergence bounds.)
SPA calls Nesterov's optimal algorithm for simple
sets~\cite{Nesterov04,Nesterov05} to solve the sub-problems.
% The main computational steps in each iteration Nesterov's
%algorithm involves computing the gradient $\lk \grad \pmk(x) + \grad \fnk(x)$,
%and solving two problems of the form
%\[
%\min_{\norm{x}_2\leq \eta}\big\{ c^Tx  + \frac{L}{2} \norm{x-z}_2^2\Big\}.
%\]
% where in one problem $c$ is the gradient of the current iterate and in the
% other, it is a linear combination of gradients of all the iterates.
%The optimal solution to this proximal gradient descent step reduces to
%projecting the unconstrained solution $ z - \frac{1}{L}c$  onto an
%$\ell_2$-ball.
%Since $ \ell_2 $-projection is $ \cO(n) $, the complexity of each
%Nesterov iteration in SPA is given by complexity of
%computing the  gradient $\lk\grad \pmk(x) +
%\grad \fnk(x)$ which is, itself, dominated by the complexity $A^T(Ax-b)$, and can,
%therefore, be computed efficiently in the CS context.
% Since the infeasibility is penalized by the appropriately smoothed version
% of $\norm{Ax-b}_2$, the iterates with
% small infeasibility are penalized harsher in SPA as compared to algorithms
% employing the smooth penalty $\norm{Ax-b}_2^2$, and, therefore,
% SPA is expected to converge faster, especially when the tolerance on feasibility is
% small.
SPA iterates provably converge to an optimal solution $x_*$ of the
basis pursuit problem
whenever it is unique.
%, i.e. $x_* \in \argmin\{\norm{x}_1: Ax
%= b\}$ and that
Moreover, for all small
enough $\epsilon$,  SPA requires $\cO(\sqrt{n}\epsilon^{-\frac{3}{2}})$
matrix-vector multiplies to compute an
$\epsilon$-feasible, i.e. $\norm{Ax^{(k)} - b}_2 \leq \epsilon$, and
$\epsilon$-optimal, $\left|~\norm{x^{(k)}}_1 - \norm{x_*}_1\right| \leq
\epsilon$ iterate.

%$ x^{(k)} $ is , where $ c_A $ denote the cost
%of computing the gradient $ A^T(Ax-b) $. Thus, in the case where $ A $ is either a partial DCT, DFT or Wavelent
%matrix, the complexity of computing an $ \epsilon$-feasible and $ \epsilon$-optimal iterate is $
%\cO(n^\frac{3}{2}\log(n)\epsilon^{-\frac{3}{2}}) $
%In addition, there exists a priori fixed set of parameter such that,
%for {\em all} small
%enough $\epsilon$, the SPA
%$\epsilon$-feasible, i.e. $\norm{Ax^{(k)} - b}_2 \leq \epsilon$, and
%$\epsilon$-optimal, $\left|~\norm{x^{(k)}}_1 - \norm{x_*}_1\right| \leq
%\epsilon$ for all $ k \geq k_{\epsilon} = \lceil\tilde{\cO}(\sqrt{n}\log(\frac{1}{\epsilon}))\rceil$ and
%the computational complexity of computing $ x^{k_{\epsilon})} $ is $\tilde{\cO}(\sqrt{n}c_A\epsilon^{-\frac{3}{2}})$,
%where $ c_A $ denote the cost of computing $ A^T(Ax-b) $. Thus, complexity of computing an $ \epsilon$-feasible and $
%\epsilon$-optimal iterate
%to compute $  $ where the complexity of each Nesterov update is either
%$\cO(n\log(n))$ or $\cO(n^2)$ depending on the complexity of vector multiplication with $A$.

Becker, Bobin and Cand\`{e}s~\cite{Can09_4J} have proposed NESTA
for solving the formulation~\eqref{ch2_eq:relaxed_l1} (NESTA can also be used to solve the basis pursuit problem~\eqref{ch2_eq:l1_minimization} by setting $\delta$ to $0$).
NESTA calls Nesterov's optimal gradient method for non-smooth convex
functions~\cite{Nesterov05} to solve the sub-problems.
% solving the relaxed $\ell_1$-
% minimization problem  and, by setting $\delta = 0$, the basis pursuit
% problem~\eqref{ch2_eq:l1_minimization}.
When the matrix $A$ is orthogonal, i.e. $AA^T = I$, NESTA requires $
\cO(\sqrt{n}\epsilon^{-1})$  matrix-vector multiplications to compute a
feasible $\epsilon$-optimal iterate to \eqref{ch2_eq:relaxed_l1}.
When the matrix $A$ is a partial transform matrix, i.e. $Ax$ and $A^Ty$ is
$\cO(n\log(n))$, but $A$ is not orthogonal, NESTA, in general, needs to
compute $(A^TA + \mu I)^{-1}$, and therefore, its $\cO(n^3)$ per iteration complexity
is quite  prohibitive for practical applications. Moreover, the sequence of NESTA iterates does not converge an optimal solution of \eqref{ch2_eq:relaxed_l1} but to a solution of a smooth approximation of \eqref{ch2_eq:relaxed_l1}.

% et tranform or a psuedo-polar tranform that appears in the context of CT
% scans~\cite{Ser09_1J}.

%,which is also an adaptation of Nesterov's optimal gradient method for
%non-smooth convex functions~\cite{Nesterov05} to
%solve~\eqref{ch1_eq:l1_minimization_relaxed}.
%NESTA, which is a direct application of Nesterov's
%algorithm~\cite{Nesterov05} for non-smooth convex functions to
%\eqref{ch2_eq:relaxed_l1}, computes an $\epsilon$-optimal solution for
%\eqref{ch2_eq:relaxed_l1} in
%$\cO(\frac{1}{\epsilon})$ Nesterov iterations, where each Nesterov
%iteration involves computing the gradient of
%suitably smoothed version of the $\ell_1$-norm $\norm{x}_1$ and  solving
%an optimization problem of the form
%\begin{equation}
%  \label{ch1_eq:NESTA-update}
%  \min_{x\in\reals^n:\norm{Ax-b}_2 \leq \delta} \Big\{c^Tx + \frac{L}{2}
%  \norm{x-z}_2^2\Big\}.
%\end{equation}
%NESTA can be embedded in a continuation scheme that allows one to obtain
%a solution with any desired accuracy.

\subsection{New results}
In this paper  we propose a first-order augmented Lagrangian~(FAL)
algorithm that solves the basis pursuit problem
% \eqref{ch2_eq:l1_minimization}
by inexactly solving a sequence of optimization problems of the form
\begin{equation}
  \label{ch2_eq:augmented_lagrangian_subproblem}
  \min_{x\in\reals^n:~\norm{x}_1\leq\etak}\Big\{\lambda^{(k)}
  \norm{x}_1-\lambda^{(k)} (\theta^{(k)})^T(Ax-b) +
  \frac{1}{2} \norm{Ax-b}_2^2\Big\},
\end{equation}
for an appropriately chosen sequence $\{(\lambda^{(k)},\theta^{(k)},\etak)\}_{k \in \integers_+}$.
%where $\eta>0$ is a bound on $\norm{x_*}_1$.
Each of these sub-problems are solved using a variant (see
Figure~\ref{ch1_alg:pga}) of the infinite-memory proximal
gradient algorithm in~\cite{Tseng08} (see, also FISTA~\cite{Beck09_1J} and Nesterov infinite-memory algorithm~\cite{Nesterov05}). Each
update in this proximal gradient algorithm involves computing the gradient $A^T(Ax-b)$
of the quadratic term $\frac{1}{2}\norm{Ax-b}_2^2$ and computing two
%``shrinkage''\cite{Dau04_2J} or
constrained ``shrinkage'' (see Equation~\ref{ch2_eq:constrained_shrinkage_problem}), which require $\cO(n\log(n))$ work.
Hence, the complexity of each  update is dominated by computing the gradient $A^T(Ax-b)$ or equivalently two matrix-vector multiplies.

In Theorem~\ref{ch2_thm:limit-point} in
Section~\ref{ch2_sec:theory} we prove that every limit point
of the FAL iterate sequence is an optimal solution of
\eqref{ch2_eq:l1_minimization}. Thus, the FAL iterates converge to the
optimal solution when the solution is unique. In
Theorem~\ref{ch2_thm:epsilon_convergence} we show that % there exists a
% priori  fixed sequence $\{\lk: k \geq 1\}$ such that
for all $\epsilon>0$, the
FAL iterates $x^{(k)}$ are $\epsilon$-feasible, i.e. $\norm{Ax^{(k)} -
  b}_2 \leq \epsilon$, and $\epsilon$-optimal, $\left|~\norm{x^{(k)}}_1 -
  \norm{x_*}_1\right| \leq \epsilon$, for $k \geq
\cO(\log(\epsilon^{-1}))$. Moreover, FAL requires at most $
\cO(n\epsilon^{-1})$ matrix-vector  multiplications to compute an $\epsilon
$-feasible, $\epsilon $-optimal solution to \eqref{ch2_eq:l1_minimization}. Thus, the overall complexity of FAL
computing an $\epsilon $-feasible and $\epsilon $-optimal iterate is $\cO(n^2\log(n)\epsilon^{-1})$ in the CS context. And in Section~\ref{ch1_sec:extensions}, we briefly discuss how to extend FAL to solve the noisy recovery problem $\min_{x\in\reals^n}\{\norm{x}_1:~\norm{Ax-b}_2\leq\delta\}$.

In Section~\ref{ch2_sec:computations} we report the results of our
experiments with FAL. We tested FAL on randomly
generated problems both with and without measurement noise and also on
known hard instances of the CS problems.
We compared the performance of FAL with SPA~\cite{AybatI09:SPA},
NESTA~\cite{Can09_4J}, FPC~\cite{Yin08_1J}, FPC-AS~\cite{Wen09_1R}, YALL1~\cite{Yang09} and SPGL1~\cite{Ber08_1J}.
On  randomly generated problem instances FAL is at least two times faster
than all the other solvers.
On known hard CS instances the run times of FAL were of the same order of
magnitude as the best solver; but FAL was able
to identify significantly sparser solutions.
We also observed that for all randomly generated instances with no measurement noise  FAL {\em
  always} correctly identified the support of the target signal $x_*$, without any
additional heuristic thresholding, when the error tolerance was set to moderate values.
Once the support is known, the signal $\xbp$
can often  be very accurately computed by solving a set of linear
equations. Moreover, although the bound in Theorem~\ref{ch2_thm:epsilon_convergence} implies
that FAL requires $\cO\left(\frac{1}{\epsilon}\right)$ matrix-vector
multiplies to compute an $\epsilon$-feasible, $\epsilon$-optimal solution,
in practice we observed that FAL required only
$\cO\left(\log\left(\frac{1}{\epsilon}\right)\right)$
matrix-vector multiplies to compute an $\epsilon$-feasible,
$\epsilon$-optimal solution. %  to the basis pursuit
% problem as opposed to $\cO\left(\frac{1}{\epsilon}\right)$  worst case
% theoretical bound.

FAL is superior to SPA~\cite{AybatI09:SPA} both in terms of the
theoretical guarantees as well as practical
performance on the basis pursuit problem. However, FAL explicitly uses the
structure of the $ \ell_1 $-norm and is, therefore, restricted to basis
pursuit and related problems. On the other hand, SPA can be extended
easily to solve the following much larger class non-smooth convex
optimization problems:
\[
\begin{array}{rl}
\min & \max_{u \in U} \{\phi(x,u)\},\\
\mbox{subject to} & \norm{Ax-b}_\gamma\leq\delta,
\end{array}
\]
where $U$ is a compact convex set and  $\phi:\reals^n\times
U\rightarrow\reals$ is a bi-affine
function~\cite{Nesterov05}, and $ \gamma
\in \{1,2,\infty\} $. This class includes as special cases, basis pursuit,
matrix games with side constraints, group
LASSO, and problems of the form $ \min\{ \sum_{k = 1}^p \norm{B_k
x}_1: Ax = b\} $ that appears in the context of reconstructing a piecewise flat sparse image.

%is a  function such that $\phi(.,u)$ is convex and
%differentiable for all $u\in U$ and $\phi(x,.)$ is linear for all
%$x\in\reals^n$. To give an example, finding saddle
%points for matrix games~\cite{Nesterov05} is of the form we discussed above:
%$\min_{x\in\Delta_1,~Ax=b}\max_{u\in\Delta_2}\{u^TQx+c_1^x+c_2^Tu\}$ for
%some $Q\in\reals^{m\times n}$,
%$c_1\in\reals^n$ and $c_2\in\reals^m$, where $\Delta_1=\{x\in\reals^n:
%\mathbf{1}^Tx=1, x\geq0\}$ and
%$\Delta_2=\{u\in\reals^m: \mathbf{1}^Tu=1, x\geq0\}$ are unit
%simplices. While SPA can solve the saddle point problem
%above, FAL cannot.
%Convergence properties: We show that
%FAL converges to an optimal solution $x_*$ of
%\eqref{ch2_eq:l1_minimization}, i.e. $x_* \in \argmin\{\norm{x}_1: Ax= b\}$.

%This paper is organized as follows. As a preliminary, in
%Section~\ref{ch2_sec:solver} we give slightly modified
%version of Algorithm~3 in~\cite{Tseng08} that can be used to solve the
%problem $\min\{\lambda\norm{x}_1+f(x):\
%\norm{x}_1\leq\eta\}$, where $f$ is a smooth function with a Lipschitz
%continuous gradient (Please note that this
%problem is a more general form of the subproblems given in
%\eqref{ch2_eq:augmented_lagrangian_subproblem}). In
%Later,
%in Section~\ref{ch2_sec:implementation}
%we discuss some implementation details and in
%Section~\ref{ch2_sec:computations} we report the results from our
%numerical experiments comparing FAL with other algorithms well known in
%the $\ell_1$-minimization literature.
\section{Preliminaries}
\label{ch2_sec:solver}
In this section we state and briefly discuss the details of a particular
variant of Tseng's Algorithm~3 in
\cite{Tseng08} that we use in FAL. Algorithm~3~\cite{Tseng08} computes $
\epsilon $-optimal solutions for the
optimization problem
\begin{align}
\label{ch1_eq:tseng_problem}
  \min_{x\in F} p(x)+f(x),
\end{align}
where $f$, $p$ and $F$ satisfy the following conditions.
\begin{equation}
  \label{eq:apg}
  \begin{array}{cl}
    \text{(i)} & p:\reals^n\rightarrow\reals \text{ proper,
      lower-semicontinuous~(lsc) and convex function, and $\dom p$ closed},\\
    \text{(ii)} & f:\reals^n\rightarrow\reals \text{ proper, lsc, convex
      function, differentiable on
      an open set containing $\dom p$},\\
    \text{(iii)} & \text{$\grad f$
      is Lipschitz continuous on $\dom p$ with constant $L$},\\
    \text{(iv)} &  F \cap \argmin_{x
  \in \reals^n}\{p(x)+f(x)\} \neq \emptyset.
  \end{array}
\end{equation}
We refer to a function $h:\reals^n\rightarrow\reals$ as a {\em prox}
function if $h$ is differentiable and strongly
convex function with convexity parameter $c > 0$,
i.e. $h(y)\geq h(x)+\grad h(x)^T(y-x)+\frac{c}{2}\norm{y-x}_2^2$ for
all $x,y\in\dom h$.
%
%shows each step of Algorithm~3 in \cite{Tseng08} for solving \eqref{ch1_eq:tseng_problem}
%when $\{X_\ell\}_{\ell\in\integers_+}$ in \cite{Tseng08} are set as
%follows: $X_\ell:=F$ for all $\ell\geq 0$. One
%way
%to choose $x^{(\ell+1)}$ in line~\ref{ch1_algeq:tseng_x} of
%\textbf{Algorithm~APG} is to set it to
%$\hat{x}^{(\ell+1)}$, i.e. $x^{(\ell+1)} \gets
%\hat{x}^{(\ell+1)}$. Another way, proposed in \cite{Tseng08}, is to
%set
%$x^{(\ell+1)}$ by solving an unconstrained problem,
%i.e. $x^{(\ell+1)}\gets \argmin_{x\in\reals^n}H^{(\ell)}(x)$,
%where $H^{(\ell)}(x)$ is defined in line~\ref{ch1_algeq:H} of
%\textbf{Algorithm~APG}. Clearly both selections satisfy
%the condition $H^{(\ell)}(x^{(\ell+1)})\leq
%H^{(\ell)}(\hat{x}^{(\ell+1)})$ given in line~\ref{ch1_algeq:tseng_x} of
%\textbf{Algorithm~APG}.
\begin{figure}[!htb]
    \rule[0in]{6.5in}{1pt}\\
    \textsc{Algorithm APG}$(p, f, L, F, x^{(0)}, h, \textsc{APGstop})$\\
    \rule[0.125in]{6.5in}{0.1mm}
    \vspace{-0.3in}
    \begin{algorithmic}[1]
%    \STATE \textbf{input:} a prox function $h(.)$, $x^{(0)} \in \dom p$,
%    $F$ such that $ F \cap
%\argmin_{x \in \reals^n}\{p(x) + f(x)\} \neq \emptyset$, and a stopping criterion $ S $
      \STATE $u^{(0)} \gets x^{(0)}$, $w^{(0)}\gets \argmin_{x\in\dom
        p}h(x)$, $\vartheta^{(0)}\gets 1$, $\ell\gets 0$
    \WHILE{($\textsc{APGstop}$ is \FALSE)}
    \STATE $v^{(\ell)} \gets (1-\vartheta^{(\ell)}) u^{(\ell)} +
    \vartheta^{(\ell)} w^{(\ell)}$
    \STATE $w^{(\ell+1)} \gets \argmin\left\{\sum_{i=0}^\ell
      \frac{1}{\vartheta^{(i)}}\left(p(z)+\nabla f(v^{(i)})^T z\right)
      + \frac{L}{c}~h(z): z\in F\right\}$ \label{ch1_algeq:tseng_z}
    \STATE $\hat{u}^{(\ell+1)}\gets (1-\vartheta^{(\ell)})u^{(\ell)}+\vartheta^{(\ell)} w^{(\ell+1)}$
    \STATE $H^{(\ell)}(x):=p(x)+\nabla f(v^{(\ell)})^T x +
    \frac{L}{2}\norm{x-v^{(\ell)}}_2^2$ \label{ch1_algeq:H}
    \STATE $u^{(\ell+1)} \gets \argmin\{H^{(\ell)}(x):~x\in F\}$ \label{ch1_eq:modified_x}
%    =\argmin_{x\in F}\left\{p(x)+ \nabla f(y^{(\ell)})^T x +
%    \frac{L}{2}\norm{x-y^{(\ell)}}_2^2\right\}
%    Choose $x^{(\ell+1)}$ \textbf{such that} $H^{(\ell)}(x^{(\ell+1)})\leq H^{(\ell)}(\hat{x}^{(\ell+1)})$
%\label{ch1_algeq:tseng_x}
    \STATE $\vartheta^{(\ell+1)}\gets\frac{\sqrt{(\vartheta^{(\ell)})^4-4(\vartheta^{(\ell)})^2}-(\vartheta^{(\ell)})^2}{2}$
    \STATE $\ell \gets \ell + 1$
    \ENDWHILE
    \RETURN $ u^{(\ell)}$ \textbf{or} $ v^{(\ell)}$ depending on $\textsc{APGstop}$
    \end{algorithmic}
    \rule[0.25in]{6.5in}{0.1mm}
    \vspace{-0.5in}
    \caption{Accelerated Proximal Gradient Algorithm}\label{ch1_alg:pga}
\end{figure}
%In this paper, we have used the following update rule for $x$ in
%line~\ref{ch1_algeq:tseng_x} of
%\textbf{Algorithm~APG}:
%Please note that unlike two update rule suggested in~\cite{Tseng08}, explained above, we solve a {\em constrained}
%problem in \eqref{ch1_eq:modified_x} to compute $x^{(\ell+1)}$, i.e. $x^{(\ell+1)}\gets\argmin_x\{H^{(\ell)}(x):~x\in
%F\}$. , which is
%needed for the proof of Lemma~\ref{ch1_lem:tseng_corollary}. Therefore, the result of
%Lemma~\ref{ch1_lem:tseng_corollary} continues to hold with this modification.

Our variant of Algorithm~3 in~\cite{Tseng08} is displayed
in~Figure~\ref{ch1_alg:pga}. \textsc{Algorithm APG} takes as
input the functions $f$ and  $p$, a prox function $h$, the set $F$, an initial iterate $x^{(0)}$ and a stopping criterion \textsc{APGstop}.
%Next, Lemma~\ref{ch1_lem:tseng_corollary} gives the iteration complexity
%of \textbf{Algorithm~APG}.
\begin{lemma}
\label{ch1_lem:tseng_corollary}
Suppose $p$, $f$ and $F$ satisfy (\ref{eq:apg}). Let $h$ be a
prox function on an open set containing $\dom p$ and $\min_{x\in \dom
  p}h(x)=0$. Fix $\epsilon>0$ and let
$\{u^{(\ell)},v^{(\ell)},w^{(\ell)}\}_{\ell\in\integers}$ be the
sequence generated by \textsc{Algorithm~APG} displayed in
Figure~\ref{ch1_alg:pga}. Then $p(u^{(\ell+1)})+f(u^{(\ell+1)})\leq
\min_{x\in\reals^n}\{p(x)+f(x)\}+\epsilon$ for all
$\ell\geq\sqrt{\frac{4L}{c\epsilon}~h(x_*)}-1$,
where $x_* \in \argmin_{x \in \reals^n}\{p(x) + f(x)\}$.
\end{lemma}
\begin{proof}
Corollary~3 in \cite{Tseng08} implies this result provided that $H(u^{(\ell+1)}) \leq
H(\hat{u}^{(\ell+1)})$ holds for all $ \ell \geq 0 $. Using induction, it is
easy to show that this is true when the update  rule
\eqref{ch1_eq:modified_x} is used for all $\ell\geq 1$.
%\textbf{Algorithm~APG} still holds when update rule
%\eqref{ch1_eq:modified_x} is used for all $\ell\geq 1$
\end{proof}

\section{Convergence Properties of FAL}
\label{ch2_sec:theory}
% In this section, we describe FAL
% % our First-order Augmented Lagrangian~(FAL) algorithm
% %for $\ell_1$ recovery
% and  prove the main convergence results for the algorithm. Let $\sigma_{\min}(A)$ (resp. $\sigma_{\max}(A)$)  denote
% the smallest (resp. largest)  singular value of the measurement matrix $A\in\reals^{m\times n}$, and let
% $\kappa(A):=\sigma_{\max}(A)/\sigma_{\min}(A)$ denote the condition
% number of $A$. We assume that $A$ has full row
% rank; consequently, $A^T$ has full column rank. Let $\xbp\in\Re^n$
% denote an optimal solution of the
% $\ell_1$-minimization problem~\eqref{ch2_eq:l1_minimization}. We also
% assume that $\eta>0$, a bound on the
% $\ell_1$-norm of $\xbp$, i.e. $\norm{\xbp}_1 \leq \eta$, is given.

In this section, we describe FAL and prove the main convergence
results for the algorithm.  The outline of FAL is given in
Figure~\ref{ch2_alg:fal}.
\textsc{Algorithm FAL} takes as inputs a sequence of
$\{(\lk,\epsk,\tk)\}_{k\in \integers_+}$, a starting point $x^{(0)}$
and a bound $\eta$ on the $\ell_1$-norm of an optimal solution
$x_*$ of the basis pursuit problem. %  and $\norm{x^{(0)}}_1 \leq
% \eta$.
One such bound $\eta$ can be computed as
follows. Let $\tilde{x} = \argmin\{\norm{x}_2: Ax =
b\}=A^T(AA^T)^{-1}b$. Clearly, $\norm{x_*}_1 \leq \eta := \norm{\tilde{x}}_1$.
We will next describe each of the steps in this
outline.

An augmented Lagrangian function for the basis pursuit
problem~\eqref{ch2_eq:l1_minimization} can be written as
\eq
P(x) := \lambda \norm{x}_1 - \lambda \theta^T(Ax-b) +
\frac{1}{2}\norm{Ax-b}_2^2,
\en
where $\lambda$ is the penalty parameter and $\theta$ is a dual
variable for the constraints $Ax = b$.
From Lines~\ref{fal:func-def}-\ref{fal:F-def} in
Figure~\ref{ch2_alg:fal} it follows that in the $k$-th iteration
of \textsc{Algorithm FAL} we inexactly minimize the augmented
Lagrangian function
\eq
  P^{(k)}(x)  :=  \lambda^{(k)} \norm{x}_1 +
  \frac{1}{2}\norm{Ax-b-\lambda^{(k)}\theta^{(k)}}_2^2 = \lambda^{(k)}
  \norm{x}_1-\lambda^{(k)} (\theta^{(k)})^T(Ax-b) +
  \frac{1}{2} \norm{Ax-b}_2^2-\frac{1}{2}\norm{\lk\thetak}_2^2,
\en
over the set $F^{(k)} := \{x: \norm{x}_1 \leq \eta^{(k)}\}$ using
\textsc{Algorithm APG}  with the
prox function $h^{(k)} = \frac{1}{2}\norm{x-x^{(k-1)}}_2^2$.

Recall that when using \alg{APG} we need to ensure that $F^{(k)} \cap
\argmin_{x \in \reals^n}\{P^{(k)}(x)\} \neq \emptyset$.  Let
$\xkopt\in\argmin_{x\in\reals^n}P^{(k)}(x)$. Since
$P^{(k)}(x_*^{(k)})\leq P^{(k)}(\xbp)$, $A\xbp = b$ and
$\norm{\xbp}_1\leq\eta$, we have $\norm{x_*^{(k)}}_1\leq
\etak$.  Thus, $x_{\ast}^{(k)} \in F^{(k)}$.

Next, we discuss the stopping criterion set in Line~\ref{fal:stop} in
Figure~\ref{ch2_alg:fal}.
The convexity parameter of the prox-function $h^{(k)}$ is 1. Hence,
Lemma~\ref{ch1_lem:tseng_corollary} establishes that
\eq
P^{(k)}(u^{(\ell)})\leq P^{(k)}(\xkopt) + \epsilon^{(k)}\ \ \hbox{ for
}\ \ \ell \geq
\sqrt{\frac{2L}{\epsilon^{(k)}}}~\norm{x_*^{(k)}-x^{(k-1)}}_2,
\en
where $\{u^{(\ell)}\}_{\ell\in\integers_+}$ is the sequence of u-iterates when \textsc{Algorithm APG} is applied to the $k$-th subproblem and $L$ denotes the Lipschitz constant of the  gradient $\grad
f^{(k)}(x) $.
Since
$\norm{x_*^{(k)}}_1\leq \etak$ and $\norm{.}_1\geq\norm{.}_2$,
triangle inequality implies that
\begin{equation}
\norm{x_*^{(k)}-x^{(k-1)}}_2 \leq
\norm{x_*^{(k)}}_1+\norm{x^{(k-1)}}_2 \leq
\etak+\norm{x^{(k-1)}}_2. %\label{ch2_eq:proxbound}
\end{equation}
Since $\grad f^{(k)}(x)= A^T(Ax-b-\lambda^{(k)} \theta^{(k)})$ it
follows that $L = \sigma_{\max}^2(A)$, where $\sigma_{\max}(A)$ denote
the largest singular value of  $A$. Thus, it follows that
\[
  \sqrt{\frac{2L}{\epsilon^{(k)}}}~\norm{x_*^{(k)}-x^{(k-1)}}_2
  \leq \sigma_{\max}(A)(\etak + \norm{x^{k-1}}_2)
    \sqrt{\frac{2}{\epsilon^{(k)}}} =: \ell_{\max}^{(k)}.
\]
Consequently, it follows that the stopping criterion \textsc{APGstop}
in Line~\ref{fal:stop} ensures that the iterate $x^{(k)}$ satisfies
one of the following two conditions
\begin{equation}
 \label{ch2_eq:inner-stopping-condition}
 \begin{array}{rl}
   \text{(a)} & P^{(k)}(x^{(k)})  \leq  \min_{x\in\reals^n}
   P^{(k)}(x) + \epsilon^{(k)},\\
   \text{(b)} & \exists g \in\partial P^{(k)}(x)|_{x^{(k)}} \text{ with }
  \norm{g}_2 \leq \tau^{(k)},  \\
\end{array}
\end{equation}
where $\partial P^{(k)}(x)|_{x^{(k)}}$ denotes the set of
 subgradients of the function $P^{(k)}$ at $\xk$. When $\ell\geq\ell^{(k)}_{\max}$ in \textsc{APGstop} holds, \alg{APG} returns $u^{(\ell)}$; otherwise, when $\exists g \in\partial P^{(k)}(x)|_{v^{(\ell)}}$ with $\norm{g}_2 \leq \tau^{(k)}$, \alg{APG} returns $v^{(\ell)}$. And we set $\xk$ to what \alg{APG} returns.

In Line~\ref{fal:dual}, we update the dual variables $\theta^{(k)}$ in a manner that is standard for augmented Lagrangian algorithms. This
completes the description of \alg{FAL} displayed in Figure~\ref{ch2_alg:fal}.
\vspace{-0.5cm}
\begin{figure}[t]
    \rule[0in]{6.5in}{1pt}\\
    \textsc{Algorithm FAL$\big(\{(\lk,\epsk,\tk)\}_{k\in \integers_+}, x^{(0)}, \eta\big)$}\\
    \rule[0.125in]{6.5in}{0.1mm}
    \vspace{-0.3in}
    \begin{algorithmic}[1]
    %\STATE \textbf{input:} $ %and $\xbp\in\argmin_{x\in\reals^n}\{\norm{x}_1:~Ax=b\}$
    \STATE $\theta^{(1)} = 0$, $k\gets 0$, $L \gets \sigma_{\max}(AA^T)$
    \WHILE{(\textsc{FALstop} is \FALSE)}
    \STATE $k \gets k + 1$
    \STATE \label{fal:func-def} $p^{(k)}(x) := \lambda^{(k)} \norm{x}_1$,
    \quad $f^{(k)}(x) := % \lambda^{(k)} \norm{x}_1 +
    \frac{1}{2} \norm{Ax-b-\lambda^{(k)}\theta^{(k)}}_2^2$
    \STATE \label{fal:h-def}  $h^{(k)}(x) := \frac{1}{2}\norm{x-x^{(k-1)}}_2^2$
    \STATE $\etak\gets\eta+\frac{\lk}{2}\norm{\thetak}_2^2$
    \STATE \label{fal:F-def} $F^{(k)} := \{x\in\reals^n:\ \norm{x}_1\leq\etak\}$
    \STATE  $\ell^{(k)}_{\max} \gets  \sigma_{\max}(A)(\etak +
      \norm{x^{(k-1)}}_2) \sqrt{\frac{2}{\epsilon^{(k)}}}$
    \STATE \label{fal:stop} $\textsc{APGstop}  := \{\ell \geq \ell^{(k)}_{\max}\} $ \OR
    $\left\{  \exists g \in\partial P^{(k)}(x)|_{v^{(\ell)}} \text{ with }
    \norm{g}_2 \leq \tau^{(k)} \right\}$
    \STATE $x^{(k)} \gets \textsc{APG}\Big(p^{(k)},
    f^{(k)}, L, F^{(k)}, x^{(k-1)}, h^{(k)}, \textsc{APGstop}\Big)$
    \STATE \label{fal:dual} $\theta^{(k+1)} \gets \theta^{(k)} -
    \frac{Ax^{(k)}-b}{\lambda^{(k)}}$ \label{ch2_alg:lagrangian_update}
    \ENDWHILE
    \RETURN $x_{\text{sol}} \gets x^{(k)}$
    \end{algorithmic}
    \rule[0.25in]{6.5in}{0.1mm}
    \vspace{-0.5in}
    \caption{Outline of First-Order Augmented Lagrangian
      Algorithm~(FAL)}\label{ch2_alg:fal}
\end{figure}
\newpage
% In the rest of this section, we establish theoretical properties of FAL.
% %At this point, we would like to define some variables that we will be
% %using in the entire chapter.
% Given $\epsilon>0$, let $N_{\rm FAL}(\epsilon)$, FAL iteration count,
% be the number of times \textbf{Algorithm~APG} is called within
% \textbf{Algorithm~FAL} until an $\epsilon$-feasible and
% $\epsilon$-optimal solution to \eqref{ch2_eq:l1_minimization} is
% found. During the $k$-th call, \textbf{Algorithm~APG} inexactly solves
% \eqref{ch2_eq:augmented_lagrangian_subproblem_2}, which we call the
% ``$k$-th subproblem". Let $\Nk$ denote the number of iterations
% \textbf{Algorithm~APG} does until one of the stopping criteria
% \eqref{ch2_eq:inner-stopping-condition} holds. Finally, let $N_{\rm
%   inner}$ be the total number \textbf{Algorithm~APG} iterations until
% an $\epsilon$-feasible and $\epsilon$-optimal solution to
% \eqref{ch2_eq:l1_minimization} is found, i.e. $N_{\rm
%   inner}=\sum_{k=1}^{N_{\rm FAL}(\epsilon)}\Nk$.

In the result below we establish that every limit point of the FAL
iterate sequence $\{\xk\}_{k \in \integers}$, is an optimal solution
of the basis pursuit problem.
\begin{theorem}
  \label{ch2_thm:limit-point}
  Fix $x^{(0)}\in\reals^n$, $\eta>0$ such that
  $\eta\geq\norm{\xbp}_1$ and a sequence of parameters
  $\{(\lk,\epsk,\tk)\}_{k\in\integers_+}$ such that
  \begin{enumerate}[(i)]
  \item penalty parameters, $\lambda^{(k)} \searrow 0$,
  \item approximate optimality parameters, $\epsilon^{(k)}\searrow 0$
    such that $\frac{\epsilon^{(k)}}{(\lambda^{(k )})^2}\leq B_1$ for all $k\geq 1$,
  \item subgradient tolerance parameters, $\tk\searrow 0$ such that
    $\frac{\tk}{\lk} \leq B_2$ for all $k \geq 1$, and
    $\frac{\tk}{\lk}\rightarrow 0$ as $k\rightarrow 0$.
  \end{enumerate}
  Let $\mathcal{X}=\{\xk\}_{k \in
    \integers_+}$ denote the  iterates computed by \alg{FAL} for this
  set of parameters.   Then, $\mathcal{X}$ is a bounded sequence and
  any limit point $\bar{x}$ of  $\cX$ is an optimal solution of the basis pursuit
  problem~(\ref{ch2_eq:l1_minimization}).
\end{theorem}
\begin{proof}
% Lemma~\ref{ch1_lem:tseng_corollary} guarantees that in the $k$-th FAL
% iteration \textbf{Algorithm~APG} displayed in
% Figure~\ref{ch1_alg:pga}
Since $\ell_{\max}^{(k)}$ is finite for all $k \geq 1$, it follows that
 the   sequence $\cX$ exists.

As a first step towards establishing that $\cX$ is bounded, we
establish a uniform bound on the sequence of dual multipliers
$\{\theta^{(k)}\}_{k\in\integers_+}$.
Suppose in the  $k$-th FAL iteration \alg{APG} terminates with the
iterate $x^{(k)}$ satisfying \eqref{ch2_eq:inner-stopping-condition}(a). Then
Corollary~\ref{ch2_cor:grad_norm_bound} applied to $P^{(k)}(x)=\lk
\norm{x}_1+ \norm{Ax-b-\lk\thetak}_2^2$ guarantees that
\eq
\norm{A^T(Ax^{(k)}-b-\lambda^{(k)}\theta^{(k)})}_\infty \leq
\sqrt{2\epsilon^{(k)}}\ \sigma_{\max}(A)+\lambda^{(k)}.
\en
Instead, if the iterate  $x^{(k)}$ satisfies
\eqref{ch2_eq:inner-stopping-condition}(b), i.e. there exists
$q^{(k)}\in\partial \norm{x}_1|_{x^{(k)}}$ such that
$\norm{\lambda^{(k)}q^{(k)} +
  A^T(Ax^{(k)}-b-\lambda^{(k)}\theta^{(k)})}_2\leq\tau^{(k)}$, then
\begin{align}
\norm{A^T(Ax^{(k)}-b-\lambda^{(k)}\theta^{(k)})}_\infty \leq \tau^{(k)} +
\lambda^{(k)} \norm{q^{(k)}}_\infty \leq \tau^{(k)} +
\lambda^{(k)}, \label{ch2_eq:subgradient_feasibility_bound}
\end{align}
where the second inequality follows from the fact that
$\norm{q^{(k)}}_{\infty} \leq 1$ for all $q^{(k)} \in \partial
\norm{x}_1|_{x^{(k)}}$.

Since $\theta^{(1)}=0$,
$\theta^{(k+1)}=\theta^{(k)}-\frac{Ax^{(k)}-b}{\lambda^{(k)}}$ for all
$k\geq 1$ and $A$ has full row-rank, it follows that
\begin{eqnarray}
\norm{\theta^{(k+1)}}_2 & \leq &  \frac{1}{\lk\sigma_{\min}(A)}
\norm{A^T(Ax^{(k)}-b-\lambda^{(k)}\theta^{(k)})}_2, \nonumber\\
&  \leq  &
\frac{\sqrt{n}}{\sigma_{\min}(A)}
\left(\max\left\{\sigma_{\max}(A)\sqrt{\frac{2\epsilon^{(k)}}{(\lambda^{(k)})^2}},
    \ \frac{\tau^{(k)}}{\lambda^{(k)}}\right\}+1\right), \quad
\forall k \geq 1. \label{ch2_eq:theta_induction}
\end{eqnarray}
% \eq
% \norm{\theta^{(k+1)}}_2 \leq \frac{1}{\lk\sigma_{\min}(A)}
% \norm{A^T(Ax^{(k)}-b-\lambda^{(k)}\theta^{(k)})}_2 \leq
% \frac{\sqrt{n}}{\sigma_{\min}(A)} \(\frac{\tau^{(k)}}{\lambda^{(k)}}+1\Big).
% \en
% Thus,
% $\{\theta^{(k)}\}_{k \in \integers}$ satisfies
% \eqref{ch2_eq:theta_induction}.
% Since $\frac{\tau^{(k)}}{\lambda^{(k)}}\rightarrow 0$, there exists
% $B_2\in\reals_+$ such that $\frac{\tau^{(k)}}{\lambda^{(k)}}\leq B_2$
% for all $k\geq 1$.
The bounds
$\frac{\epsilon^{(k)}}{(\lambda^{(k)})^2}\leq B_1$, and
$\frac{\tau^{(k)}}{\lambda^{(k)}}\leq B_2$, together with
\eqref{ch2_eq:theta_induction} imply that
\begin{equation}
  \norm{\theta^{(k)}}_2\leq   B_\theta :=
  \frac{\sqrt{n}}{\sigma_{\min}(A)}\left(\max\{\sqrt{2B_1}\
    \sigma_{\max}(A),\ B_2\}+1\right), \quad \forall k> 1.
  \label{ch2_eq:theta_bound}
\end{equation}
From this bound, it follows that
\begin{equation}
\label{ch2_eq:fal_iter_bound}
\norm{x^{(k)}}_1\leq \etak  = \eta + \frac{\lk}{2}
\norm{\theta^{k}}_2^2 \leq B_x := \eta + \frac{1}{2} \lambda^{(1)}B_{\theta}^2.
\end{equation}
Thus, $\cX$ is a bounded sequence and it has a
limit point. Let $\bar{x}$ denote any limit point and let
$\cK\subset\integers_+$ denote a subsequence such that $\lim_{k \in \cK} x^{(k)} = \bar{x}$.

% Therefore,
% \begin{equation}
%   \lim_{k\rightarrow\infty}\lambda^{(k)}\theta^{(k)}=0, \label{ch2_eq:lambda_theta_product}
% \end{equation}
% and
% \begin{equation}
% \lim_{k\rightarrow\infty}\lambda^{(k)}\norm{\theta^{(k)}}_2^2=0.
%  \label{ch2_eq:lambda_theta_sq_product}
% \end{equation}
% Also,  $\frac{\epsilon^{(k)}}{(\lambda^{(k)})^2}\leq B_1$,  for all $k\geq 1$, implies that
% \begin{align}
% \lim_{k\rightarrow\infty}\frac{\epsilon^{(k)}}{\lambda^{(k)}}=0.
% \label{ch2_eq:epsilon_lambda_division}
% \end{align}

Suppose that there exists a further sub-sequence $\cK_a\subset\cK$ such that
for all $k\in\cK_a$  calls to \alg{APG} terminates with an
iterate $\xk$ satisfying \eqref{ch2_eq:inner-stopping-condition}(a).
% Since $\xkopt\in\argmin_{x\in\reals^n}P^{(k)}(x)$, it follows that
% $P^{(k)}(x_*^{(k)})\leq P^{(k)}(\xbp)$.
% Thus, \eqref{ch2_eq:inner-stopping-condition}(a) implies that for all $k\in\cK_a$,
% $P^{(k)}(x^{(k)})\leq P^{(k)}(\xbp)+\epsilon^{(k)}$.
Then, for $k\in\cK_a$, we have that
\begin{eqnarray}
  \lefteqn{\norm{x^{(k)}}_1  \leq  \frac{P^{(k)}(\xk)}{\lk} \leq \frac{P^{(k)}(x^{(k)}_{\ast}) +
    \epsilon^{(k)}}{\lambda^{(k)}}
  \leq  \frac{P^{(k)}(\xbp)+\epsilon^{(k)}}{\lambda^{(k)}}} \nonumber\\
  & = &  \norm{\xbp}_1 +
  \frac{\lambda^{(k)}}{2}\norm{\theta^{(k)}}_2^2 +
  \frac{\epsilon^{(k)}}{\lambda^{(k)}}
  \leq  \norm{\xbp}_1 +
  \frac{1}{2}\lambda^{(k)}B_{\theta}^2 + \lk B_1, \label{ch2_eq:inexact_minimizer_bound}
\end{eqnarray}
where the first inequality follows from the fact $f^{(k)}(x) \geq 0$,
second follows from the stopping condition, the third follows from the fact that
$P^{(k)}(x_*^{(k)})\leq P^{(k)}(\xbp)$, the equality follows from the
fact that $A\xbp = b$, and the last inequality follows from the
bounds $\norm{\theta^{(k)}}_2 \leq B_{\theta}$ and
  $\frac{\epsilon^{(k)}}{(\lk)^2} \leq B_1$.
Since $\lk \rightarrow 0$, taking the limit % of both sides of
% \eqref{ch2_eq:inexact_minimizer_bound} along
% the subsequence $\cK_a$ and using
% \eqref{ch2_eq:lambda_theta_sq_product} and
% \eqref{ch2_eq:epsilon_lambda_division},
along $\cK_a$ we get
\begin{align}
\norm{\bar{x}}_1=\lim_{k\in\cK_a}\norm{x^{(k)}}_1\leq\norm{\xbp}_1 +
\lim_{k\in\cK_a}\left\{\lk\left(\frac{1}{2}B_\theta^2 + B_1\right)\right\}
= \norm{\xbp}_1. \label{ch2_eq:inexact_minimizer_optimality}
\end{align}
Next, consider feasibility of the limit point $\bar{x}$.
\begin{eqnarray*}
\norm{A\xk-b-\lk\thetak}_2^2 &\leq& \Pk(\xk) \leq \Pk(\xkopt)+\epsk \leq \Pk(\xbp)+\epsk,\\
&\leq& \lk\norm{\xbp}_1+\frac{1}{2}\norm{\lk\thetak}_2^2+\epsk\leq\lk\norm{\xbp}_1+\frac{1}{2}(\lk B_\theta)^2+\epsk,
\end{eqnarray*}
where the first inequality follows from the fact $\lk\norm{\xk} \geq 0$, the third follows
from the fact that $P^{(k)}(\xkopt) \leq P^{(k)}(\xbp)$, the fourth follows from the fact $A\xbp = b$,
and the last follows from the bound $\norm{\theta^{(k)}}_2 \leq
B_{\theta}$.
Taking the limit along $\cK_a$, we have
\eq
\frac{1}{2}\norm{A\bar{x}-b}_2^2 \leq 0,
\en
i.e. $A\bar{x} = b$.
Since  $\bar{x}$ is feasible, and $\norm{\bar{x}}_1 \leq
\norm{\xbp}_1$, it follows that $\bar{x}$ is an optimal solution for
the basis pursuit problem~\eqref{ch2_eq:l1_minimization}.

Now, consider the complement case, i.e. there exists $K\in\cK$ such
that for all $k\in\cK_b := \cK \cap \{k
\geq K\}$, calls to \alg{APG} terminate with an iterate
$x^{(k)}$ that satisfies \eqref{ch2_eq:inner-stopping-condition}(b).
For all $k \in \cK_b$,   there exists
$q^{(k)}\in\partial \norm{x}_1|_{x^{(k)}}$ such that
\begin{align}
\norm{\lambda^{(k)}q^{(k)}+A^T(Ax^{(k)}-b
  -\lambda^{(k)}\theta^{(k)})}_2\leq\tau^{(k)}.  \label{ch2_eq:detailed_subgradient_condition}
\end{align}
Then, for all $k \in \cK_b$,
\begin{eqnarray*}
  \norm{A\xk - b}_2 & \leq & \frac{1}{\sigma_{\min}(A)}
  \norm{A^T(A\xk-b)}_2,\\
  &\leq & \frac{1}{\sigma_{\min}(A)} \Big( \norm{A^T(A\xk-b -
    \lk\theta^{(k)}) + \lk q^{(k)}}_2 + \norm{\lk q^{(k)}}_2 + \norm{A^T(\lk\theta^{(k)})}_2\Big) ,\\
  & \leq & \frac{1}{\sigma_{\min}(A)} \big( \tau^{(k)} + \lk
  \sqrt{n} + \sigma_{\max}(A) \lk B_{\theta}\Big),
\end{eqnarray*}
where $\sigma_{\min}(A)$ denotes the smallest non-zero singular value
of $A$. The first inequality follows from the definition of
$\sigma_{\min}(A)$, the second inequality follows from triangle
inequality, and last inequality follows from
\eqref{ch2_eq:detailed_subgradient_condition}, the bound
$\norm{\theta^{(k)}}_2 \leq B_{\theta}$, and fact that
$\norm{q}_2 \leq \sqrt{n}$ for any $q\in\partial
\norm{x}_1|_{x^{(k)}}$. Taking the limit along $\cK_b$, we have
$\norm{A\bar{x}-b}_2\leq 0$, or equivalently $A\bar{x} = b$.

For all $k\in \cK_b$,
$q^{(k)}\in\partial\norm{x}_1|_{x^{(k)}}$, therefore,
$\norm{q^{(k)}}_{\infty} \leq 1$.
Hence, there exists a subsequence $\cK'_b\subset\cK_b$ such
that $\lim_{k\in\cK'_b}q^{(k)}=\bar{q}$ exists. One can easily show that
$\bar{q}\in\partial\norm{x}_1|_{\bar{x}}$. Dividing both sides of
\eqref{ch2_eq:detailed_subgradient_condition} by $\lambda^{(k)}$, we get
\begin{align}
  \norm{q^{(k)}-A^T\theta^{(k+1)}}_2\leq\frac{\tau^{(k)}}{\lambda^{(k)}},
  \label{ch2_eq:subgradient_theta_convergence}
\end{align}
for all $k\in\cK'_b$. Since
$\lim_{k\in\cK'_b}q^{(k)}=\bar{q}$,
$\lim_{k\in\integers_+}\frac{\tau^{(k)}}{\lambda^{(k)}}=0$ and $A$ has
full row rank, it follows that $\{\theta^{(k)}: k \in \cK_b'\}$ is a
Cauchy sequence; therefore,
$\lim_{k\in\cK'_b}\theta^{(k+1)}=\bar{\theta}$ exists.
Taking the limit
of both sides of \eqref{ch2_eq:subgradient_theta_convergence} along $\cK_b'$, we have
\begin{align}
\bar{q}=A^T\bar{\theta}. \label{ch2_eq:subgradient_optimality}
\end{align}
\eqref{ch2_eq:subgradient_optimality} together with that the fact that
$A\bar{x} = b$ and $q \in \partial \norm{\bar{x}}_1$, it follows that
the KKT conditions for optimality is satisfied at $\bar{x}$; thus,
$\bar{x}$ is optimal for the basis pursuit problem.
% together imply that $\bar{x}$ is feasible to the basis pursuit problem,
% i.e. $\min\{\norm{x}_1:Ax=b\}$, and there exists $\bar{\theta}\in\reals^m$ such that
% % together with $\bar{x}$,
% KKT conditions are satisfied at $\bar{x}$. Since the basis
% pursuit problem is convex, we can conclude that $\bar{x}$ is an optimal
% solution of the basis pursuit problem.
\end{proof}

In compressed sensing exact recovery occurs only when
$\min\{\|x\|_1 : Ax=b\}$ has a {\em unique} solution.
The following Corollary establishes that FAL converges to this
solution.
\begin{corollary}
  \label{ch2_cor:unique}
  Suppose the basis pursuit problem $\min\{\|x\|_1 : Ax=b\}$ has  a
  {\em unique} optimal solution~$\xbp$. Let $\{x^{(k)}\}_{k \in \integers_+}$
  denote the sequence of iterates generated by \alg{FAL}, displayed in Figure~\ref{ch2_alg:fal}, corresponding to a sequence
  $\{(\lk,\epsk,\tk)\}_{k\in\integers_+}$ that satisfies all the conditions in Theorem~\ref{ch2_thm:limit-point}. Then
  $\lim_{k\rightarrow\infty}x^{(k)} = \xbp$.
\end{corollary}

Next, we characterize the finite iteration performance of FAL. This
analysis will lead to a convergence rate result in
Theorem~\ref{ch2_thm:epsilon_convergence}.
\begin{theorem}
  \label{ch2_thm:finite}
  Fix $x^{(0)}\in\reals^n$, $\eta>0$ such that
  $\eta\geq\norm{\xbp}_1$ and a sequence of parameters
  $\{(\lk,\epsk,\tk)\}_{k\in\integers_+}$ satisfying all the conditions
  in Theorem~\ref{ch2_thm:limit-point}. In addition, suppose for all $k
  \geq 1$, $\tk \leq c~\epsk$ for
  some $0<c<1$.
  Let $\{x^{(k)}\}_{k \in \integers_+}$
  denote the sequence of iterates generated by \alg{FAL},
  displayed in Figure~\ref{ch2_alg:fal}, for this set of
  parameters. Then, for all $k \geq 1$,
  \begin{enumerate}[(i)]
    \item $\norm{Ax^{(k)}-b}_2 \leq 2B_\theta
      \lambda^{(k)}$, \label{ch2_eq:infeasibility_bound}
    \item $\left|\norm{x^{(k)}}_1 - \norm{x_*}_1\right| \leq \max
      \left\{\left(\frac{B_\theta^2}{2}+B_1~\max\{1,~2c B_x\}\right),
        \frac{\left(\frac{\sqrt{n}}{\sigma_{\min}(A)}+B_\theta\right)^2}{2}\right\}
      \lambda^{(k)}$, \label{ch2_eq:optimality_bound}
  \end{enumerate}
  where $ B_\theta = \frac{\sqrt{n}}{\sigma_{\min}(A)}\left(\max\left\{\sqrt{2B_1}\
      \sigma_{\max}(A),\ B_2\right\}+1\right)$, and $B_x = \eta +
  \frac{1}{2} \lambda^{(1)}B_{\theta}^2$.
  % % where $B_1>0$ is a constant such that
  % % $\frac{\epsilon^{(k)}}{(\lambda^{(k)})^2}\leq B_1$;
  % $B_\theta>0$ is the constant given in \eqref{ch2_eq:theta_bound} and
  % $B_x\geq\max\{1,~\norm{\xk}_1\}$ for all $k\geq 1$.
\end{theorem}
\begin{proof}
First note that we have established the uniform bounds
$\norm{\theta^{(k)}}_2 \leq B_{\theta}$ and $\norm{\xk}_1 \leq B_x$ in
\eqref{ch2_eq:theta_bound}
and \eqref{ch2_eq:fal_iter_bound}, respectively.

The dual update in Line~\ref{fal:dual} of \alg{FAL} implies that
\begin{eqnarray*}
\norm{Ax^{(k)}-b}_2
 \leq  \norm{Ax^{(k)}-b-\lambda^{(k)}\theta^{(k)}}_2+\lambda^{(k)}\norm{\theta^{(k)}}_2
 =  \lambda^{(k)} \norm{\theta^{(k+1)}}_2 + \lambda^{(k)}\norm{\theta^{(k)}}_2
 \leq  2B_\theta \lambda^{(k)},
\end{eqnarray*}
where the last inequality follows the fact that
$\norm{\theta^{(k)}}_2 \leq B_{\theta}$. This establishes~\eqref{ch2_eq:infeasibility_bound}.

The dual update in Line~\ref{fal:dual} of \alg{FAL} also implies that
\eq
  P^{(k)}(x^{(k)})  = \lambda^{(k)}\norm{x^{(k)}}_1 +
  \frac{1}{2}\norm{Ax^{(k)}-b-\lambda^{(k)}\theta^{(k)}}_2^2
   =  \lambda^{(k)} \left(\norm{x^{(k)}}_1 + \frac{\lambda^{(k)}}{2}
     \norm{\theta^{(k+1)}}_2^2\right).
\en
Thus, for all $k\geq 1$,
\begin{equation}
  \norm{\xk}_1 \geq  \frac{P^{(k)}(\xkopt)}{\lk} -
  \frac{\lk}{2}\norm{\theta^{(k+1)}}_2^2.
  \label{ch2_eq:lbd-1}
\end{equation}
Next, we establish a lower bound for $P^{(k)}(\xkopt)$. Consider the
following primal-dual pair of problems:
\begin{eqnarray}
  \begin{array}[t]{rl}
    \min_{x\in\reals^n} & \norm{x}_1\\
    \mbox{subject to} & Ax = b,
  \end{array}
  \quad \quad
  \begin{array}[t]{rl}
    \max_{w\in\reals^m} & b^Tw\\
    \mbox{subject to} & \norm{A^Tw}_{\infty} \leq 1.
  \end{array} \label{ch2_eq:dual_problem}
\end{eqnarray}
Let $w_*\in\reals^m$ denote an optimal solution of the maximization problem in
\eqref{ch2_eq:dual_problem}.  Next, consider the primal-dual pair of
problems corresponding to the penalty formulation for the basis pursuit problem:
\begin{equation}
  \label{ch2_eq:penalty-primal-dual}
  \begin{array}[t]{rl}
      \min_{x\in\reals} & \lambda \norm{x}_1 + \frac{1}{2}\norm{Ax-b-\lambda\theta}^2_2\\
    \end{array}
    \quad \quad
    \begin{array}[t]{rl}
      \max_{w\in\reals^m} & \lambda (b+\lambda\theta)^Tw -\frac{\lambda^2}{2}\norm{w}_2^2 \\
      \mbox{subject to} & \norm{A^Tw}_{\infty} \leq 1.
    \end{array}
  \end{equation}
  Since $w_*$ is feasible for the maximization problem in
  \eqref{ch2_eq:penalty-primal-dual} is as well, it follows that
  \begin{eqnarray}
    P^{(k)}(\xkopt) & = & \min_{x\in\reals^n}
    \big\{\lambda^{(k)}\norm{x}_1 +
    \frac{1}{2}\norm{Ax-b-\lambda^{(k)}\theta^{(k)}}_2^2\big\} \nonumber\\
    & \geq & \lambda^{(k)} \left(b^Tw_*+\lk (\thetak)^T w_*-
      \frac{\lambda^{(k)}}{2}\norm{w_*}_2^2\right),\label{ch2_eq:dualopt}\\
    & \geq & \lambda^{(k)}\left(\norm{x_*}_1-\lk \norm{\thetak}_2
      \norm{w_*}_2
      -\frac{\lambda^{(k)}}{2}\norm{w_*}_2^2\right),\label{ch2_eq:dual_cauchy}
  \end{eqnarray}
  where \eqref{ch2_eq:dualopt} follows from weak duality for primal-dual
  problems~\eqref{ch2_eq:penalty-primal-dual} and
  \eqref{ch2_eq:dual_cauchy}
  follows from strong duality for primal-dual
  problems~\eqref{ch2_eq:dual_problem}, i.e. $b^Tw_*=\norm{x_*}_1$, and
  Cauchy-Schwartz inequality.
  Thus, (\ref{ch2_eq:lbd-1}) implies that
  \eq
    \norm{x^{(k)}}_1
    \geq \norm{x_*}_1-\lk\left(\norm{\thetak}_2 \norm{w_*}_2 +
      \frac{1}{2}\norm{w_*}_2^2+\frac{1}{2}\norm{\theta^{(k+1)}}_2^2\right)
    \geq \norm{x_*}_1-\frac{\left(\frac{\sqrt{n}}{\sigma_{\min}(A)}+
        B_\theta\right)^2}{2}\lambda^{(k)} \label{ch2_eq:suboptimality_lower_bound},
  \en
  where the second inequality follows from the fact that $\norm{A^Tw_*}_\infty\leq 1$.

  The final step in the proof is to establish the upper bound in
  \eqref{ch2_eq:optimality_bound}.
  Suppose the iterate $\xk$ satisfies the stopping condition
  \eqref{ch2_eq:inner-stopping-condition}(a).
  In~\eqref{ch2_eq:inexact_minimizer_bound} we show that
  \begin{equation}
    \norm{x^{(k)}}_1 \leq \norm{x_*}_1+
    \frac{\lambda^{(k)}}{2}\norm{\theta^{(k)}}_2^2+\frac{\epsilon^{(k)}}{\lambda^{(k)}}
    \leq \norm{\xbp}_1 + \lk \left( \frac{1}{2}B_{\theta}^2 + B_1\right),
    \label{ch2_eq:opt_upperbound_a}
  \end{equation}
  where we have used the fact that $\norm{\theta^{(k)}}_2 \leq
  B_{\theta}$ and $\frac{\epsilon^{(k)}}{(\lk)^2} \leq B_1$.

  Next, suppose $\xk$ satisfies
  \eqref{ch2_eq:inner-stopping-condition}(b). Let $g^{(k)} \in \partial
  P^{(k)}(\xk)$ such that $\norm{g^{(k)}}_2 \leq \tk$. Then the convexity of
  $P^{(k)}$ implies  that
  \begin{align}
    P^{(k)}(\xk) - P^{(k)}(\xbp)  \leq - (g^{(k)})^T(\xbp-\xk)\leq
    \norm{g^{(k)}}_2 \norm{\xbp-\xk}_2\leq \tk \norm{\xbp-\xk}_2.
  \end{align}
  Now, an argument similar to the one that establishes the
  bound~\eqref{ch2_eq:inexact_minimizer_bound}, implies that
  \begin{align}
    \norm{\xk}_1\leq \norm{\xbp}_1 + \frac{\lk}{2}\norm{\thetak}_2^2 +
    \frac{\tk}{\lk} \norm{\xbp-\xk}_2
    \leq
    \norm{\xbp}_1 + \lk \left( \frac{1}{2} B_{\theta}^2 + 2cB_xB_1\right),
    \label{ch2_eq:opt_upperbound_b}
  \end{align}
  where we have used the fact that $\norm{\theta^{(k)}}_2 \leq
  B_{\theta}$, $\tk \leq c \epsilon^{(k)}$,  and
  $\frac{\epsilon^{(k)}}{(\lk)^2} \leq B_1$.
  The upper bound in \eqref{ch2_eq:optimality_bound} follows from
  \eqref{ch2_eq:opt_upperbound_a} and
  \eqref{ch2_eq:opt_upperbound_b}. % , we have
  % \begin{align}
  %   \norm{\xk}_1\leq\norm{x_*}_1 +
  %   \left(\frac{B_\theta^2}{2}+B_1~\max\{1,~c
  %     \norm{\xbp-\xk}_2\}\right)\lambda^{(k)}, \label{ch2_eq:L1_norm_bound}
  % \end{align}
  % where \eqref{ch2_eq:L1_norm_bound} follows from the bound on
  % $\norm{\theta^{(k)}}_2$ established in \eqref{ch2_eq:theta_bound} and
  % from the fact that for all $k\geq 1$, $\tk=c\epsk$ and
  % $B_x\geq\norm{\xk}_1$.
\end{proof}

Next, we use the bounds in Theorem~\ref{ch2_thm:finite} to compute a
bound on the convergence rate of \alg{FAL}.
\begin{theorem}
  \label{ch2_thm:epsilon_convergence}
  Fix an $0<\alpha<1$. Then there exists and one can construct a sequence of parameters
  $\{(\lk,\epsk,$ $\tk)\}_{k\in\integers_+}$ such that the iterates generated
  by \alg{FAL}, displayed in Figure~\ref{ch2_alg:fal}, are
  $\epsilon$-feasible, i.e. $\norm{A\xk - b}_2 \leq \epsilon$, and
  $\epsilon$-optimal, $\left|~\norm{\xk}_1 - \norm{\xbp}_1\right|
  \leq \epsilon$, for all $k\geq N_{\rm
    FAL}(\epsilon)=\cO\left(\log_{\frac{1}{\alpha}}\left(\frac{1}{\epsilon}\right)\right)$.
  Moreover, FAL requires % the computational complexity of computing  $N_{\rm
    % FAL}(\epsilon)$  iterations is at most
  % $N_{\rm inner}$
  % iterations of \textbf{Algorithm~APG}, displayed in
  % Figure~\ref{ch1_alg:pga}, such that
  \begin{equation}
  \label{ch2_eq:order_relation_2}
  N_{\rm mat}\leq 2n\kappa(A)^2 \left( \frac{16
        \norm{\xbp}_1}{\alpha(1-\alpha)} \cdot\frac{1}{\epsilon}+
      \frac{9}{\alpha}\cdot
      \log_{\frac{1}{\alpha}}\left(\frac{8n\kappa^2(A)}{\epsilon}\right)
      \right) =\cO\left(\frac{1}{\epsilon}\right),
  \end{equation}
  matrix-vector multiplies to compute an $\epsilon$-feasible,
  $\epsilon$-optimal iterate.
  % where $c_A$ denotes the complexity of computing $A^T(Ax-b)$ for an
  % arbitrary $x\in\reals^n$.
\end{theorem}
\begin{proof}
Rescale the problem parameters $(\bar{A},\bar{b}) =
\frac{1}{\sigma_{\max}(A)} (A, b)$. Then for the rescaled problem
$L=\sigma_{\max}(\bar{A})=1$, but the condition number
$\kappa(\bar{A}) = \kappa(A)$.  We will use \alg{FAL} to solve the
rescaled problem~$(\bar{A},\bar{b})$.

Set $\lambda^{(1)}=1$, $\epsilon^{(1)}=2$, and update
\begin{equation}
  \label{ch2_eq:param-update}
  \begin{array}{rcl}
    \lambda^{(k+1)} &= & \alpha  \cdot  \lambda^{(k)},\\
    \epsilon^{(k+1)} &= & \alpha^2 \cdot  \epsilon^{(k)},\\
    \tau^{(k)} &= & \frac{1}{2\max\{1, \eta +
      \frac{9n}{2}\kappa(A)^2\}} \cdot \epsilon^{(k)}.
  \end{array}
\end{equation}
For this choice of problem parameters, the constants
\[
B_1  =  \max_{k \geq 1} \left\{\frac{\epsilon^{(k)}}{(\lk)^2} \right\} =
  \frac{\epsilon^{(1)}}{(\lambda^{(1)})^2} = 2, \quad
B_2 =  \max_{k \geq 1}
  \left\{\frac{\tk}{\epsilon^{(k)}}\right\} = \frac{\alpha^k}{ 2\max\{1, \eta +
    \frac{9n}{2}\kappa(A)^2\}} \leq 1.
\]
Therefore, the uniform bounds on $\norm{\theta^{(k)}}_2$ and
$\norm{\xk}_1$ are given by
\[
B_{\theta} =  \frac{\sqrt{n}}{\sigma_{\min}(\bar{A})}\left(\max\{\sqrt{2B_1}\
    \sigma_{\max}(\bar{A}),\ B_2\}+1\right) \leq
  3\kappa(A)\sqrt{n}, \quad
B_x  =  \eta + \frac{1}{2} B_{\theta}^2
  \leq \eta + \frac{9n}{2}\kappa(A)^2.
\]
% where $B_x\geq\max\{1,~\norm{\xk}_1\}$ (\eqref{ch2_eq:fal_iter_bound} in
% Theorem~\ref{ch2_thm:limit-point} shows that such $B_X>0$ exists).
% For this specific choice of
% $\{(\lambda^{(k)},\epsilon^{(k)},\tk)\}_{k\in\integers_+}$ sequence, we
% have
% $\frac{\epsilon^{(k)}}{(\lambda^{(k)})^2}=\frac{\epsilon^{(1)}}{(\lambda^{(1)})^2}=2$
% for all $k\geq 1$. Hence, the constant $B_1$ in
% Lemma~\ref{ch2_lem:tseng} and Theorem~\ref{ch2_thm:finite} can be set
% to
% $2$.
% Since $B_1=2$ and $B_x\geq\max\{1,~\norm{\xk}_1\}$,
% \eqref{ch2_eq:theta_bound} imply that $B_\theta=3\sqrt{n}\kappa(A)$,
% where $\kappa(A):=\frac{\sigma_{\max}(A)}{\sigma_{\min}(A)}$.

For the rescaled problem $(\bar{A},\bar{b})$, the
Theorem~\ref{ch2_thm:finite}
guarantees that for all $k \geq 1$,
\[
\left|\norm{x^{(k)}}_1 - \norm{x_*}_1\right| \leq
\max
\left\{\left(\frac{B_\theta^2}{2}+B_1\right),\frac{
    \left(\sqrt{n}\kappa(A)+B_\theta\right)^2}{2}\right\}\lambda^{(1)}\;\alpha^{k-1}
\leq 8n\kappa(A)^2 \alpha^{k-1},
\]
where we use the fact that $\kappa(A) \geq 1$.
Thus, $\left|\norm{x^{(k)}}_1-\norm{x_*}_1 \right| \leq \epsilon$,
for all $k\in\integers_+$ such that
\begin{equation}
  \label{ch2_eq:K-suboptimality}
  k \geq % \max
  % \left\{\frac{\log\left(\left(\frac{B_\theta^2}{2}+B_1\right)
  % \frac{\lambda^{(1)}}{\epsilon}\right)}{\log\left(\frac{1}{\alpha}\right)},
  %   \frac{\log\left(\left(\sqrt{n}\kappa(A)+B_\theta\right)^2
  %   \frac{\lambda^{(1)}}{2\epsilon}\right)}
  %   {\log\left(\frac{1}{\alpha}\right)} \right\}+1 .
  \ln_{\frac{1}{\alpha}} \left(\frac{8n\kappa(A)^2}{\epsilon}\right) + 1.
\end{equation}
From Theorem~\ref{ch2_thm:finite} we also have that for all $k \geq 1$,
\eq
\norm{Ax^{(k)}-b}_2 \leq 2B_\theta \lambda^{(1)}\; \alpha^{k-1} \leq
6 \sqrt{n} \kappa(A).
\en
Thus $\norm{Ax^{(k)}-b}_2 \leq \epsilon$ for all
\begin{equation}
  \label{ch2_eq:K-infeasibility}
  k \geq  % \frac{\log\left(\frac{2B_\theta
      %   \lambda^{(1)}}{\epsilon}\right)}{\log\left(\frac{1}{\alpha}\right)}+1.
  \ln_{\frac{1}{\alpha}} \left(\frac{6\sqrt{n}\kappa(A)}{\epsilon}\right) + 1.
\end{equation}
% Define
% \eq
% U := \max\left\{\left(\frac{B_\theta^2}{2}+B_1\right),\;
%   \frac{1}{2}\left(\sqrt{n}\kappa(A)+B_\theta\right)^2,\;
%   2B_\theta \right\}.
% \en
% Since $\kappa(A)\geq 1$, from $B_1=2$ and $B_\theta=3\sqrt{n}\kappa(A)$,
% it follows that $U=8n\kappa^2(A)$.
From\eqref{ch2_eq:K-suboptimality} and \eqref{ch2_eq:K-infeasibility}
it follows that for all $\epsilon > 0$,  $N_{\rm FAL}(\epsilon)$, the number of FAL iterations
required to compute an $\epsilon$-feasible and $\epsilon$-optimal
solution, is at most
\begin{equation}
  \label{ch2_eq:nout-bnd}
  N_{\rm FAL}(\epsilon) % \leq
                        % eft\lceil\frac{\log(\frac{U\lambda^{(1)}}{\epsilon})}{\log(\frac{1}{\alpha})}\right\rceil+1.
  \leq \left\lceil \ln_{\frac{1}{\alpha}}
  \left(\frac{8n\kappa(A)^2}{\epsilon}\right) \right\rceil + 1.
\end{equation}

% Next, we establish a bound on the number of iterations $L^{(k)}$
% required to compute an $\epsk$-optimal solution of the $k$-th
% subproblem $P^{(k)}(x)$.

% Recall that
% \eq
% L^{(k)} =  \left\lfloor \sigma_{\max}(A)\left(\norm{\xbp}_1 +
% \frac{\lambda^{(k)}}{2}\norm{\theta^{(k)}}_2^2
%          + \norm{x^{k-1}}_1\right)\sqrt{\frac{2}{\epsilon^{(k)}}} \right\rfloor.
% \en
% From \eqref{ch2_eq:theta_bound} and \eqref{ch2_eq:inexact_minimizer_bound},
% it follows that
% \eq
% \norm{x^{(k-1)}}_1 \leq  \norm{\xbp}_1 +
% \frac{\lambda^{(k-1)}}{2}\norm{\theta^{(k-1)}}_2^2
% +\frac{\epsilon^{(k-1)}}{\lambda^{(k-1)}}
% \leq  \norm{\xbp}_1 +
% \frac{\lambda^{(k-1)}}{2}B^2_\theta+\frac{\epsilon^{(k-1)}}{\lambda^{(k-1)}}.
% \en
% Therefore, using the fact that
% $\left\{\lambda^{(k)}\right\}_{k\in\integers_+}$ is a decreasing
% sequence,
% \eq
% L^{(k)} \leq   \left\lfloor \sigma_{\max}(A)\left(2\norm{\xbp}_1 +
% \lambda^{(k-1)}B_{\theta}^2
%          + \frac{\epsilon^{(k-1)}}{\lambda^{(k-1)}}
%          \right)\sqrt{\frac{2}{\epsilon^{(k)}}} \right\rfloor,
% \en
% and from Lemma~\ref{ch2_lem:tseng} it follows that $N_{\rm FAL}(\epsilon)$ FAL iterations
% require at most
From the stopping condition \textsc{APGstop} defined in
Line~\ref{fal:stop} of \alg{FAL}, it follows that the total number
 of the \alg{APG} iterations, $N_{\text{APG}}$, required during
$N_{\text{FAL}}(\epsilon)$ many FAL iterations is bounded by
\begin{eqnarray}
  \bar{N}_{\text{APG}} &   =  &\sum_{k=1}^{N_{\rm
      FAL}(\epsilon)}\left\lceil\ell_{\max}^{(k)} \right\rceil\nonumber \\
  & = &
  \sum_{k=1}^{N_{\rm
      FAL}(\epsilon)}
  \left\lceil \sigma_{\max}(\bar{A})\left(\eta^{(k)} +
      \norm{x^{(k-1)}}_2\right) \sqrt{\frac{2}{\epsilon^{(k)}}}
  \right\rceil \nonumber \\
  & \leq  &  \sum_{k=1}^{N_{\rm
      FAL}(\epsilon)}  \big(2\eta + \lambda^{(k-1)} B_{\theta}^2\big)
  \sqrt{\frac{2}{\epsilon^{(k)}}} \nonumber \\
  & = & 2\eta \sum_{k=0}^{N_{\rm
      FAL}(\epsilon)-1}  \alpha^{-k} +
  \frac{B_{\theta}^2}{\alpha} \cdot N_{\text{FAL}}(\epsilon) \\
  & \leq & \frac{2\alpha\eta}{1-\alpha} \cdot
  \alpha^{-N_{\text{FAL}}(\epsilon)} + \frac{B_{\theta}^2}{\alpha}
  \cdot N_{\text{FAL}}(\epsilon), \label{eq:num_apg}
\end{eqnarray}
where the first inequality follows from the fact that
$\norm{x^{(k-1)}}_2\leq\norm{x^{(k-1)}}_1 \leq \eta^{(k-1)}$ and $\norm{\theta^{(k)}}_2 \leq
B_{\theta}$, the third equality follows from substituting for the
parameters $\lambda^{(k-1)}$ and $\epsilon^{(k)}$, and the last
inequality follows from the summing the geometric series.

\alg{FAL} calls \alg{APG} with a quadratic prox function of the
form $h(x) = \frac{1}{2}\norm{x-\bar{x}}^2_2$ and the smooth function
of the form
$f(x) = \frac{1}{2} \norm{Ax-b-\lambda\theta}_2^2$.  In each
iteration of \alg{APG}, we need to compute the
gradient $\grad f (x) = A^T(Ax-b-\lambda\theta)$ and solve two
constrained shrinkage problems of the form
\eq
\min_{x\in\reals^n}\Big\{\lambda\norm{x}_1+\frac{1}{2}\norm{x-\tilde{x}}_2^2:\;\norm{x}_1\leq\eta\Big\}
\en
We show in Lemma~\ref{ch2_lem:constrained_shrinkage} in Appendix~A
that the complexity of solving a constrained shrinkage problem is
$\cO(n\log(n))$. Thus, the computational complexity of each \alg{APG}
iteration is dominated by the complexity of computing
$A^T(Ax-b-\lambda\theta)$. The complexity result now follows from the
bound in \eqref{eq:num_apg}.

\end{proof}
\section{Extension of FAL to noisy recovery}
\label{ch1_sec:extensions}
In this section, we briefly discuss how FAL can be  extended to solve
the noisy signal recovery problem of the form
\eqref{ch2_eq:relaxed_l1}. See~\cite{Ser10_1J} for the further
extensions of the methodology proposed here.
Consider a noisy recovery problem
\eq
\begin{array}{rl}
    \mbox{min} &  \|x\|_1,\\
    \mbox{s.t.} & \norm{Ax-b}_{\gamma}\leq\delta,
  \end{array}
\en
where $\gamma \in \{1,2,\infty\}$. The formulation with $\gamma \in
\{1,\infty\}$ are interesting when the measurement noise has a
Laplacian or Extreme Value distribution.
By introducing a slack variable
$s\in\reals^m$,  the noisy recovery problem can be formulated as follows:
% \eqref{ch2_eq:relaxed_l1_modified} is an equivalent problem, obtained
\begin{equation}
\label{ch2_eq:relaxed_l1_modified}
\begin{array}{rl}
    \mbox{min} &  \|x\|_1,\\
    \mbox{s.t.} & Ax+s=b, \quad \norm{s}_{\gamma} \leq \delta,
  \end{array}
\end{equation}
We solve \eqref{ch2_eq:relaxed_l1_modified} by inexactly minimizing a
sequence of sub-problems of the form
\begin{equation}
  \label{ch2_eq:augmented_lagrangian_subproblem_noisy}
  P^{(k)}(x,s) = \lk\norm{x}_1 +
  \frac{1}{2}\norm{Ax+s-b-\lambda^{(k)}\theta^{(k)}}_2^2
\end{equation}
over sets $F^{(k)} = \{(x,s): \norm{x}_1\leq\etak, \norm{s}_{\gamma}
\leq \delta\}$ using \alg{APG} with the
prox function
$h^{(k)}(x,s):=\frac{1}{2} \norm{x-x^{(k-1)}}_2^2 +
\frac{1}{2}\norm{s-s^{(k-1)}}_2^2$ and initial iterate $(x^{(k-1)},s^{(k-1)})$, where $\etak:=\eta +
\frac{\lambda^{(k)}}{2}\norm{\theta^{(k)}}_2^2$.

In order to
efficiently solve (\ref{ch2_eq:relaxed_l1_modified}) we need a good
stopping condition for terminating \alg{APG}.  Since $\max\{h^{(k)}(x,s):
(x,s) \in F^{(k)}\} \leq (\mu^{(k)})^2 := \frac{1}{2}\left(\etak + \norm{x^{(k-1)}}_2\right)^2 +
\frac{1}{2}\left(\upsilon(\gamma)\delta + \norm{s^{(k-1)}}_2\right)^2$, where
$\upsilon(\gamma) = 1$, when $\gamma = 1, 2$ and $\sqrt{m}$ when $\gamma =
\infty$, Lemma~\ref{ch1_lem:tseng_corollary} implies that terminating
\alg{APG} at iteration $\left\lceil\ell_{\max}^{(k)}\right\rceil$, where $\ell_{\max}^{(k)} := \sigma_{\max}(A)\mu^{(k)}\sqrt{\frac{2}{\epsilon^{(k)}}}$,
guarantees that $(\xk,s^{(k)})$ is $\epsilon^{(k)}$-optimal for $P^{(k)}$.

Recall that in solving the basis pursuit problem using \alg{FAL} we
terminate \alg{APG} when either we are guaranteed that the iterate
$\xk$ is $\epsilon^{(k)}$-optimal for $P^{(k)}$ or there exists a
sub-gradient $g^{(k)} \in \partial P^{(k)}(\xk)$ with a sufficiently
small norm.
In our numerical experiments, we found that we always terminated the call to \alg{APG} using the sub-gradient
stopping condition. In order to extend FAL to efficiently solve
(\ref{ch2_eq:relaxed_l1_modified})  we need an analog of the sub-gradient
condition.

In FAL we  were able to set the tolerance $\tk$ small since we are
guaranteed that $\argmin_{x \in \reals^n}\{P^{(k)}(x)\} \subseteq
F^{(k)}$. Let $\partial_x P^{(k)}(x,s)$ denote the projection of the set
of sub-gradients on the $x$-variables. The definition of $F^{(k)}$
guarantees that $\xkopt \in F^{(k)}$ for
all $(\xkopt,\skopt ) \in \argmin\{P^{(k)}(x,s): x\in \reals^n,
\norm{s}_{\gamma} \leq \delta\} $. Therefore, we can continue to use the
gradient condition $g_{x} \in \partial_x P^{(k)}(\xk,s^{(k)})$ with
$\norm{g_{x}}_2 \leq \tau_x$. However, since $s$ is constrained, we
cannot force $\norm{g_s}_2$ to be close to zero; therefore, we need an
alternative gradient condition.

Fix $\xk$. Define $\zeta(s) = P^{(k)}(\xk,s)$ and $Q = \{s: \norm{s}_{\gamma}
\leq \delta\}$. The function $\zeta(s)$ is
differentiable and $\grad \zeta(s)$ is Lipschitz continuous.
Let $\pi_Q(s) := \min_{y \in Q}\{\norm{y - (s-\frac{1}{L}\grad
  \zeta(s))}_2\}$ denote the projection of the gradient step $s -
\frac{1}{L} \grad \zeta(s)$ onto the constraint set $Q$. Then
Theorem~2.2.7 in \cite{Nesterov04} establishes that  $\hat{s} \in
\argmin\{P^{(k)}(\xk,s): \norm{s}_\gamma \leq \delta\}$
if and only if $d(\hat{s},Q) := L\norm{\hat{s} - \pi_Q(\hat{s})} = 0$.
Thus, we set the stopping condition for \alg{APG} as follows:
\begin{eqnarray}
\textsc{APGstop}&:=&\left\{\ell \geq \ell^{(k)}_{\max} := \sigma_{\max}(A)\mu^{(k)}\sqrt{\frac{2}{\epsilon^{(k)}}}\right\} \mbox{ \textbf{or} } \nonumber \\
&&\left\{  \exists~g_x \in \partial_x P^{(k)}(x,s)|_{v_x^{(\ell)},v_s^{(\ell)}} \text{ with } \norm{g_x}_2  \leq   \tau_x^{(k)},
\text{ and } d\left(v_s^{(\ell)},F^{(k)}\right) \leq \tau_s^{(k)}\right\}, \nonumber
\end{eqnarray}
where $v_x^{(\ell)}$ and $v_s^{(\ell)}$ are the components of $v^{(\ell)}$ corresponding to $x$ and $s$ variables in \eqref{ch2_eq:augmented_lagrangian_subproblem_noisy}.

Consequently, it follows that the stopping criterion \textsc{APGstop} ensures that the iterate $(\xk,\sk)$ satisfies
one of the following two conditions
\begin{equation}
\label{ch2_eq:inner-stopping-condition_noisy}
\begin{array}{ll}
  (a) & P^{(k)}(\xk, \sk)  \leq  \min_{x\in\reals^n, \norm{s}_\gamma\leq\delta} P^{(k)}(x,s) + \epsilon^{(k)},\\
  (b) & \exists~ g_x^{(k)} \in \partial_x P^{(k)}(x,s)|_{\xk,\sk} \text{ with }
  \norm{g_x^{(k)}}_2  \leq   \tau_x^{(k)}, \text{ and } d\left(\sk,F^{(k)}\right) \leq \tau_s^{(k)}.
\end{array}
\end{equation}
% \begin{lemma}
% Let $f:\reals^m\rightarrow\reals$ be a convex function such that
% $\grad f$ is Lipschitz continuous on $Q\subset\reals^m$ with Lipschitz
% constant $L$ and $Q$ be a closed convex set. Then,
% $s_*\in\argmin_{s\in\reals^m}\{f(s):~s\in Q\}$ if and only if
% $d(s_*,Q)=0$, where $d(s,Q):=L(s-\pi_Q(s))$ and
% $\pi_Q(s)=\argmin_{q\in\reals^m}\{\norm{q-\left(s-\frac{1}{L}\grad
%     f(s)\right)}_2:~q\in Q\}$.
% \end{lemma}
With this modification, Theorem~\ref{ch2_thm:limit-point}, Corollary~\ref{ch2_cor:unique},
Theorem~\ref{ch2_thm:finite} and Theorem~\ref{ch2_thm:epsilon_convergence}
all remain valid for the relaxed recovery problem~\eqref{ch2_eq:relaxed_l1_modified}.
Thus, FAL efficiently computes a solution for the noisy recovery
problem~\eqref{ch2_eq:relaxed_l1}. The per iteration cost of FAL is the
cost of computing $A^T(Ax-b)$. Thus, the per
iteration complexity of FAL is always $\cO(n^2)$ for any $A$;
whereas the per iteration complexity of NESTA~\cite{Can09_4J} is $\cO(n^3)$ when $A$ is
not orthogonal. Examples of non-orthogonal $A$ include Gaussian measurement
matrices, partial psuedo-polar Fourier tranforms, and  non-orthogonal partial wavelet
transforms.

% Let $\{\lk,\epsk,\tau_1^{(k)},\tau_2^{(k)}\}_{k\in\integers_+}$ be a
% sequence parameters such that $\lk \searrow 0$, $\epsk\searrow 0$,
% $\frac{\epsilon^{(k)}}{(\lambda^{(k)})^2}\leq B_1$ for all $k\geq 1$
% for some $B_1>0$, $\tau_i^{(k)}\searrow 0$ and
% $\frac{\tau_i^{(k)}}{\lk}\rightarrow 0$ as $k\rightarrow 0$ for all
% $i\in\{1,2\}$. In the $k$-th FAL iteration, we run
% \textbf{Algorithm~APG} displayed in Figure~\ref{ch1_alg:pga} on
% \eqref{ch2_eq:augmented_lagrangian_subproblem_noisy} over the set
% $F^{(k)}=\{x\in\reals^n,~s\in\reals^m:~\norm{x}_1\leq\etak,~\norm{s}_2\leq\delta\}$,
% with the initial iterate $\left(x^{(k-1)},s^{(k-1)}\right)$ and the
% prox function

% to compute $\left(\xk,\sk\right)$ such that one of the following two
% he brief discussion given in this section can be extended to solve noisy
% recovery problems $\min\{\norm{x}_1: \norm{Ax-b}_1 \leq \delta\}$ and
% $\min\{\norm{x}_1:\norm{Ax-b}_{\infty} \leq\delta\}$.
\section{Implementation details of \alg{FAL} for numerical experiments}
\label{ch2_sec:implementation}
In this section we describe the details of the implemented version of
\alg{FAL} that we used in our numerical experiments described in the next
section.
\subsection{Initial iterate $x^{(0)}$ and bound $\eta$}
\label{ch2_sec:bounds}
%  on the iterates $\{\xk \}_{k \in \integers_+}$}
% Let $x^{(0)} = \argmin\{\norm{x}_2: Ax = b\}=A^T(AA^T)^{-1}b$. Computing
% $x^{(0)}$ requires a
% projection onto the affine space $\{x\in\Re^n: Ax=b\}$. By
% pre-computing the inverse matrix $(AA^T)^{-1} \in \reals^{m\times m}$
% once, which has a complexity of $\cO(m^3)$, one can reduce the running
% time of this projection to $\cO(m^2 + k_b(m,n))$ where
% $k_b(m,n)$ denotes the complexity of computing $A^Tx$ for $A\in\reals^{m\times n}$. In
% our experimental studies, $A$ is a partial DCT or Gaussian matrix.
%
% When $A$ is a partial DCT matrix, computing $x^{(0)}$ requires
% $\cO(n\log(n))$ operations. Since $Ax^{(0)}=b$ and
% $\xbp\in\argmin\{\norm{x}_1:~ Ax=b\}$, we have
% $\norm{x_*}_1\leq\norm{x^{(0)}}_1$. Hence, we set the bound $\eta$ in
% \textbf{Algorithm~FAL-Implementable} to $\norm{x^{(0)}}_1$. Then, for
% all $k\geq 1$,
% $\etak=\norm{x^{(0)}}_1+\frac{\lk}{2}\norm{\thetak}_2^2$ and
% $\norm{\xkopt}_1\leq\etak$.
%
% In Section~\ref{ch2_sec:theory} and Section~\ref{ch1_sec:extensions} we set
% the bound $\etak = \norm{y} + \frac{\lk}{2}\norm{\thetak}_2^2$ where
% $y = \argmin\{\norm{x}_2: Ax = b\} = A^T(AA^T)^{-1}b$.
In our numerical experiments, when $A$ was a
partial DCT matrix, we set $x^{(0)}= \argmin\{\norm{x}_2: Ax = b\} =
A^T(AA^T)^{-1}b = A^Tb$, where the last equality follows from the fact
that $A$ has orthogonal rows. The complexity of computing $x^{(0)}$ in this
case is  $\cO(n\log(n))$. Since $\xbp\in\argmin\{\norm{x}_1:~Ax=b\}$ and $Ax^{(0)}=b$, we have $\norm{\xbp}_1 \leq
\norm{x^{(0)}}_1$, we use $\eta = \norm{x^{(0)}}_1$ in FAL.

For general $A$, the computational cost of computing $A(AA^T)^{-1}b$ is
$\cO(n^2+m^{3})$. In order to avoid $\cO(m^3)$ inversion cost, we set $x^{(0)}=A^Tb$ when $A$ was a
standard Gaussian matrix, i.e. each element $A_{ij}$ is
an independent sample from the  standard $\cN(0,1)$. Since
$\norm{x^{(0)}}_1=\norm{A^Tb}_1$ is no longer an upper bound on
$\norm{\xbp}_1$, we set $\eta$ as follows. It well known that for an
$m\times n$ standard
Gaussian matrix
$\sigma_{\min}(A) \approx
\left(1-\sqrt{\frac{m}{n}}\right)\sqrt{n}$ for large $n$ (in fact, the
approximation is very accurate even at $n = 100$)~\cite{Ede88_1J}. Then,
\begin{equation}
\label{ch2_eq:gaussian_xbound}
\norm{\xbp}_1\leq \norm{A^T(AA^T)^{-1}b}_1 \leq
\frac{\sqrt{n}}{\sigma_{\min}(A)}\norm{b}_2   \leq \eta := \frac{1}{1 -
  \sqrt{\frac{m}{n}}} \norm{b}_2.
\end{equation}
\subsection{\textsc{APGstop} and \textsc{FALstop} conditions}
\label{sec:stop}
When the \alg{APG} terminates with $\ell \geq \ell_{\max}^{(k)}$ we return
$u^{(l)}$. Since gradient computation is computationally the most
expensive step in \alg{APG}, and we are required to compute the gradient
at the $v$ iterates, we checked the sub-gradient stopping condition at the
these iterates.  Hence, we stopped \alg{APG} and returned $v^{(\ell)}$ when
$\min\{\norm{g}_2^2: g \in \partial P^{(k)}(x)|_{v^{(\ell)}} \} \leq
\tau^{(k)}$.

$g \in \partial P^{(k)}(x)|_{v^{(\ell)}}$ if, and only if, there
exists $q
\in \partial \norm{x}_1\mid_{v^{(\ell)}}$ such that $g = \lk q +
\grad f^{(k,l)}$, where $\grad f^{(k,l)} := \grad
f^{(k)}(v^{(\ell)})$. Thus, it follows that
\begin{eqnarray*}
  \min\{\norm{g}_2^2: g
  \in \partial P^{(k)}(x)|_{v^{(\ell)}}\} & = & \min\{\norm{\lk q +
     \grad f^{(k,l)}}_2^2: q \in \partial
  \norm{x}_1\mid_{v^{(\ell)}}\}\\
  & = & \sum_{\{i: v^{(\ell)}_i >0\}} |\lk + \grad_i
  \grad f^{(k,l)}_i|^2 + \sum_{\{i: v^{(\ell)}_i <0\}} |-\lk + \grad f^{(k,l)}_i|^2 \\
  && \mbox{} + \sum_{\{i: v^{(\ell)}_i=0\}} \min_{q_i
    \in [-1,1]} |\lk q_i + \grad f^{(k,l)}_i|^2,\\
  & = & \sum_{\{i: v^{(\ell)}_i >0\}} |\lk + \grad f^{(k,l)}_i|^2 +
  \sum_{\{i: v^{(\ell)}_i <0\}} |-\lk + \grad f^{(k,l)}_i|^2 \\
  && \mbox{} + \sum_{\{i: v^{(\ell)}_i=0\}} \min\{\big||
    \grad f^{(k,l)}_i| - \lk\big|,0\},
\end{eqnarray*}
where $\grad f^{(k,l)}_i$ denote the $i$-th component of the gradient
$\grad f^{(k,l)}$. The first equality follows from the fact  $q_i = +1$
(resp. $-1$) whenever
$v^{(\ell)}_i >0$ (resp. $v^{(\ell)}_i <0$), and $q_i \in [-1,1]$ for
$i$ such that $v^{(\ell)}_i = 0$, and the last equality follows from
explicitly computing the minimum. Thus, given $\grad f^{(k,l)}$, the
complexity of computing $g$ is $\cO(n)$.

We used different stopping conditions depending on the existence of measurement noise. First, we modified the stopping condition \textsc{APGstop} as follows
  \eq
  \textsc{APGstop} = \textsc{FALstop}
  \textbf{ or } \left\{\ell \geq \ell^{(k)}_{\max}\right\} \textbf{ or }
  \left\{\exists g \in\partial P^{(k)}(x)|_{v^{(\ell)}} \text{ with }
    \norm{g}_2 \leq \tau^{(k)} \right\},
  \en
  Then, we set the FAL stopping condition \textsc{FALstop} as follows.
\begin{enumerate}[(i)]
\item\label{stop:noiseless} Noiseless measurements:
  \begin{equation}
  \textsc{FALstop} = \{\norm{u^{(l)}-u^{(l-1)}}_\infty \leq \gamma\},
% \textsc{FALstop} =
%  \left\{
%    \begin{array}{ll}
%      1, & \text{$\norm{u^{(l)}-u^{(l-1)}} \leq \gamma$ in \textsc{APGstop} is true},\\
%      0, & \text{otherwise}.
%    \end{array}
%  \right.
  \end{equation}
  where we set the threshold $\gamma$ by experimenting with a small instance of the problem. \alg{FAL} produces $x_{sol}=u^{(\ell)}$ when \textsc{FALstop} is \textbf{true}.
\item\label{stop:noisy} Noisy measurements:
  \begin{equation}
  \textsc{FALstop}=\left\{{\norm{u^{(l)}-u^{(l-1)}}_2 \over \norm{u^{(l-1)}}_2} \leq \gamma\right\},
  \end{equation}
  where $\gamma$ was set equal to the standard deviation of the noise, i.e. when $b = A\xbp + \zeta$ such that $\zeta$ is a vector of
i.i.d. random variables with standard deviation $\varrho$, we set $\gamma=\varrho$. As in the noiseless case, we terminated FAL and set $x_{sol}=u^{(\ell)}$ when \textsc{FALstop} is \textbf{true}.
\end{enumerate}
%It is possible to show that the stopping condition \textsc{FALstop}
%guarantees that the FAL $x_{\text{sol}}$ is $\gamma$-optimal.
%
% When there is no measurement noise, we terminate when the
% $\ell_{\infty}$ difference between successive inner iterates is below a
% threshold $\gamma$, i.e. for any $k\geq 1$ we stop when
% \begin{equation}
% \label{ch2_eq:stop_noiseless}
% \norm{x^{(k,\ell)}-x^{(k,\ell-1)}}_{\infty}\leq\gamma,
% \end{equation}
% holds for some $\ell\geq 1$.
% On the other hand,
% when $b$ is corrupted with noise, then we stop the algorithm when the
% relative error of successive inner iterates is below a threshold, i.e. for
% any $k\geq 1$ we stop when
% holds for some $\ell\geq 1$. In the noisy measurement case, where
% $b=A\xbp+\zeta$ such that $\zeta\in\reals^m$ is a vector of i.i.d. random
% variables with standard deviation $\varrho$, we set $\gamma=\varrho$.
\subsection{Parameter sequence selection}
\label{ch2_sec:multiplier_selection}
Recall that we require
\eq
\lambda^{(k)} \searrow 0,\qquad
\epsilon^{(k)}\searrow 0, \quad \frac{\epsilon^{(k)}}{(\lambda^{(k
    )})^2}\leq B_1, \qquad
\tk\searrow 0, \quad \frac{\tk}{\lk}\rightarrow 0,\qquad
\tk \leq c \epsilon^{(k)},
\en
for Theorem~\ref{ch2_thm:limit-point}, Corollary~\ref{ch2_cor:unique},
Theorem~\ref{ch2_thm:finite} and Theorem~\ref{ch2_thm:epsilon_convergence}
to hold. These conditions are satisfied if one updates the parameters as
follows: for $k\geq 1$,
\eq
\lambda^{(k+1)} = c_{\lambda}\ \lk, \qquad \tkp =
\min\{c_{\tau}\ \tk,\hat{c}_{\tau}\ \norm{g_0^{(k+1)}}_2\}, \qquad \epskp = c_{\lambda}^2\
\epsk,
\en
where $g_0^{(k)} = \argmin\{\norm{\lk q + \grad f^{(k)}(x^{(k-1)})}_2^2: q \in \partial
\norm{x}_1\mid_{x^{(k-1)}}\}$ is the minimum norm sub-gradient at the
initial iterate $x^{(k-1)}$ for the $k$-th sub-problem for $k\geq 1$, and $c_{\lambda}$, $c_{\tau}$ and
$\hat{c}_{\lambda}$ are appropriately chosen constants in $(0,1)$. Note
that we still have to set the initial iterates $\lambda^{(1)}$,
$\tau^{(1)}$, and $\epsilon^{(1)}$.

In our preliminary numerical experiments with FAL we found that it was
sufficient to set $\hat{c}_{\tau} = 0.9$, and
the optimal choice for the constants  $c_{\lambda}$ and $c_{\tau}$ was only a function
% of the measurement ratio $m/n$ and
sparsity ratio $\xi = \norm{\xbp}_0/m$,
and was effectively independent of the problem size $n$. Moreover, the
optimal choice for $c_{\tau} = c_{\lambda} - 0.01$. We set
\begin{align}
\label{ch2_eq:xi_function}
c_{\lambda}(\xi)=
\left\{
  \begin{array}{ll}
    0.9,  & \hbox{if $\xi \geq 0.9$;} \\
    0.85, & \hbox{if $0.9 > \xi\geq 0.6$;} \\
    0.8,  & \hbox{if $0.6 > \xi\geq 0.25$;} \\
    0.6,  & \hbox{if $0.25 > \xi\geq 0.1$;} \\
    0.4,  & \hbox{if $\xi< 0.1$,}
  \end{array}
\right.
\end{align}
and approximated the sparsity ratio $\xi$ at the beginning of $k$-th FAL
by $\xik = \norm{x^{(k-1)}}_0/m$. Since $x^{(0)}$ is set arbitrarily, and
is unlikely to be sparse, we use the following parameter update rule for $k \geq 2$,
\begin{equation}
  \label{ch2_eq:parameter}
    % g^{(k)}& = & \min\{\norm{\lk q +
    %   A^T(Ax^{(k-1)}-b-\lk\theta^{(k)})}_2: q \in \partial
    % \norm{x}_1\mid_{x^{(k-1)}}\},\
  \tau^{(k)}  = \min\big\{ c_{\tau}(\xik)\,\tau^{(k-1)},
  \hat{c}_{\tau}\norm{g_0^{(k)}}_2\big\}, \quad   \lambda^{(k)}  =  c_\lambda(\xik)
  \, \lambda^{(k-1)}, \quad \epskp = c_{\lambda}(\xik)^2\ \epsk,
\end{equation}
where $c_{\tau}(\xik) =c_{\lambda}(\xik) - 0.01$, and $\hat{c}_{\tau} = 0.9$. We set $c^{(1)}_{\lambda}$ and $c_{\tau}^{(1)}$ for
each problem class separately.

We set the initial parameter values $\lambda^{(1)}$, $\tau^{(1)}$ and
$\epsilon^{(1)}$ as follows.  Let $x^{(0)} = A^Tb$ denote the initial FAL
iterate.
We set
\eq
\tau^{(1)}  =  \hat{c}_{\tau} \norm{g_0^{(1)}}_2, \qquad \lambda^{(1)} =  0.99
\norm{x^{(0)}}_\infty.
\en
We use the duality gap at $x^{(0)}$ to
set $\epsilon^{(1)}$. Since the bound $\eta^{(1)} = \eta$ is set to ensure
that $\argmin_{x \in \reals} P^{(1)}(x) \in F^{(k)} = \{x: \norm{x}_1 \leq
\eta^{(1)}\}$, we are effectively computing an unconstrained
minimum. Since $P^{(1)}(x) \geq 0$ for all $x \in \reals^n$, it follows
that the duality gap of the initial iterate $x^{(0)}$ is at most
$P^{(1)}(x^{(0)})$. We set $\epsilon^{(1)} = 0.99 P^{(1)}(x^{(0)})$.
Then set
$\epsilon^{(k+1)} = \big(c^{(k)}_{\lambda}\big)^2 \epsilon^{(k)}$ for all
$k\geq 1$.

% In CS problems, $n$ and $m$ values are always known; however, the sparsity
% $s$ may not be known. In our numerical experiments, we used a unified
% updating rule for all values of the ratios $m/n$ and $s/m$.
% At every FAL iteration (i.e. every outer iteration), we estimate $s/m$ ratio.
% When stopping condition \eqref{ch2_eq:inner-stopping-condition} holds,  we
% set the $k$-th FAL iterate $x^{(k)}$ to either $y^{(k,\ell)}$ or
% $x^{(k,\ell)}$ depending on the alternative in
% \eqref{ch2_eq:inner-stopping-condition}. Then we estimate $s/m \approx
% \xi^{(k)}=\norm{x^{(k,\ell)}}_0/m$. % , which approximates $s/m$ ratio x
% In our experiments, we set $c^{(k)}_\tau=c^{(k)}_\lambda-0.01$ and
% $c^{(k)}_\lambda = c(\xi^{(k)})$, where $c:[0,1] \mapsto
% (0,1)$ is an increasing,  piecewise constant function with $\min_{x \in
%   [0,1]} \{c(x)\}>0.01$. The specific
% $c(.)$ function used in all our numerical experiments is given in
% \eqref{ch2_eq:xi_function}.

The step length is \alg{APG} is proportional to $\frac{1}{L}$ where $L$
denotes the Lipschitz constant of $\grad f$. In our numerical tests,
we observed that taking long steps, i.e. steps of size $\frac{t}{L}$, for
$t>1$, improves the speed of convergence in practice. And we chose the step-size $t$ as a function of the
sparsity ratio $\norm{\xbp}_0/m$. In iteration $k$, we approximate the
sparsity ratio by $\xik = \norm{x^{(k-1)}}_0/m$, and set the Lipschitz
constant to be used in \alg{APG} to
\eq
L^{(k)} = \frac{\sigma_{\max}(AA^T)}{t(\xik)},
\en
where the function
\begin{align}
\label{ch2_eq:t_function}
t(\xi)=
\left\{
  \begin{array}{ll}
    1.8,  & \hbox{if $\xi \geq 0.9$;} \\
    1.85, & \hbox{if $0.9 > \xi\geq 0.6$;} \\
    1.9,  & \hbox{if $0.6 > \xi\geq 0.25$;} \\
    2,  & \hbox{if $0.25 > \xi\geq 0.1$;} \\
    3,  & \hbox{if $\xi< 0.1$.}
  \end{array}
\right.
\end{align}
\section{Numerical experiments}
\label{ch2_sec:computations}
We conducted three sets of numerical experiments with FAL.
\begin{enumerate}
\item In the first set of experiments we solve randomly generated basis
  pursuit problems when there is no measurement noise. Our goal in this set  of
  experiments were to benchmark the practical performance of
  FAL and to compare FAL
  with the Nesterov-type algorithms, SPA~\cite{AybatI09:SPA}, and
  NESTA~\cite{Can09_4J},  fixed point continuation algorithms,
  FPC, FPC-BB~\cite{Yin07_1R,Yin08_1J}, and FPC-AS~\cite{Wen09_1R}, the
  alternating direction proximal gradient method YALL1~\cite{Yang09}, and a root finding algorithm SPGL1~\cite{Ber08_1J}.
  We describe the results for this set of experiments in Section~\ref{ch2_sec:expt-setup-nonoise}.
\item In the second set of experiments, we compare the performance of FAL
  with the performances of the same set solvers, on randomly generated
  basis pursuit denoising problems when there is a non-trivial level of noise on the measurement vector $b$. The results for this set of experiments is described in
  Section~\ref{ch2_sec:expt-setup-noise}.
\item In the third set of experiments,
  we compare the performance and robustness of FAL with the same solvers on a set of
  small sized, hard compressed sensing problems,
  CaltechTest\cite{caltech08}. The results for this set of experiments is
  reported in Section~\ref{sec:hard}.
\end{enumerate}
All the numerical experiments were conducted on a desktop with 4 dual-core
AMD Opteron
2218~@2.6 GHz processors, 16GB RAM running MATLAB 7.12 on Fedora
14 operating system.
\subsection{Experiments with no measurement noise}
\subsubsection{Signal generation}
\label{ch2_sec:expt-setup-nonoise}
% We tested FAL on randomly generated target signals.
% The target signal $x_*\in\Re^n$ was chosen to be $s$-sparse, i.e. exactly $s$
% out of $n$ components were nonzero.
We generated the target signal $\xbp$ and the measurement matrix $A$ using
the experimental setup
in~\cite{Can09_4J}. In particular, we set
\begin{equation}
  \label{ch2_eq:target_signal}
  (x_*)_i=\mathbf{1}(i\in\Lambda)\ \Theta^{(1)}_i 10^{5\Theta^{(2)}_i},
\end{equation}
where,
\begin{enumerate}[(i)]
\item $\Lambda$ is constructed by randomly selecting $s$ indices
  from the set $\{1,\ldots,n\}$,
\item  $\Theta^{(1)}_i$, $i \in \Lambda$, are IID Bernoulli random
  variables taking values $\pm 1$ with equal probability,
\item $\Theta^{(2)}_i$, $i \in \Lambda$, are IID  uniform~$[0,1]$ random variables.
\end{enumerate}
We then scale $\Theta^{(2)}$ such that $\min_i\Theta^{(2)}_i=0$ and $\max_i\Theta^{(2)}_i=1$. Therefore, the signal $\xbp$ has a
dynamic range of $100dB$.

We randomly selected $m = \frac{n}{4}$ frequencies from the
set $\{0, \ldots, n\}$ and set the measurement matrix $A \in \Re^{m \times
  n}$ to the partial DCT matrix corresponding to the chosen frequencies.
The measurement vector, $b$, is then set to the DCT evaluated at the chosen
frequencies, i.e. $b = Ax_*$.

% We found that for fixed measurement and sparsity ratios, i.e. $m/n$ and $s/n$,
% and the accuracy tolerance $\gamma$, the total number of inner
% iterations, i.e. the number of updates in \textbf{Algorithm~APG} displayed in Figure~\ref{ch1_alg:pga}, is
% effectively independent of the
% dimension $n$ of the target signal. In our experiments we exploit this
% empirical result by first tuning the constants controlling the parameter
% updates on a small sized problem and subsequently using these
% constants to come up with an increasing piecewise constant function
% $c:[0,1]\rightarrow(0,1)$ described in
% Section~\ref{ch2_sec:multiplier_selection}. The $c(.)$ function used in all
% the numerical experiments is as follows:

% In order to speed up the algorithm in practice, we have used a constant step size of $\frac{t^{(k)}}{L}$ in \textbf{Algorithm~APG} for the $k$-th subproblem. The $t(.)$ function used in all
% the numerical experiments is as follows:

\subsubsection{Algorithm scaling results}
\label{ch2_sec:selftest_results}
% We tested the algorithm for $s$-sparse signals with
% \begin{enumerate}[(i)]
% \item three different sizes: small $n = 64\times 64$, medium $n=256\times
%   256$, and large $n= 512\times 512$,
% \item two sparsity levels: high $s=\lceil n/400\rceil$, and low $s=\lceil n/40\rceil$.
% \end{enumerate}
% In order to assess the convergence properties of the FAL, we replaced the
% condition in the stopping criterion \eqref{ch2_eq:stop_noiseless}:\\ ``For
% any $k\geq 1$, \textbf{stop}
% whenever $\norm{x^{(k,\ell)}-x^{(k,\ell-1)}}_{\infty}\leq\gamma$ holds for
% some $\ell\in\integers_+$" with the following condition
For this set of numerical experiments,
\begin{eqnarray}
  \label{ch2_eq:stopping}
  \textsc{FALstop} = \{\norm{u^{(\ell)}-x_*}_{\infty} \leq \gamma\},
\end{eqnarray}
and \alg{FAL} produces $x_{sol}=u^{(\ell)}$ when \textsc{FALstop} is \textbf{true}, where $\xbp$ is the randomly generated target signal. Since the largest magnitude of the target signal, i.e. $\max_i\abs{(\xbp)_i}$ is $10^5$, the stopping condition \textsc{FALstop}
implies that $x_{sol}$ has $5 + \log_{10}(1/\gamma)$ digits of accuracy. We report results for $\gamma = 1$, $10^{-1}$ and
$10^{-2}$. % The signal
% model in~\eqref{ch2_eq:target_signal} and the stopping criterion implies that
% the algorithm produces a  Note that the stopping criterion
% \eqref{ch2_eq:stopping} is only
% used in the first set of numerical
% experiments to test the convergence properties the algorithm.
For the first iteration of FAL, we set  $c^{(1)}_{\tau} = c^{(1)}_\lambda
= 0.4$.  For $k\geq 2$,
% $c_\lambda^{(k)}$, $c_\tau^{(k)}$ and $t^{(k)}$ values were  set as
% described in
we used the functions $c_{\lambda}(\cdot)$ and $t(\cdot)$  described in
\eqref{ch2_eq:xi_function} and \eqref{ch2_eq:t_function}, respectively.
The parameters $(\lk,\tk,\epsk)$ were set as described in
Section~\ref{ch2_sec:multiplier_selection}.

% For
% $k\geq 2$, we update parameters using~\eqref{ch2_eq:parameter}. % using the
% values are set as described in Section~\ref{ch2_sec:multiplier_selection}
% using the $c(.)$ function given in
%as in \eqref{ch2_eq:parameter}.
\begin{table}[!htb]
  \centering
  {\footnotesize
  \begin{tabular}{|c|c|c|}
    \hline
    Sparsity  & $\gamma$ & Table\\
    \hline\hline
    $s = m/100$  & $1$  & Table~\ref{ch2_tab:sparse_tol1}\\
    $s = m/100$  & $0.1$ & Table~\ref{ch2_tab:sparse_tol01}\\
    $s = m/100$  & $0.01$ & Table~\ref{ch2_tab:sparse_tol001}\\
    \hline
    $s = m/10$ & $1$    & Table~\ref{ch2_tab:nonsparse_tol1}\\
    $s = m/10$ & $0.1$  & Table~\ref{ch2_tab:nonsparse_tol01}\\
    $s = m/10$ & $0.01$ & Table~\ref{ch2_tab:nonsparse_tol001}\\
    \hline
  \end{tabular}
  }
  \caption{Summary of numerical experiments}
  \label{ch2_tab:summary}
  \vspace{-5mm}
\end{table}
The Table~\ref{ch2_tab:summary} summarizes the sparsity conditions and the
parameter settings for this set of experiments. The column marked {\tt
  Table} lists the table where we
display the results  corresponding to the parameter setting of the
particular row, e.g. the results for  $s = m/10$ and $\gamma = 0.1$ are
displayed in Table~\ref{ch2_tab:nonsparse_tol01}.

We generated $10$
random instances for each of the experimental conditions.
In
Tables~\ref{ch2_tab:sparse_tol1}--\ref{ch2_tab:nonsparse_tol001}, the
column labeled {\tt average} lists the average taken over the $10$ random
instances, the columns labeled {\tt max} list the maximum over the $10$ instances.
The rows labeled $\mathbf{N_{\rm FAL}}$ and $\mathbf{N_{\rm APG}}$
list the total number of FAL and APG iterations required, respectively,
to solve the instance for the given tolerance parameter $\gamma$. The row labeled $\mathbf{CPU}$ lists
the running time in seconds and the row labeled $\mathbf{nMat}$ lists the
total number of matrix-vector  multiplies of the form $Ax$ or $A^Ty$
computed during the FAL run. In Section~\ref{ch2_sec:test}, we
report two $\mathbf{nMat}$ numbers for FPC-AS: the first one is the number of multiplications
with $A$ during the fixed point iterations and the second one is the
number of multiplications with a reduced form of $A$ during the subspace
optimization iterations. All other rows are self-explanatory.

% For any $1\leq k\leq \mathbf{N_{\rm FAL}}$ and $1\leq \ell \leq
% \mathbf{\Nk}$, define
% $x^{(\sum^{k-1}_{i=1}\mathbf{N^{(i)}}+\ell)}_{\mathbf{in}}:=x^{(k,\ell)}$.
% Figure~\ref{ch2_fig:log_plot}
% plots the relative error of inner iterates $x^{(j)}_{\mathbf{in}}$,
% relative feasibility and relative optimality at the $j$-th cumulative
% iteration of \textbf{Algorithm~APG} for $1\leq j\leq \mathbf{N_{\rm
%     inner}}$. As we have discussed before, we observe
% $\cO(\log(1/\epsilon))$ complexity in Figure~\ref{ch2_fig:log_plot} as
% opposed to $\cO(1/\epsilon)$ worst case complexity of
% Theorem~\ref{ch2_thm:epsilon_convergence}.
%
The experiment results support the following conclusions. FAL is very
efficient -  it requires only $11$-$20$ iterations to converge to an
high accuracy solution of the basis pursuit problem. For a given sparsity level $s$
and a stopping criterion $\gamma$, $\mathbf{N_{\sc FAL}}$ % and
% $\mathbf{N_{\sc APG}}$ are both
is a
very slowly growing function of the dimension $n$ of the
target signal. The total number of matrix-vector multiplies increases
with the number of non-zero elements in the target signal
$\xbp$ -- increasing $s$ from $m/100$ to
$m/10$ increases the number of matrix-vector multiplies by 30\%.
On problems with high sparsity, FAL always recovers the support of
the target signal. We find that FAL is always able to discover the
support of the target signal when the tolerance $\gamma$ is set sufficiently
low.
%%%%%%%%%%%%%%%%%%%%%%%%%%%%%%%%%%%%%%%%%%%%%%%%%%%%%%%%%%%%%%%%%%%%%%%%%%%%%%%%%
\begin{table}[!htb]
    \centering
    {\scriptsize
    \begin{tabular}{c|c|c|c|c|c|c|}
    \cline{2-7}
    &\multicolumn{2}{|c|}{\textbf{n=512}$\times$\textbf{512}} &\multicolumn{2}{|c|}{\textbf{n=256}$\times$\textbf{256}}&\multicolumn{2}{|c|}{\textbf{n=64}$\times$\textbf{64}}\\ \cline{2-7}
    &\textbf{Average}&\textbf{Max}&\textbf{Average}&\textbf{Max}&\textbf{Average}&\textbf{Max}\\ \hline
    \multicolumn{1}{|c|}{$\mathbf{N_{\rm APG}}$}
    &27.0&27&26.5&27&29.3&34 \\ \hline
    \multicolumn{1}{|c|}{$\mathbf{|\norm{x_{sol}}_1-\norm{x_*}_1|/\norm{x_*}_{1}}$}
    &1.70E-06&2.12E-06&3.96E-06&1.12E-05&1.29E-05&2.62E-05 \\ \hline
    \multicolumn{1}{|c|}{$\mathbf{\max\{|(x_{sol})_i-(x_*)_i|: (x_*)_i\neq 0\}}$}
    &3.23E-01&3.78E-01&5.73E-01&1.00E+00&9.09E-01&1.00E+00\\ \hline
    \multicolumn{1}{|c|}{$\mathbf{\max\{|(x_{sol})_i|: (x_*)_i=0\}}$}
    &0.00E+00&0.00E+00&0.00E+00&0.00E+00&0.00E+00&0.00E+00\\ \hline
    \multicolumn{1}{|c|}{$\mathbf{\|Ax_{sol}-b\|_2}$}
    &0.703&0.781&0.799&1.809&0.604&0.850 \\ \hline
    \multicolumn{1}{|c|}{$\mathbf{\|x_{sol}\|_1}$}
    &5588229.9&7000555.1&1508014.9&1838186.7&193826.1&311446.4 \\ \hline
    \multicolumn{1}{|c|}{$\mathbf{\|x_*\|_1}$}
    &5588239.3&7000565.2&1508021.0&1838194.7&193828.2&311447.1 \\ \hline
    \multicolumn{1}{|c|}{$\mathbf{CPU}$}
    &13.5&13.7&3.1&3.1&0.2&0.3 \\ \hline
    \multicolumn{1}{|c|}{$\mathbf{N_{\rm FAL}}$}
    &14.0&14.0&13.6&14.0&12.6&14.0 \\ \hline
    \multicolumn{1}{|c|}{$\mathbf{nMat}$}
    &56&56&55&56&60.6&70 \\ \hline
    \end{tabular}
    \caption{FAL scaling results: $m=n/4$, $s=m/100$ and $\|x_{sol}-x_*\|_{\infty}\leq 1$}
    \label{ch2_tab:sparse_tol1}
    }
    \vspace{3mm}
    \centering
    {\scriptsize
    \begin{tabular}{c|c|c|c|c|c|c|}
    \cline{2-7}
    &\multicolumn{2}{|c|}{\textbf{n=512}$\times$\textbf{512}} &\multicolumn{2}{|c|}{\textbf{n=256}$\times$\textbf{256}}&\multicolumn{2}{|c|}{\textbf{n=64}$\times$\textbf{64}}\\ \cline{2-7}
    &\textbf{Average}&\textbf{Max}&\textbf{Average}&\textbf{Max}&\textbf{Average}&\textbf{Max}\\ \hline
    \multicolumn{1}{|c|}{$\mathbf{N_{\rm APG}}$} &28.9&29&28.4&29&35.1&45 \\ \hline
    \multicolumn{1}{|c|}{$\mathbf{|\norm{x_{sol}}_1-\norm{x_*}_1|/\norm{x_*}_{1}}$} &6.79E-08&2.69E-07&1.03E-07&2.49E-07&4.96E-07&1.06E-06 \\ \hline
    \multicolumn{1}{|c|}{$\mathbf{\max\{|(x_{sol})_i-(x_*)_i|: (x_*)_i\neq 0\}}$}
    &2.58E-02&9.51E-02&5.57E-02&9.07E-02&6.22E-02&8.76E-02 \\ \hline
    \multicolumn{1}{|c|}{$\mathbf{\max\{|(x_{sol})_i|: (x_*)_i=0\}}$}
    &0.00E+00&0.00E+00&0.00E+00&0.00E+00&0.00E+00&0.00E+00\\ \hline
    \multicolumn{1}{|c|}{$\mathbf{\|Ax_{sol}-b\|_2}$} &0.047&0.187&0.062&0.109&0.037&0.065 \\ \hline
    \multicolumn{1}{|c|}{$\mathbf{\|x_{sol}\|_1}$} &5588239.4&7000565.5&1508020.8&1838194.5&193828.1&311446.9 \\ \hline
    \multicolumn{1}{|c|}{$\mathbf{\|x_*\|_1}$} &5588239.3&7000565.2&1508021.0&1838194.7&193828.2&311447.1 \\ \hline
    \multicolumn{1}{|c|}{$\mathbf{CPU}$} &14.5&14.7&3.3&3.4&0.3&0.4 \\ \hline
    \multicolumn{1}{|c|}{$\mathbf{N_{\rm FAL}}$} &14.9&15.0&14.4&15.0&14.0&14.0 \\ \hline
    \multicolumn{1}{|c|}{$\mathbf{nMat}$} &59.8&60&58.8&60&72.2&92 \\ \hline
    \end{tabular}
    \caption{FAL scaling results: $m=n/4$, $s=m/100$ and
      $\|x_{sol}-x_*\|_{\infty}\leq 10^{-1}$}
    \label{ch2_tab:sparse_tol01}
    }
    \vspace{3mm}
    \centering
    {\scriptsize
    \begin{tabular}{c|c|c|c|c|c|c|}
    \cline{2-7}
    &\multicolumn{2}{|c|}{\textbf{n=512}$\times$\textbf{512}} &\multicolumn{2}{|c|}{\textbf{n=256}$\times$\textbf{256}}&\multicolumn{2}{|c|}{\textbf{n=64}$\times$\textbf{64}}\\ \cline{2-7}
    &\textbf{Average}&\textbf{Max}&\textbf{Average}&\textbf{Max}&\textbf{Average}&\textbf{Max}\\ \hline
    \multicolumn{1}{|c|}{$\mathbf{N_{\rm APG}}$} &29.9&30&29.5&30&37.7&49 \\ \hline
    \multicolumn{1}{|c|}{$\mathbf{|\norm{x_{sol}}_1-\norm{x_*}_1|/\norm{x_*}_{1}}$} &6.36E-08&9.54E-08&4.77E-08&6.85E-08&4.57E-08&1.66E-07\\ \hline
    \multicolumn{1}{|c|}{$\mathbf{\max\{|(x_{sol})_i-(x_*)_i|: (x_*)_i\neq 0\}}$}
    &7.66E-03&9.27E-03&6.92E-03&8.60E-03&5.63E-03&9.63E-03 \\ \hline
    \multicolumn{1}{|c|}{$\mathbf{\max\{|(x_{sol})_i|: (x_*)_i=0\}}$}
    &0.00E+00&0.00E+00&0.00E+00&0.00E+00&0.00E+00&0.00E+00 \\ \hline
    \multicolumn{1}{|c|}{$\mathbf{\|Ax_{sol}-b\|_2}$} &0.027&0.032&0.013&0.015&0.004&0.006 \\ \hline
    \multicolumn{1}{|c|}{$\mathbf{\|x_{sol}\|_1}$} &5588239.7&7000565.6&1508021.0&1838194.8&193828.2&311447.1 \\ \hline
    \multicolumn{1}{|c|}{$\mathbf{\|x_*\|_1}$} &5588239.3&7000565.2&1508021.0&1838194.7&193828.2&311447.1 \\ \hline
    \multicolumn{1}{|c|}{$\mathbf{CPU}$} &15.0&15.2&3.4&3.5&0.3&0.4 \\ \hline
    \multicolumn{1}{|c|}{$\mathbf{N_{\rm FAL}}$} &15.0&15.0&15.0&15.0&14.7&16.0 \\ \hline
    \multicolumn{1}{|c|}{$\mathbf{nMat}$} &61.8&62&61&62&77.4&100 \\ \hline
    \end{tabular}
    \caption{FAL scaling results: $m=n/4$, $s=m/100$ and
      $\|x_{sol}-x_*\|_{\infty}\leq 10^{-2}$}
    \label{ch2_tab:sparse_tol001}
    }
\end{table}
%%%%%%%%%%%%%%%%%%%%%%%%%%%%%%%%%%%%%%%%%%%%%%%%%%%%%%%%%%%%%%%%%%%%%%%%%%%%%%%%%%%%%%%
\begin{table}[!htb]
    \centering
    {\scriptsize
    \begin{tabular}{c|c|c|c|c|c|c|}
    \cline{2-7}
    &\multicolumn{2}{|c|}{\textbf{n=512}$\times$\textbf{512}} &\multicolumn{2}{|c|}{\textbf{n=256}$\times$\textbf{256}}&\multicolumn{2}{|c|}{\textbf{n=64}$\times$\textbf{64}}\\ \cline{2-7}
    &\textbf{Average}&\textbf{Max}&\textbf{Average}&\textbf{Max}&\textbf{Average}&\textbf{Max}\\ \hline
    \multicolumn{1}{|c|}{$\mathbf{N_{\rm APG}}$} &28.9&29&28.2&29&27.2&28 \\ \hline
    \multicolumn{1}{|c|}{$\mathbf{|\norm{x_{sol}}_1-\norm{x_*}_1|/\norm{x_*}_{1}}$} &6.82E-07&2.51E-06&2.01E-06&3.31E-06&4.93E-06&8.35E-06 \\ \hline
    \multicolumn{1}{|c|}{$\mathbf{\max\{|(x_{sol})_i-(x_*)_i|: (x_*)_i\neq 0\}}$}
    &6.46E-01&9.83E-01&8.26E-01&9.88E-01&7.92E-01&1.00E+00 \\ \hline
    \multicolumn{1}{|c|}{$\mathbf{\max\{|(x_{sol})_i|: (x_*)_i=0\}}$}
    &1.31E-01&2.24E-01&1.83E-01&4.07E-01&1.19E-01&2.12E-01\\ \hline
    \multicolumn{1}{|c|}{$\mathbf{\|Ax_{sol}-b\|_2}$} &2.432&4.802&2.020&2.441&0.857&1.203 \\ \hline
    \multicolumn{1}{|c|}{$\mathbf{\|x_{sol}\|_1}$} &56631758.9&59669790.2&14250619.6&15030777.3&1033569.1&1289376.0 \\ \hline
    \multicolumn{1}{|c|}{$\mathbf{\|x_*\|_1}$} &56631797.7&59669841.3&14250648.3&15030813.1&1033574.2&1289377.9 \\ \hline
    \multicolumn{1}{|c|}{$\mathbf{CPU}$} &14.5&14.6&3.3&3.4&0.2&0.6 \\ \hline
    \multicolumn{1}{|c|}{$\mathbf{N_{\rm FAL}}$} &14.9&15.0&14.2&15.0&13.8&14.0 \\ \hline
    \multicolumn{1}{|c|}{$\mathbf{nMat}$} &59.8&60&58.4&60&56.4&58 \\ \hline
    \end{tabular}
    \caption{FAL scaling results: $m=n/4$, $s=m/10$ and $\|x_{sol}-x_*\|_{\infty}\leq 1$}
    \label{ch2_tab:nonsparse_tol1}
    }
    \vspace{3mm}
    {\scriptsize
    \begin{tabular}{c|c|c|c|c|c|c|}
    \cline{2-7}
    &\multicolumn{2}{|c|}{\textbf{n=512}$\times$\textbf{512}} &\multicolumn{2}{|c|}{\textbf{n=256}$\times$\textbf{256}}&\multicolumn{2}{|c|}{\textbf{n=64}$\times$\textbf{64}}\\ \cline{2-7}
    &\textbf{Average}&\textbf{Max}&\textbf{Average}&\textbf{Max}&\textbf{Average}&\textbf{Max}\\ \hline
    \multicolumn{1}{|c|}{$\mathbf{N_{\rm APG}}$} &33.7&35&32.7&34&31.2&33 \\ \hline
    \multicolumn{1}{|c|}{$\mathbf{|\norm{x_{sol}}_1-\norm{x_*}_1|/\norm{x_*}_{1}}$} &5.30E-07&7.38E-07&6.70E-07&1.03E-06&4.56E-07&8.46E-07 \\ \hline
    \multicolumn{1}{|c|}{$\mathbf{\max\{|(x_{sol})_i-(x_*)_i|: (x_*)_i\neq 0\}}$}
    &9.21E-02&9.98E-02&8.96E-02&9.58E-02&7.06E-02&8.92E-02 \\ \hline
    \multicolumn{1}{|c|}{$\mathbf{\max\{|(x_{sol})_i|: (x_*)_i=0\}}$}
    &1.17E-03&6.17E-03&2.38E-03&1.26E-02&7.16E-03&2.40E-02 \\ \hline
    \multicolumn{1}{|c|}{$\mathbf{\|Ax_{sol}-b\|_2}$} &0.601&0.748&0.357&0.469&0.087&0.120 \\ \hline
    \multicolumn{1}{|c|}{$\mathbf{\|x_{sol}\|_1}$} &56631827.7&59669880.2&14250657.8&15030821.1&1033574.7&1289378.1 \\ \hline
    \multicolumn{1}{|c|}{$\mathbf{\|x_*\|_1}$} &56631797.7&59669841.3&14250648.3&15030813.1&1033574.2&1289377.9 \\ \hline
    \multicolumn{1}{|c|}{$\mathbf{CPU}$} &16.8&17.6&3.8&4.0&0.2&0.7 \\ \hline
    \multicolumn{1}{|c|}{$\mathbf{N_{\rm FAL}}$} &17.1&18.0&16.7&17.0&15.7&16.0 \\ \hline
    \multicolumn{1}{|c|}{$\mathbf{nMat}$} &69.4&72&67.4&70&64.4&68 \\ \hline
    \end{tabular}
    \caption{FAL scaling results: $m=n/4$, $s=m/10$ and $\|x_{sol}-x_*\|_{\infty}\leq 10^{-1}$}
    \label{ch2_tab:nonsparse_tol01}
    }
    \vspace{3mm}
    \centering
    {\scriptsize
    \begin{tabular}{c|c|c|c|c|c|c|}
    \cline{2-7}
    &\multicolumn{2}{|c|}{\textbf{n=512}$\times$\textbf{512}} &\multicolumn{2}{|c|}{\textbf{n=256}$\times$\textbf{256}}&\multicolumn{2}{|c|}{\textbf{n=64}$\times$\textbf{64}}\\ \cline{2-7}
    &\textbf{Average}&\textbf{Max}&\textbf{Average}&\textbf{Max}&\textbf{Average}&\textbf{Max}\\ \hline
    \multicolumn{1}{|c|}{$\mathbf{N_{\rm APG}}$} &38.7&39&38.3&39&37.7&39 \\ \hline
    \multicolumn{1}{|c|}{$\mathbf{|\norm{x_{sol}}_1-\norm{x_*}_1|/\norm{x_*}_{1}}$} &4.94E-08&5.80E-08&4.32E-08&5.54E-08&2.16E-08&5.06E-08 \\ \hline
    \multicolumn{1}{|c|}{$\mathbf{\max\{|(x_{sol})_i-(x_*)_i|: (x_*)_i\neq 0\}}$}
    &8.71E-03&9.96E-03&8.51E-03&9.88E-03&7.19E-03&9.41E-03\\ \hline
    \multicolumn{1}{|c|}{$\mathbf{\max\{|(x_{sol})_i|: (x_*)_i=0\}}$}
    &1.58E-04&1.11E-03&0.00E+00&0.00E+00&0.00E+00&0.00E+00\\ \hline
    \multicolumn{1}{|c|}{$\mathbf{\|Ax_{sol}-b\|_2}$} &0.069&0.084&0.038&0.042&0.010&0.011 \\ \hline
    \multicolumn{1}{|c|}{$\mathbf{\|x_{sol}\|_1}$} &56631794.9&59669838.5&14250647.7&15030812.4&1033574.2&1289377.9 \\ \hline
    \multicolumn{1}{|c|}{$\mathbf{\|x_*\|_1}$} &56631797.7&59669841.3&14250648.3&15030813.1&1033574.2&1289377.9 \\ \hline
    \multicolumn{1}{|c|}{$\mathbf{CPU}$} &19.3&19.5&4.4&4.5&0.3&0.7 \\ \hline
    \multicolumn{1}{|c|}{$\mathbf{N_{\rm FAL}}$} &19.7&20.0&19.3&20.0&18.8&19.0 \\ \hline
    \multicolumn{1}{|c|}{$\mathbf{nMat}$} &79.4&80&78.6&80&77.4&80 \\ \hline
    \end{tabular}
    \caption{FAL scaling results: $m=n/4$, $s=m/10$ and $\|x_{sol}-x_*\|_{\infty}\leq 10^{-2}$}
    \label{ch2_tab:nonsparse_tol001}
    }
\end{table}
%%%%%%%%%%%%%%%%%%%%%%%%%%%%%%%%%%%%%%%%%%%%%%%%%%%%%%%%%%%%%%%%%%%%%%%%%%%%%%%%%%%
\vspace{-5mm}

The worst case bound in Theorem~\ref{ch2_thm:epsilon_convergence} suggests that FAL requires
$\cO(\frac{1}{\epsilon})$ \alg{APG} iterations (or equivalently,
matrix-vector multiplies) to compute an $\epsilon$-feasible and $\epsilon$-optimal solution to the basis pursuit
problem. In our numerical experiments we found that we required only $4\pm
1$ APG iterations per FAL iteration; therefore, we
required only $\cO(\log(\frac{1}{\epsilon}))$ APG iterations to
compute an $\epsilon$-optimal solution. In order to clearly demonstrate this phenomenon, we created $5$ random instances of $x_*\in\reals^n$ and
partial DCT matrix $A\in\reals^{m\times n}$ such that $n=64^2$ and
$m=\frac{n}{4}$ as described in
Section~\ref{ch2_sec:expt-setup-nonoise}. Any $x_*$ created contains $\lceil\frac{m}{10}\rceil$
nonzero components such that the largest and smallest magnitude of those
components are $10^5$ and $1$, respectively.

 We solved this set of random instances with \alg{FAL} using $\textsc{FALstop} = \{\norm{u^{(\ell)}-u^{(\ell-1)}}_{\infty} \leq 5\times 10^{-11}\}$. As before, let $\mathbf{N_{\rm FAL}}$ denote the number of FAL iterations required to compute $x_{sol}$ satisfying the stopping condition $\textsc{FALstop}$. Let $\mathbf{\Nk}$ be the number of
 \alg{APG} iterations done on the $k$-th call until the inner
 stopping condition \eqref{ch2_eq:inner-stopping-condition} holds and
 $\mathbf{N_{\rm APG}}=\sum^{\mathbf{N_{\rm FAL}}}_{k=1}\mathbf{\Nk}$ be
 the \emph{total} number of inner iterations, i.e. \emph{total} number of APG iterations, to compute $x_{sol}$.

 For all five instances $\mathbf{N_{\rm FAL}}\approx 45$, $\mathbf{N_{\rm APG}}\approx 95$, $\max_{1\leq i\leq n}\{|(x_{sol})_i-(x_*)_i|:\
 (x_*)_i\neq 0\}\approx7\times10^{-11}$,  $\max_{1\leq i\leq
   n}\{|(x_{sol})_i|:\ (x_*)_i=0\}=0$ and $\norm{Ax_{sol}-b}_2\approx
 1\times 10^{-10}$. These numbers show that each output $x_{sol}$ is $15$
 digits accurate and very close to feasibility.

 Let $u^{(k,\ell)}$ denote $u^{(\ell)}$ iterate on the $k$-th APG call. For any $1\leq k\leq \mathbf{N_{\rm FAL}}$ and $1\leq \ell \leq
 \mathbf{\Nk}$, define
 $x^{(\sum^{k-1}_{i=1}\mathbf{N^{(i)}}+\ell)}_{\mathbf{in}}:=u^{(k,\ell)}$.
 In Figure~\ref{ch2_fig:log_plot}, we plot the relative error,
 relative feasibility and relative optimality of the inner iterates $x^{(j)}_{\mathbf{in}}$ as functions of \alg{APG} cumulative
 iteration counter $j\in\{1,...,\mathbf{N_{\rm APG}}\}$. From the plots in
 Figure~\ref{ch2_fig:log_plot}, it is clear that, in practice, the complexity of computing an $\epsilon$-feasible,
 $\epsilon$-optimal iterate is  $\cO(\log(1/\epsilon))$, as opposed to the $\cO(1/\epsilon)$ worst case complexity bound established
 Theorem~\ref{ch2_thm:epsilon_convergence}.
\begin{figure} [!htb]
\centering
    \includegraphics[scale=0.65]{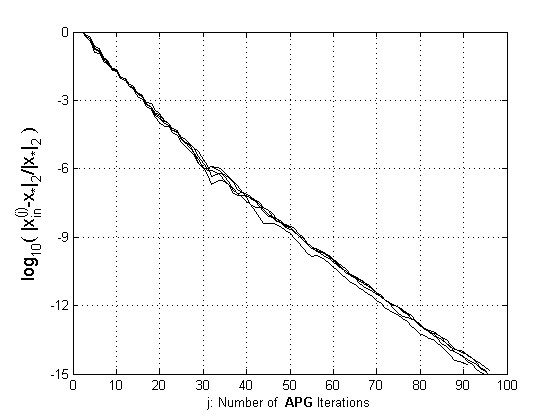}
    \includegraphics[scale=0.65]{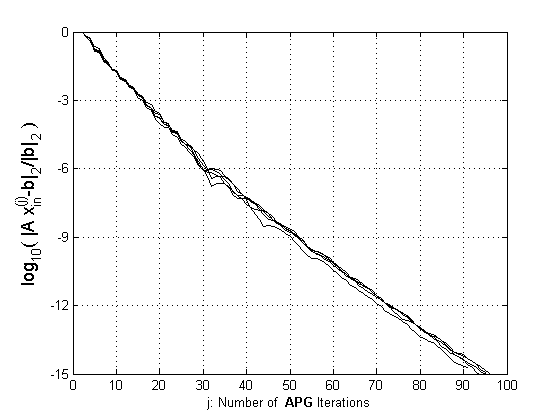}
    \includegraphics[scale=0.65]{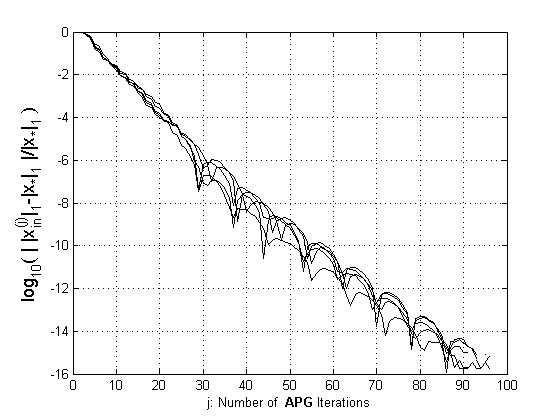}
    \caption{Relative solution error, feasibility and optimality vs APG iterations }
\label{ch2_fig:log_plot}
\end{figure}
\clearpage
\subsubsection{Comparison with other solvers}
\label{ch2_sec:test}
In this section, we report the results of our numerical experiments
comparing FAL with SPA~\cite{AybatI09:SPA}, NESTA
v1.1~\cite{Can09_4J}~[\url{http://www.acm.caltech.edu/~nesta/}],
FPC and FPC-BB from
FPC~v2.0~\cite{Yin07_1R,Yin08_1J}
~[\url{http://www.caam.rice.edu/~optimization/L1/fpc/}], FPC-AS
v1.21~\cite{Wen09_1R}
[\url{http://www.caam.rice.edu/~optimization/L1/FPC_AS/}],
YALL1 v1.4~\cite{Yang09}
[\url{http://www.yall1.blogs.rice.edu}] and SPGL1 v1.7~\cite{Ber08_1J} [\url{http://www.cs.ubc.ca/labs/scl/spgl1/}]. % , FPC
% and FPC-BB, which is a state-of-the-art version of FPC that uses
% Barzilai-Borwein
% steps~\cite{Barzilai88} to accelerate the performance of the algorithm in
% practice, that come with
% v2.0 package~[\url{http://www.caam.rice.edu/~optimization/L1/fpc/}], FPC-AS
% v1.21~[\url{http://www.caam.rice.edu/~optimization/L1/FPC_AS/}]  and
%  YALL1 v1.3~[\url{http://www.www.yall1.blogs.rice.edu}].
We set the parameter values for each of the six solvers so that they
all produce a  solution with  $\ell_{\infty}$-error approximately equal to $5\times 10^{-4}$,
i.e. $\norm{x_{sol}-x_*}_\infty \approx 5\times 10^{-4}$. This criterion
results in the following set of parameters (all other parameters not mentioned below are set to their default values).
\begin{enumerate}[(a)]
\item \textbf{FAL}: We set $\gamma=2.5\times10^{-4}$, and the initial
  update coefficients  $c_\lambda^{(1)}=0.4$,
  $c_\tau^{(1)}=0.4$ and $t^{(1)}=2$. % We set $\gamma=5\times10^{-9}$ and the initial update
  % coefficients $c_\lambda^{(1)}=0.8$, $c_\tau^{(1)}=0.8$ and $t^{(1)}=1.9$.
  For $k\geq 2$,
  % $c_\lambda^{(k)}$, $c_\tau^{(k)}$ and $t^{(k)}$ values were  set as
  % described in
  we used the functions $c_{\lambda}(\cdot)$ and $t(\cdot)$  described in
  \eqref{ch2_eq:xi_function} and \eqref{ch2_eq:t_function}, respectively.
  The parameters $(\lk,\tk,\epsk)$ were set as described in
  Section~\ref{ch2_sec:multiplier_selection}.
\item \textbf{SPA}:  $\gamma=5\times10^{-5}$, $c_\tau^{(0)} = 0.2$,
  $c^{(1)}_\tau = 0.855$, $c_\epsilon = 0.8$, $c_\lambda = 0.9$ and $c_\mu
  = c_\nu= 0.1$. For details on these parameters refer to
  \cite{AybatI09:SPA}.
\item \textbf{NESTA}: $\mu=1\times 10^{-4}$ and $\gamma=1\times
  10^{-10}$. NESTA solves $\min_{\norm{Ax-b}_2\leq \delta} p_{\mu}(x)$,
  where $p_\mu(x)=\max\{x^Tu-\frac{\mu}{2}\norm{u}_2^2:
  \norm{u}_\infty\leq 1\}$. NESTA terminates when
  $\frac{|p_\mu(\xk)-\bar{p}_\mu(\xk)|}{\bar{p}_\mu(\xk)}<\gamma$, for
  some $\gamma>0$, where
  $\bar{p}_\mu(\xk)=\frac{1}{min\{10,k\}}\sum_{\ell=1}^{min\{10,k\}}
  p_{\mu}(x^{(k-\ell)})$.
\item  \textbf{FPC} and \textbf{FPC-BB}: $\frac{1}{\lambda}= 1.5 \times
  10^4$.  FPC and FPC-BB  solve
  $\min_{x\in\Re^n}\norm{x}_1+\frac{1}{\lambda}\norm{Ax-b}_2^2$.
\item \textbf{FPC-AS}: $\lambda=7.5\times 10^{-5}$. FPC-AS  solves
  $\min_{x\in\Re^n}\lambda\norm{x}_1+\frac{1}{2}\norm{Ax-b}_2^2$.
\item \textbf{YALL1~(BP)}: $\gamma=2\times 10^{-9}$ and $nonorth=0$. YALL1~(BP) algorithm solves the basis pursuit problem
  $\min_{x\in\Re^n}\{\norm{x}_1:\ Ax=b\}$ and terminates when $\frac{\norm{x_{k+1}-x_k}_2}{\norm{x_{k+1}}_2}\leq\gamma$. $nonorth=0$ indicates that $AA^T=I$.
\item \textbf{SPGL1~(BP)}: $optTol=5\times 10^{-3}$ and $bpTol=1\times 10^{-6}$. SPGL1~(BP) algorithm solves the basis pursuit problem
  $\min_{x\in\Re^n}\{\norm{x}_1:\ Ax=b\}$. For the optimality and basis pursuit tolerance parameters, $optTol$ and $bpTol$, refer to \cite{Ber08_1J}.
\end{enumerate}
% For each   we set the parameters  by solving a set of small size problems
% and these parameter values were fixed throughout the experiments, all
% other parameters are set to their
% default values.
The termination criteria for the different solvers were not directly
comparable since the different solvers solve slightly different
formulations of the basis pursuit problem. However, we attempted to set
the stopping parameter $\gamma$ for FAL so that on average the stopping
criterion for FAL was more stringent than any of the other solvers.

We tested each solver on the same set of 10 random instances of size
$n=512\times512$ that were generated using the procedure
described in Section~\ref{ch2_sec:expt-setup-nonoise}.
The results of the experiments are displayed in
Table~\ref{ch2_tab:comparative_test_results}.
% In Table~\ref{ch2_tab:comparative_test_results},
% the row labeled $\mathbf{CPU}$ lists the running time of each algorithm in
% seconds.
%the row labeled $\mathbf{N_{outer}}$ lists the number of FAL, SPA and Bregman iterations
% and all the other rows labels are as defined in
% Section~\ref{ch2_sec:selftest_results}.
The experimental results in
Table~\ref{ch2_tab:comparative_test_results}, show that % in order to
% produce a solution of smaller
% $\ell_\infty$-error,
FAL was {\em six} times faster than SPA and NESTA, approximately
{\em four} times faster than FPC, and {\em two} times faster than FPC-BB
and FPC-AS algorithms. Moreover, unlike the other solvers, for all $10$
instances, FAL accurately identified
the support of the target signal, % $x_*$, i.e. $I_0=\{i\in\{1,2,...,n\}:
% (x_*)_i=0\}$
without any heuristic thresholding step. This feature of FAL is very
appealing in practice. For signals
with a large dynamic range, almost all of the state-of-the-art efficient
algorithms
% with fast performance will
produce a solution with many small non zeros terms, and it is often
% and it can be
hard to determine this threshold. % The minimum, average and maximum of FAL
% iterations, $\mathbf{N_{FAL}}$, are $24$, $24.5$ and $25$,
% respectively. Also we have noticed that in each FAL iteration,  FAL does
% around $4\pm 1$ inner iterations. This empirical fact is the reason why we
% see $\cO(\log(1/\epsilon))$ iteration complexity in practice as reported
% in Section~\ref{ch2_sec:modified-inner}, instead of $\cO(1/\epsilon)$
% worst case theocratical bound proved in
% Theorem~\ref{ch2_thm:epsilon_convergence}.
\subsection{Experiments with measurement noise}
\subsubsection{Signal generation}
\label{ch2_sec:expt-setup-noise}
For this set of experiments the target signal $\xbp \in \Re^n$ was
generated as follows: $(x_*)_i=\mathbf{1}(i\in\Lambda)\ \Theta_i$,
%\begin{equation}
%  \label{ch2_eq:target_signal_noise}
%  (x_*)_i=\mathbf{1}(i\in\Lambda)\ \Theta_i,
%\end{equation}
%\vspace{-3mm}
where
\begin{enumerate}[(i)]
\item the set $\Lambda$ was constructed by randomly selecting $s$ indices
  from the set $\{1,\ldots,n\}$,
\item  $\Theta_i$, $i \in \Lambda$, were independently, and
  identically distributed standard Gaussian random variables.
\end{enumerate}
The measurement matrix $A$ and the measurement vector $b$ were constructed
as follows. We set the  number of observations $m =
\lceil\frac{n}{4}\rceil$. Each element $A_{ij}$ were sampled IID from a
standard Normal distribution. The
measurement $b=A x_*+\zeta$,
where each component $\zeta_i\in\reals^m$ was sampled IID
from a mean $0$ and variance $\varrho^2$ Normal
distribution. Therefore, the signal to noise ratio~(SNR) of the
measurement $b$ was
\begin{align}
\label{ch2_eq:snr-relation}
\proc{SNR}(b) =
10\log_{10}\left(\frac{\mE[\norm{Ax_*}_2^2]}{\mE[\norm{\zeta}_2^2]}\right) =
10\log_{10}\left(\frac{s}{\varrho^2}\right),
\end{align}
or equivalently, $\varrho^2 = s 10^{-\text{SNR}(b)/10}$.
% Hence, for a given \proc{SNR} value we select $\varrho$ according to
% \eqref{ch2_eq:snr-relation}.
For each random $\xbp$ and $A$, we considered SNR equal to $20$dB, $30$dB and
$40$dB.
\begin{table}[!htb]
    \centering
    {\scriptsize
    \begin{tabular}{cc|c|c|c|}
    \cline{2-5}
    &\multicolumn{2}{|c|}{\textbf{FAL}} &\multicolumn{2}{|c|}{\textbf{FPC-AS}}\\ \cline{2-5}
    &\multicolumn{1}{|c|}{\textbf{Average}}&\textbf{Max}&\textbf{Average}&\textbf{Max}\\ \hline
    \multicolumn{1}{|c|}{$\mathbf{|\norm{x_{sol}}_1-\norm{x_*}_1|/\norm{x_*}_{1}}$}
    &2.6E-09&3.2E-09&3.5E-08&3.6E-08\\ \hline
    \multicolumn{1}{|c|}{$\mathbf{\max\{|(x_{sol})_i-(x_*)_i|: (x_*)_i\neq 0\}}$}
    &5.1E-04&6.2E-04&6.5E-04&7.1E-04\\ \hline
    \multicolumn{1}{|c|}{$\mathbf{\max\{|(x_{sol})_i|: (x_*)_i=0\}}$}
    &0&0&1.2E-04&1.5E-04\\ \hline
    \multicolumn{1}{|c|}{$\mathbf{\|Ax_{sol}-b\|_2}$}
    &3.7E-03&4.4E-03&1.2E-02&1.2E-02\\ \hline
    \multicolumn{1}{|c|}{$\mathbf{\|x_{sol}\|_1}$}
    &56631797.8&59669841.4&56631795.7&59669839.3\\ \hline
    \multicolumn{1}{|c|}{$\mathbf{\|x_*\|_1}$}
    &56631797.7&59669841.3&56631797.7&59669841.3\\ \hline
    \multicolumn{1}{|c|}{$\mathbf{CPU}$}
    &11.0&12.3&22.2&23.9\\ \hline
    \multicolumn{1}{|c|}{$\mathbf{nMat}$}
    &98&99&109~/~205.6&109~/~208\\ \hline\\
    \cline{2-5}
    &\multicolumn{2}{|c|}{\textbf{SPA}}&\multicolumn{2}{|c|}{\textbf{NESTA}}\\ \cline{2-5}
    &\multicolumn{1}{|c|}{\textbf{Average}}&\textbf{Max}&\textbf{Average}&\textbf{Max}\\ \hline
    \multicolumn{1}{|c|}{$\mathbf{|\norm{x_{sol}}_1-\norm{x_*}_1|/\norm{x_*}_{1}}$}
    &1.0E-08&1.1E-08&6.5E-08&6.7E-08\\ \hline
    \multicolumn{1}{|c|}{$\mathbf{\max\{|(x_{sol})_i-(x_*)_i|: (x_*)_i\neq 0\}}$}
    &6.0E-04&6.8E-04&7.4E-04&8.4E-04\\ \hline
    \multicolumn{1}{|c|}{$\mathbf{\max\{|(x_{sol})_i|: (x_*)_i=0\}}$}
    &6.6E-05&7.1E-05&2.3E-04&3.1E-04\\ \hline
    \multicolumn{1}{|c|}{$\mathbf{\|Ax_{sol}-b\|_2}$}
    &6.0E-03&6.3E-03&4.0E-10&4.1E-10\\ \hline
    \multicolumn{1}{|c|}{$\mathbf{\|x_{sol}\|_1}$}
    &56631798.3&59669841.9&56631801.4&59669845.0\\ \hline
    \multicolumn{1}{|c|}{$\mathbf{\|x_*\|_1}$}
    &56631797.7&59669841.3&56631797.7&59669841.3\\ \hline
    \multicolumn{1}{|c|}{$\mathbf{CPU}$}
    &67.3&73.0&72.1&80.1\\ \hline
    \multicolumn{1}{|c|}{$\mathbf{nMat}$}
    &583.2&587&632.4&636\\ \hline\\
    \cline{2-5}
    &\multicolumn{2}{|c|}{\textbf{FPC}} &\multicolumn{2}{|c|}{\textbf{FPC-BB}}\\ \cline{2-5}
    &\multicolumn{1}{|c|}{\textbf{Average}}&\textbf{Max}&\textbf{Average}&\textbf{Max}\\ \hline
    \multicolumn{1}{|c|}{$\mathbf{|\norm{x_{sol}}_1-\norm{x_*}_1|/\norm{x_*}_{1}}$}
    &3.5E-08&3.5E-08&3.2E-08&3.3E-08\\ \hline
    \multicolumn{1}{|c|}{$\mathbf{\max\{|(x_{sol})_i-(x_*)_i|: (x_*)_i\neq 0\}}$}
    &6.8E-04&7.3E-04&6.1E-04&6.7E-04\\ \hline
    \multicolumn{1}{|c|}{$\mathbf{\max\{|(x_{sol})_i|: (x_*)_i=0\}}$}
    &1.6E-04&1.9E-04&1.3E-04&1.6E-04\\ \hline
    \multicolumn{1}{|c|}{$\mathbf{\|Ax_{sol}-b\|_2}$}
    &1.2E-02&1.2E-02&1.1E-02&1.1E-02\\ \hline
    \multicolumn{1}{|c|}{$\mathbf{\|x_{sol}\|_1}$}
    &56631795.7&59669839.3&56631795.9&59669839.5\\ \hline
    \multicolumn{1}{|c|}{$\mathbf{\|x_*\|_1}$}
    &56631797.7&59669841.3&56631797.7&59669841.3\\ \hline
    \multicolumn{1}{|c|}{$\mathbf{CPU}$}
    &40.4&50.0&22.7&26.4\\ \hline
    \multicolumn{1}{|c|}{$\mathbf{nMat}$}
    &383.0&387&195.0&195\\ \hline\\
    \cline{2-5}
    &\multicolumn{2}{|c|}{\textbf{YALL1~(BP)}} &\multicolumn{2}{|c|}{\textbf{SPGL1}}\\ \cline{2-5}
    &\multicolumn{1}{|c|}{\textbf{Average}}&\textbf{Max}&\textbf{Average}&\textbf{Max}\\ \hline
    \multicolumn{1}{|c|}{$\mathbf{|\norm{x_{sol}}_1-\norm{x_*}_1|/\norm{x_*}_{1}}$}
    &9.4E-10&1.4E-09&3.2E-09&6.7E-09\\ \hline
    \multicolumn{1}{|c|}{$\mathbf{\max\{|(x_{sol})_i-(x_*)_i|: (x_*)_i\neq 0\}}$}
    &5.7E-04&8.0E-04&5.3E-04&7.6E-04\\ \hline
    \multicolumn{1}{|c|}{$\mathbf{\max\{|(x_{sol})_i|: (x_*)_i=0\}}$}
    &1.5E-19&1.5E-19&2.4E-04&3.3E-04\\ \hline
    \multicolumn{1}{|c|}{$\mathbf{\|Ax_{sol}-b\|_2}$}
    &4.4E-03&5.5E-03&4.2E-03&4.9E-03\\ \hline
    \multicolumn{1}{|c|}{$\mathbf{\|x_{sol}\|_1}$}
    &56631797.7&59669841.3&56631797.5&59669841.1\\ \hline
    \multicolumn{1}{|c|}{$\mathbf{\|x_*\|_1}$}
    &56631797.7&59669841.3&56631797.7&59669841.3\\ \hline
    \multicolumn{1}{|c|}{$\mathbf{CPU}$}
    &44.9&53.3&24.7&28.4\\ \hline
    \multicolumn{1}{|c|}{$\mathbf{nMat}$}
    &453.0&477&200.7&209\\ \hline
    \end{tabular}
    }
    \caption{Noiseless comparison tests:  $m=n/4$, $s=m/10$ and $\|x_{sol}-x_*\|_{\infty}\approx 5\times 10^{-4}$}
    \label{ch2_tab:comparative_test_results}
    \vspace{-1cm}
\end{table}
\subsubsection{Comparison with other solvers}
\label{ch2_sec:test_noise}
% In this section we report the results of our numerical experiments
% comparing FAL with NESTA~\cite{Can09_4J},
% FPC-BB~\cite{Yin07_1R,Yin08_1J}, FPC-AS~\cite{Wen09_1R} and YALL1~\cite{Yang09}. In our
% experiments we used NESTA
% v.1.1~[\url{http://www.acm.caltech.edu/~nesta/}], FPC-BB that comes with
% v.2.0 package~[\url{http://www.caam.rice.edu/~optimization/L1/fpc/}], FPC-AS
% v.1.21~[\url{http://www.caam.rice.edu/~optimization/L1/FPC_AS/}]  and
%  YALL1 v1.3~[\url{http://yall1.blogs.rice.edu/}].
For each noise level, we created 10 random instances of size $n=128\times128$ using the procedure
described in Section~\ref{ch2_sec:expt-setup-noise}. %  We have used the same
% termination criterion in all the algorithms
We stopped each algorithm when the relative $\ell_2$-distance  of  consecutive iterates are less than
$\varrho$, i.e. we impose the noisy stopping condition in Section~\ref{sec:stop} for all the solvers.
% \begin{align}
% \frac{\norm{x^{(k+1)}-\xk}_2}{\norm{\xk}_2}\leq\varrho,
% \end{align}
% for the algorithms, where $\varrho$ is the standard deviation of Gaussian
% noise in $b$, which is defined in
% Section~\ref{ch2_sec:expt-setup-noise}.

Some of the solvers we tested solve the penalty formulation
$\min_{x\in\Re^n}\norm{x}_1+\frac{1}{\lambda}\norm{Ax-b}_2^2$.
Hale et. al.~\cite{Yin07_1R} proposed that
when the measurement noise vector $\zeta$
is a $N(0,\sigma)$ Gaussian vector, the penalty parameter
$\lambda$ should be set to $\lambda  =
\frac{\varrho~\sigma_{\min}(A)}{\sigma^2_{\max}(A)}~
\sqrt{\frac{\chi^2_{1-\alpha, m}}{n}}$, where $\chi^2_{1-\alpha, m}$
denotes the $1-\alpha$ critical value
of the $\chi^2$ distribution with $m$ degrees of freedom.
% Since each element $A_{ij}$ is an  IID samples from the standard
% Normal distribution and $n$ is large, we approximate
% $\sigma_{\min}(A) = \sqrt{n}(1-\sqrt{\frac{m}{n}})$ and $\sigma_{\max}(A)
% = \sqrt{n}(1+\sqrt{\frac{m}{n}})$~\cite{edelman88:eigen}.
%  is approximated
% by $(1-\sqrt{\frac{m}{n}})^2n$ and $($,
% respectively.
% $(1-\sqrt{\frac{m}{n}})^2n\leq\sigma^2_{\min}(A) \leq
% \sigma^2_{\max}(A)\leq(1+\sqrt{\frac{m}{n}})^2n$, and the
% bounds become tight as $n\rightarrow\infty$
We used the function \texttt{getM\_mu.m} from FPC~v.2.0
package to compute $\lambda$  according to this formula. %  using
% the approximations for $\sigma_{\min}(A)$ and
% $\sigma_{\max}(A)$.
% Depending on the SNR, we calculate different
% $\frac{1}{\lambda}$ values, e.g. $\frac{1}{\lambda}$ take values $5.7
% \times 10^3$, $1.8 \times 10^3$ and $5.7 \times 10^2$ when SNR is $40dB$,
% $30dB$ and $20dB$, respectively.
The other parameters were set as follows.
\begin{enumerate}[1.]
\item\textbf{FAL}: We set the initial update coefficients
  $c_\lambda^{(1)}=0.4$, $c_\tau^{(1)}=0.4$ and
  $t^{(1)}=2$. % We set $\gamma=5\times10^{-9}$ and the initial update
  % coefficients $c_\lambda^{(1)}=0.8$, $c_\tau^{(1)}=0.8$ and $t^{(1)}=1.9$.
  For $k\geq 2$,
  % $c_\lambda^{(k)}$, $c_\tau^{(k)}$ and $t^{(k)}$ values were  set as
  % described in
  we used the functions $c_{\lambda}(\cdot)$ and $t(\cdot)$  described in
  \eqref{ch2_eq:xi_function} and \eqref{ch2_eq:t_function}, respectively.
  The parameters $(\lk,\tk,\epsk)$ were set as described in
  Section~\ref{ch2_sec:multiplier_selection}.
\item \textbf{SPA}: $c_\tau^{(0)} = 0.2$, $c^{(1)}_\tau = 0.855$, $c_\epsilon = 0.8$, $c_\lambda = 0.9$ and $c_\mu
  = c_\nu= 0.1$. See \cite{AybatI09:SPA} the parameter definitions.
\item \textbf{NESTA}: NESTA solves $\min_{\norm{Ax-b}_2\leq \delta}
  p_{\mu}(x)$, where $p_\mu(x)=\max\{x^Tu-\frac{\mu}{2}\norm{u}_2^2:
  \norm{u}_\infty\leq 1\}$. $\mu=2\times 10^{-3}$ and the model parameter
  $\delta$ was  set to $\sqrt{m+2\sqrt{2m}}~\varrho$ as described in
  \cite{Can09_4J}.
\item \textbf{FPC} and \textbf{FPC-BB}: FPC and FPC-BB solve
  $\min_{x\in\Re^n}\norm{x}_1+\frac{1}{2\lambda}\norm{Ax-b}_2^2$;
  $\lambda$ was set as described above.
\item \textbf{FPC-AS}: FPC-AS solves
  $\min_{x\in\Re^n}\lambda\norm{x}_1+\frac{1}{2}\norm{Ax-b}_2^2$ and
  $\lambda$ was set as described.
\item \textbf{YALL1 (L1/L2)}: (L1/L2) option solves $\min_{x\in\Re^n} \norm{x}_1+\frac{1}{2\lambda}\norm{Ax-b}_2^2$ and
  $\lambda$ was set as described above.
\item \textbf{YALL1 (L1/L2con)}: (L1/L2con) option solves
  $\min_{\norm{Ax-b}_2\leq \delta} \norm{x}_1$, where the model parameter
  $\delta$ was set to $\sqrt{m+2\sqrt{2m}}~\varrho$.
\end{enumerate}
All the parameters other than ones explained above were set to their
default values. The results of the experiments are displayed in
Tables~\ref{ch2_tab:noisy_comperative_test_results_1} --
\ref{ch2_tab:noisy_comperative_test_results_3}.
% n Table~\ref{ch2_tab:noisy_comperative_test_results_1} --
% \ref{ch2_tab:noisy_comperative_test_results_3},
% the row labeled $\mathbf{CPU}$ lists the running time of each algorithm in
% \emph{seconds and all the other rows labels are as defined in
% Section~\ref{ch2_sec:selftest_results}. The experimental results in
% Table~\ref{ch2_tab:noisy_comperative_test_results_1}
% --\ref{ch2_tab:noisy_comperative_test_results_3},
\begin{table}[!htb]
    \begin{adjustwidth}{-2em}{-2em}
    \centering
    {\scriptsize
    \begin{tabular}{cc|c|c|c|c|c|c|c|}
    \cline{2-9}
    &\multicolumn{2}{|c|}{\textbf{FAL}}&\multicolumn{2}{|c|}{\textbf{FPC-AS}}&\multicolumn{2}{|c|}{\textbf{FPC}}&\multicolumn{2}{|c|}{\textbf{YALL1 (L1/L2)}}\\ \cline{2-9}
    &\multicolumn{1}{|c|}{\textbf{Average}}&\textbf{Max}&\textbf{Average}&\textbf{Max}&\textbf{Average}&\textbf{Max}&\textbf{Average}&\textbf{Max}\\ \hline
    \multicolumn{1}{|c|}{$\mathbf{\norm{x_{sol}-x_*}_2/\norm{x_*}_2}$}
    &0.007&0.008&0.007&0.008&0.012&0.013&0.008&0.009\\ \hline
    \multicolumn{1}{|c|}{$\mathbf{\max\{|(x_{sol})_i-(x_*)_i|: (x_*)_i\neq 0\}}$}
    &1.4E-02&1.7E-02&2.2E-02&2.9E-02&2.4E-02&2.6E-02&1.7E-02&2.0E-02 \\ \hline
    \multicolumn{1}{|c|}{$\mathbf{\max\{|(x_{sol})_i|: (x_*)_i=0\}}$}
    &1.1E-02&1.3E-02&9.0E-03&1.1E-02&1.4E-02&1.6E-02&1.4E-02&1.6E-02 \\ \hline
    \multicolumn{1}{|c|}{$\mathbf{\|Ax_{sol}-b\|_2}$}
    &4.1E-02&4.7E-02&4.4E-02&5.0E-02&2.5E-02&2.5E-02&5.1E-02&5.3E-02\\ \hline
    \multicolumn{1}{|c|}{$\mathbf{CPU}$}
    &13.1&13.8&18.9&20.3&156.9&166.9&31.8&34.1\\ \hline
    \multicolumn{1}{|c|}{$\mathbf{nMat}$}
    &62.8&65&67.8/113.8&71/117&735.4&769&149.5&157\\ \hline\\
    \cline{2-9}
    &\multicolumn{2}{|c|}{\textbf{SPA}}&\multicolumn{2}{|c|}{\textbf{NESTA}}&\multicolumn{2}{|c|}{\textbf{FPC-BB}}&\multicolumn{2}{|c|}{\textbf{YALL1 (L1/L2con)}}\\ \cline{2-9}
    &\multicolumn{1}{|c|}{\textbf{Average}}&\textbf{Max}&\textbf{Average}&\textbf{Max}&\textbf{Average}&\textbf{Max}&\textbf{Average}&\textbf{Max}\\ \hline
    \multicolumn{1}{|c|}{$\mathbf{\norm{x_{sol}-x_*}_2/\norm{x_*}_2}$}
    &0.011&0.012&0.019&0.020&0.012&0.012&0.013&0.014 \\ \hline
    \multicolumn{1}{|c|}{$\mathbf{\max\{|(x_{sol})_i-(x_*)_i|: (x_*)_i\neq 0\}}$}
    &2.3E-02&3.5E-02&4.1E-02&4.5E-02&2.3E-02&2.6E-02&2.2E-02&2.6E-02 \\ \hline
    \multicolumn{1}{|c|}{$\mathbf{\max\{|(x_{sol})_i|: (x_*)_i=0\}}$}
    &1.0E-02&1.4E-02&1.3E-02&1.5E-02&1.3E-02&1.6E-02&1.6E-02&1.8E-02 \\ \hline
    \multicolumn{1}{|c|}{$\mathbf{\|Ax_{sol}-b\|_2}$}
    &2.9E-02&6.8E-02&6.9E-02&6.9E-02&2.5E-02&2.6E-02&1.5E-02&1.6E-02 \\ \hline
    \multicolumn{1}{|c|}{$\mathbf{CPU}$}
    &66.8&76.1&264.1&293.0&39.4&43.7&28.5&30.8 \\ \hline
    \multicolumn{1}{|c|}{$\mathbf{nMat}$}
    &326.6&375&536.0&553&180.8&189&137.0&137 \\ \hline
    \multicolumn{1}{|c|}{$\mathbf{Preprocess Time}$}
    &N/A&N/A&581.7&667.8&N/A&N/A&N/A&N/A\\ \hline
    \end{tabular}
    }
    \end{adjustwidth}
    \caption{Noisy comparative tests: SNR($b$)=40dB}
    \label{ch2_tab:noisy_comperative_test_results_1}
    \vspace{-0.7cm}
\end{table}
\begin{table}[!htb]
    \begin{adjustwidth}{-2em}{-2em}
    \centering
    {\scriptsize
    \begin{tabular}{cc|c|c|c|c|c|c|c|}
    \cline{2-9}
    &\multicolumn{2}{|c|}{\textbf{FAL}}&\multicolumn{2}{|c|}{\textbf{FPC-AS}}&\multicolumn{2}{|c|}{\textbf{FPC}}&\multicolumn{2}{|c|}{\textbf{YALL1 (L1/L2)}}\\ \cline{2-9}
    &\multicolumn{1}{|c|}{\textbf{Average}}&\textbf{Max}&\textbf{Average}&\textbf{Max}&\textbf{Average}&\textbf{Max}&\textbf{Average}&\textbf{Max}\\ \hline
    \multicolumn{1}{|c|}{$\mathbf{\norm{x_{sol}-x_*}_2/\norm{x_*}_2}$}
    &0.024&0.027&0.023&0.027&0.036&0.038&0.031&0.033 \\ \hline
    \multicolumn{1}{|c|}{$\mathbf{\max\{|(x_{sol})_i-(x_*)_i|: (x_*)_i\neq 0\}}$}
    &5.4E-02&6.1E-02&5.8E-02&7.4E-02&7.3E-02&8.0E-02&6.0E-02&6.9E-02 \\ \hline
    \multicolumn{1}{|c|}{$\mathbf{\max\{|(x_{sol})_i|: (x_*)_i=0\}}$}
    &3.7E-02&4.0E-02&3.8E-02&4.2E-02&4.1E-02&4.9E-02&4.1E-02&4.7E-02 \\ \hline
    \multicolumn{1}{|c|}{$\mathbf{\|Ax_{sol}-b\|_2}$}
    &1.2E-01&1.3E-01&1.0E-01&1.1E-01&7.8E-02&7.9E-02&1.1E-01&1.2E-01 \\ \hline
    \multicolumn{1}{|c|}{$\mathbf{CPU}$}
    &10.8&11.0&19.8&22.2&90.1&101&21.7&22.3 \\ \hline
    \multicolumn{1}{|c|}{$\mathbf{nMat}$}
    &51.8&53&72.4/118.4&79/125&436.8&493&106.0&107 \\ \hline\\
    \cline{2-9}
    &\multicolumn{2}{|c|}{\textbf{SPA}}&\multicolumn{2}{|c|}{\textbf{NESTA}}&\multicolumn{2}{|c|}{\textbf{FPC-BB}}&\multicolumn{2}{|c|}{\textbf{YALL1 (L1/L2con)}}\\ \cline{2-9}
    &\multicolumn{1}{|c|}{\textbf{Average}}&\textbf{Max}&\textbf{Average}&\textbf{Max}&\textbf{Average}&\textbf{Max}&\textbf{Average}&\textbf{Max}\\ \hline
    \multicolumn{1}{|c|}{$\mathbf{\norm{x_{sol}-x_*}_2/\norm{x_*}_2}$}
    &0.023&0.025&0.078&0.082&0.035&0.036&0.036&0.038 \\ \hline
    \multicolumn{1}{|c|}{$\mathbf{\max\{|(x_{sol})_i-(x_*)_i|: (x_*)_i\neq 0\}}$}
    &5.1E-02&5.7E-02&1.7E-01&2.0E-01&6.9E-02&7.6E-02&6.0E-02&6.8E-02 \\ \hline
    \multicolumn{1}{|c|}{$\mathbf{\max\{|(x_{sol})_i|: (x_*)_i=0\}}$}
    &3.2E-02&3.9E-02&6.2E-02&7.3E-02&3.9E-02&4.6E-02&4.5E-02&5.3E-02 \\ \hline
    \multicolumn{1}{|c|}{$\mathbf{\|Ax_{sol}-b\|_2}$}
    &1.1E-01&1.1E-01&2.2E-01&2.2E-01&7.8E-02&7.9E-02&3.6E-02&3.8E-02 \\ \hline
    \multicolumn{1}{|c|}{$\mathbf{CPU}$}
    &55.4&58.6&207.1&221.1&25.8&27.9&20.6&21.5 \\ \hline
    \multicolumn{1}{|c|}{$\mathbf{nMat}$}
    &267.8&287&354.0&363&122.4&135&103.5&107 \\ \hline
    \multicolumn{1}{|c|}{$\mathbf{Preprocess Time}$}
    &N/A&N/A&581.7&667.8&N/A&N/A&N/A&N/A \\ \hline
    \end{tabular}
    }
    \end{adjustwidth}
    \caption{Noisy comparative tests: SNR($b$)=30dB}
    \label{ch2_tab:noisy_comperative_test_results_2}
    \vspace{-0.65cm}
\end{table}
\begin{table}[!htb]
    \begin{adjustwidth}{-2em}{-2em}
    \centering
    {\scriptsize
    \begin{tabular}{cc|c|c|c|c|c|c|c|}
    \cline{2-9}
    &\multicolumn{2}{|c|}{\textbf{FAL}}&\multicolumn{2}{|c|}{\textbf{FPC-AS}}&\multicolumn{2}{|c|}{\textbf{FPC}}&\multicolumn{2}{|c|}{\textbf{YALL1 (L1/L2)}}\\ \cline{2-9}
    &\multicolumn{1}{|c|}{\textbf{Average}}&\textbf{Max}&\textbf{Average}&\textbf{Max}&\textbf{Average}&\textbf{Max}&\textbf{Average}&\textbf{Max}\\ \hline
    \multicolumn{1}{|c|}{$\mathbf{\norm{x_{sol}-x_*}_2/\norm{x_*}_2}$}
    &0.090&0.103&0.099&0.105&0.104&0.111&0.100&0.107 \\ \hline
    \multicolumn{1}{|c|}{$\mathbf{\max\{|(x_{sol})_i-(x_*)_i|: (x_*)_i\neq 0\}}$}
    &2.0E-01&2.4E-01&2.0E-01&2.3E-01&2.1E-01&2.4E-01&1.9E-01&2.2E-01 \\ \hline
    \multicolumn{1}{|c|}{$\mathbf{\max\{|(x_{sol})_i|: (x_*)_i=0\}}$}
    &1.3E-01&1.9E-01&1.2E-01&1.4E-01&1.2E-01&1.4E-01&1.3E-01&1.6E-01 \\ \hline
    \multicolumn{1}{|c|}{$\mathbf{\|Ax_{sol}-b\|_2}$}
    &3.4E-01&5.1E-01&2.8E-01&2.9E-01&2.5E-01&2.6E-01&2.9E-01&2.9E-01 \\ \hline
    \multicolumn{1}{|c|}{$\mathbf{CPU}$}
    &8.3&8.6&22.6&23.9&71.2&74.5&13.8&14.1 \\ \hline
    \multicolumn{1}{|c|}{$\mathbf{nMat}$}
    &39.6&41&78.8/124.8&83/129&349.6&365&67.0&67 \\ \hline\\
    \cline{2-9}
    &\multicolumn{2}{|c|}{\textbf{SPA}}&\multicolumn{2}{|c|}{\textbf{NESTA}}&\multicolumn{2}{|c|}{\textbf{FPC-BB}}&\multicolumn{2}{|c|}{\textbf{YALL1 (L1/L2con)}}\\ \cline{2-9}
    &\multicolumn{1}{|c|}{\textbf{Average}}&\textbf{Max}&\textbf{Average}&\textbf{Max}&\textbf{Average}&\textbf{Max}&\textbf{Average}&\textbf{Max}\\ \hline
    \multicolumn{1}{|c|}{$\mathbf{\norm{x_{sol}-x_*}_2/\norm{x_*}_2}$}
    &0.108&0.116&0.251&0.267&0.100&0.109&0.124&0.132 \\ \hline
    \multicolumn{1}{|c|}{$\mathbf{\max\{|(x_{sol})_i-(x_*)_i|: (x_*)_i\neq 0\}}$}
    &2.0E-01&2.4E-01&5.4E-01&6.1E-01&2.0E-01&2.3E-01&2.1E-01&2.3E-01 \\ \hline
    \multicolumn{1}{|c|}{$\mathbf{\max\{|(x_{sol})_i|: (x_*)_i=0\}}$}
    &1.3E-01&1.5E-01&1.8E-01&2.1E-01&1.2E-01&1.4E-01&1.6E-01&1.8E-01 \\ \hline
    \multicolumn{1}{|c|}{$\mathbf{\|Ax_{sol}-b\|_2}$}
    &1.4E-01&1.4E-01&6.9E-01&6.9E-01&2.5E-01&2.5E-01&1.4E-01&1.6E-01 \\ \hline
    \multicolumn{1}{|c|}{$\mathbf{CPU}$}
    &54.8&58.9&170.1&176.4&23.8&27.0&15.4&16.1 \\ \hline
    \multicolumn{1}{|c|}{$\mathbf{nMat}$}
    &268.0&289&225.0&233&104.2&109&76.5&77 \\ \hline
    \multicolumn{1}{|c|}{$\mathbf{Preprocess Time}$}
    &N/A&N/A&581.7&667.8&N/A&N/A&N/A&N/A \\ \hline
    \end{tabular}
    }
    \end{adjustwidth}
    \caption{Noisy comparative tests: SNR($b$)=20dB}
    \label{ch2_tab:noisy_comperative_test_results_3}
    \vspace{-0.65cm}
\end{table}
\newpage
The results clearly show that % in order to produce a solution of smaller
                              % $\ell_\infty$-error,
%even though the FAL implementation that we tested was not designed to solve constrained basis pursuit denoising problems, it is significantly
FAL is faster then the other state-of-the-art algorithms over the SNR range 20dB--40dB.
Since  $\mathbf{nMat}$ only keeps tracks of matrix-vector multiplies with
the full $m\times n$ measurement matrix,  the CPU time $\mathbf{CPU}$ for
some of the solvers is not completely determined by
$\mathbf{nMat}$. For instance, at the
40dB SNR level, FAL computed 62.8 and FPC-AS computed 67.8 matrix-vector
multiplications on average; but the average CPU time for FAL was 13.1
secs, whereas it was 18.9 secs for FPC-AS.
This difference in the CPU time is due to
smaller size matrix-vector multiplies % with a reduced $A$
that FPC-AS computes during subspace optimization iterations. NESTA has the highest overhead cost: it needs SVD of $A=U\Sigma V^T$ at the beginning since Gaussian $A$ does not satisfy $AA^T=I$. The preprocess time reported for NESTA shows the time to compute the SVD of $A$. Moreover, on top of the reported number of matrix vector multiplications with $m\times n$ matrices, NESTA also computes 2 matrix vector multiplications with $m\times m$ matrices at each iteration, which is not reported.
\subsection{Comparison with other solvers on hard instances}
\label{sec:hard}
In order to demonstrate the robustness of FAL, we tested it on the Caltech
test problems: \textbf{CaltechTest1},
\textbf{CaltechTest2}, \textbf{CaltechTest3} and \textbf{CaltechTest4}
\cite{caltech08}. % The reason of pathology of
This is a set of small-sized hard instances of CS problems. The hardness
of these instances is due to the very large dynamic range of
the nonzero components~(see
Table~\ref{ch2_tab:caltech_properties}). For example, the target signal
$x_*\in\reals^{512}$ in \textbf{CaltechTest1} has $33$ nonzero components
with magnitude of $10^5$ and 5 components with magnitude of $1$,
i.e. $x_*$ has a dynamic range of 100dB.
\begin{table}[!h]
  \centering
  {\scriptsize
  \begin{tabular}{|c|c|c|c|c|}
    \hline
    \textbf{problem} & \textbf{n} & \textbf{m} & \textbf{s} & \textbf{(magnitude, \# elements of this magnitude)} \\ \hline
    \textbf{CaltechTest1} & 512 & 128 & 38 & $(10^5, 33),\ (1,5)$ \\ \hline
    \textbf{CaltechTest2} & 512 & 128 & 37 & $(10^5, 32),\ (1,5)$ \\ \hline
    \textbf{CaltechTest3} & 512 & 128 & 32 & $(10^{-1}, 31),\ (10^{-6},1)$ \\ \hline
    \textbf{CaltechTest4} & 512 & 102 & 26 & $(10^4, 13),\ (1, 12),\ (10^{-2},1)$ \\ \hline
  \end{tabular}
  }
  \caption{Characteristics of The Problems in CaltechTest Problem Set}\label{ch2_tab:caltech_properties}
\end{table}
\vspace{-5mm}

The Caltech problems have measurement noise. However, the $\proc{SNR}(b)=20\log_{10}\left(\frac{\norm{Ax_*}_2}{\norm{\zeta}_2}\right)$ for
\textbf{CaltechTest1}--\textbf{CaltechTest4}
% \textbf{CaltechTest2}, \textbf{CaltechTest3} and
problems is 228dB, 265dB, 168dB and 261dB,
respectively. Since the SNR values are very high,
we solved this set of problems solve them via basis pursuit formulation~\eqref{ch2_eq:l1_minimization} using FAL, SPA, NESTA, YALL1 and SPGL1; and via unconstrained basis pursuit denoising formulation \eqref{ch2_eq:l1_minimization_normsq} with small $\bar{\lambda}>0$ values using FPC, FPC-BB and FPC-AS.
% we treat this problem set under this section and solve them via basis pursuit
% formulation~\eqref{ch2_eq:l1_minimization}.

For FAL, SPA, NESTA~v1.1, FPC and FPC-BB that come with v2.0 solver
package, FPC-AS~v1.21, YALL1~v1.4 and SPGL1~v1.7, we chose parameter values that
produced a solution $x_{\text{sol}}$ with high accuracy in reasonable time.
\begin{enumerate}[1.]
\item \textbf{FAL}: We set $\gamma=5\times10^{-9}$ and the initial update
  coefficients $c_\lambda^{(1)}=0.8$, $c_\tau^{(1)}=0.8$ and $t^{(1)}=1.9$. For $k\geq 2$,
  % $c_\lambda^{(k)}$, $c_\tau^{(k)}$ and $t^{(k)}$ values were  set as
  % described in
  we used the functions $c_{\lambda}(\cdot)$ and $t(\cdot)$  described in
  \eqref{ch2_eq:xi_function} and \eqref{ch2_eq:t_function}, respectively.
  The parameters $(\lk,\tk,\epsk)$ were set as described in
  Section~\ref{ch2_sec:multiplier_selection}.
\item \textbf{SPA}: $\gamma=1\times10^{-8}$, $c_\tau^{(0)} = 0.1$,
  $c^{(1)}_\tau = 0.76$, $c_\epsilon = 0.8$, $c_\lambda = 0.95$ and $c_\mu
  = c_\nu= 0.4$. For details on these parameters refer to
  \cite{AybatI09:SPA}.
\item \textbf{NESTA}: $\mu=1\times 10^{-6}$ and $\gamma=1\times 10^{-16}$. See
  Section~\ref{ch2_sec:test} for the definition of $\mu$ and $\gamma$.
\item \textbf{FPC} and \textbf{FPC-BB}:  $\frac{1}{\lambda}= 1 \times
  10^{10}$, ${\rm xtol}=10^{-10}$, ${\rm gtol} =10^{-8}$ and ${\rm mxitr}=20000$, where
  ${\rm xtol}$, ${\rm gtol}$ set the termination conditions on the relative change in
  iterates and gradient, respectively, and ${\rm mxitr}$ is the iteration limit
  allowed. See Section~\ref{ch2_sec:test} for the definition of $\lambda$.
\item \textbf{FPC-AS}: $\lambda=1\times 10^{-10}$ and ${\rm gtol}= 10^{-14}$,
  where ${\rm gtol}$ is the termination criterion on the maximum norm of
  sub-gradient. See Section~\ref{ch2_sec:test} for the definition of $\lambda$.
\item \textbf{YALL1~(BP)}: $\gamma=1\times 10^{-11}$ and $nonorth=0$. See the item describing YALL1~(BP) for the definition of
 of $\gamma$ and $nonorth$ in Section~\ref{ch2_sec:test}.
 \item \textbf{SPGL1~(BP)}: $optTol=1\times 10^{-7}$, $bpTol=1\times 10^{-9}$ and $decTol=1\times 10^{-7}$. For the details on optimality and basis pursuit tolerance parameters, $optTol$, $bpTol$ and $decTol$, refer to \cite{Ber08_1J}.
\end{enumerate}
These parameter
values were fixed for all $4$ test problems  and all other parameters are
set to their default values.  The termination criteria were not directly
comparable since the different solvers use a slightly different
formulation of the basis pursuit problem. However,  we attempted to set
the stopping parameter $\gamma$ such that on the average the stopping
criterion for FAL was more stringent than those for the  other algorithms
we tested. The results of the experiments are displayed in Table~\ref{ch2_tab:caltech_results}.
In Table~\ref{ch2_tab:caltech_results}, the row labeled $\mathbf{CPU}$ lists
% the running time of each algorithm in \emph{seconds},
the row labeled
$\mathbf{rel.err}$ lists the relative error the solution, i.e
$\mathbf{rel.err}=\frac{\norm{x_{sol}-x_*}_2}{\norm{x_*}_2}$,
the row labeled $\mathbf{inf.err_+}$ lists the absolute error on the
nonzero components of $x_*$, i.e
$\mathbf{inf.err_+}=\max\{|(x_{sol})_i-(x_*)_i|: (x_*)_i\neq 0\}$, the row
labeled $\mathbf{inf.err_0}$ lists the absolute error on the zero
components of $x_*$,
without any thresholding or
post-processing.  None of the solvers, other than FAL,
were able to identify the true support of the target signal for any of the {\bf CaltechTest} instances.
\begin{table}
    \begin{adjustwidth}{-2em}{-2em}
    \centering
    {\scriptsize
    \begin{tabular}{|c|c|c|c|c|c|c|c|c|c|c|} \hline
        \textbf{Problem}&\textbf{Solver}& $\mathbf{\norm{x_*}_1}$ &
        $\mathbf{\norm{x_{sol}}_1}$ & rel.err & $inf.err_+$ & $inf.err_0$
        & $\mathbf{\norm{r}_2}$ & \textbf{CPU} & \textbf{nMat} &
        \textbf{nnz} \\ \hline
        \multirow{6}{*}{\textbf{Caltech1}}
        & \textbf{FAL} &\multirow{6}{*}{3300005}
        &3300005.00&5.15E-12&9.94E-07&0&1.16E-08&0.598&1715&38 \\
        & \textbf{SPA} &
        &3300005.00&1.85E-10&3.05E-05&1.68E-05&4.85E-06&7.783&20305&512 \\
        & \textbf{NESTA} &
        &3300005.00&2.43E-10&4.01E-05&2.18E-05&1.06E-10&9.902&18432&512 \\
        & \textbf{FPC} &
        &3300002.44&3.05E-06&5.17E-01&2.64E-01&1.78E-01&22.509&40001&109 \\
        & \textbf{FPC-BB} &
        &3300002.44&8.46E-06&5.16E-01&2.63E-01&1.78E-01&44.600&40001&109 \\
        & \textbf{FPC-AS} &
        &3300005.00&5.15E-12&9.97E-07&8.97E-10&1.62E-09&0.375&109~/~393&63 \\
        & \textbf{YALL1} &
        &3300005.00&5.61E-11&8.51E-05&1.19E-18&6.09E-05&5.486&14492&276 \\
        & \textbf{SPGL1} &
        &3300005.11&1.19E-06&1.20E+00&6.53E-01&9.94E-08&9.989&17705&171 \\ \hline
        \multirow{6}{*}{\textbf{Caltech2}}
        & \textbf{FAL} &\multirow{6}{*}{3300005}
        &3200005.00&7.17E-14&1.41E-08&0&7.03E-09&0.358&971&37 \\
        & \textbf{SPA} &
        &3200005.00&1.19E-10&2.04E-05&1.38E-05&4.73E-06&5.651&14001&512 \\
        & \textbf{NESTA} &
        &3200005.00&1.24E-10&2.10E-05&1.47E-05&9.34E-11&3.826&7204&512 \\
        & \textbf{FPC} &
        &3200004.97&2.15E-08&3.72E-03&2.32E-03&2.39E-03&23.540&40001&96 \\
        & \textbf{FPC-BB} &
        &3200004.47&3.43E-07&5.92E-02&3.70E-02&3.82E-02&43.019&40001&96 \\
        & \textbf{FPC-AS} &
        &3200005.00&7.58E-14&1.78E-08&2.03E-09&1.88E-09&0.222&127~/~407&63 \\
        & \textbf{YALL1} &
        &3200005.00&6.29E-11&1.01E-04&1.15E-18&5.76E-05&1.337&3137&275 \\
        & \textbf{SPGL1} &
        &3200005.00&1.35E-11&1.35E-05&8.15E-06&6.96E-08&16.413&28008&212 \\ \hline
        \multirow{6}{*}{\textbf{Caltech3}}
        & \textbf{FAL} &\multirow{6}{*}{6.200000974}
        &6.20000101&4.03E-08&1.49E-08&0&1.35E-08&0.166&359&32 \\
        & \textbf{SPA} &
        &6.19999388&5.82E-06&1.85E-06&8.99E-07&8.78E-07&4.663&9767&512 \\
        & \textbf{NESTA} &
        &6.20007451&5.02E-05&1.51E-05&8.72E-06&1.96E-16&5.131&8326&512 \\
        & \textbf{FPC} &
        &6.20000076&6.50E-08&2.01E-08&1.04E-08&1.80E-08&27.730&40001&78 \\
        & \textbf{FPC-BB} &
        &6.19975503&7.09E-05&2.22E-05&1.06E-05&1.84E-05&37.365&40001&80 \\
        & \textbf{FPC-AS} &
        &6.20000098&1.46E-09&3.78E-10&4.73E-10&1.23E-09&0.137&93~/~271&67 \\
        & \textbf{YALL1} &
        &6.30373200&8.53E-02&1.47E-01&1.19E-01&3.95E-16&23.048&50002&321 \\
        & \textbf{SPGL1} &
        &6.20000438&4.14E-06&6.62E-06&5.92E-06&9.99E-08&8.275&11885&131 \\ \hline
        \multirow{6}{*}{\textbf{Caltech4}}
        & \textbf{FAL} &\multirow{6}{*}{130012.01}
        &130012.010&1.28E-12&2.16E-08&0&1.24E-08&0.207&487&26 \\
        & \textbf{SPA} &
        &130012.010&3.80E-09&4.92E-05&1.86E-05&1.15E-05&3.788&8221&512 \\
        & \textbf{NESTA} &
        &130012.010&1.87E-09&2.37E-05&9.61E-06&5.71E-12&3.583&5904&512 \\
        & \textbf{FPC} &
        &130012.008&2.01E-08&2.62E-04&9.07E-05&1.39E-04&26.283&40001&71 \\
        & \textbf{FPC-BB} &
        &130010.234&1.92E-05&2.39E-01&7.46E-02&1.39E-01&44.804&40001&62 \\
        & \textbf{FPC-AS} &
        &130012.010&8.31E-13&9.01E-09&8.62E-09&3.86E-09&0.270&145~/~523&50 \\
        & \textbf{YALL1} &
        &130012.010&8.99E-11&5.34E-06&7.03E-20&3.64E-06&4.747&10682&305 \\
        & \textbf{SPGL1} &
        &130012.010&9.57E-11&4.42E-06&2.40E-06&9.97E-08&9.641&14647&97 \\ \hline
    \end{tabular}
    }
    \end{adjustwidth}
    \caption{CaltechTest Problem Set}\label{ch2_tab:caltech_results}
\end{table}
\section{Conclusion}
We propose a first-order augmented lagrangian algorithm~(FAL) for basis
pursuit.  FAL computes a solution to the basis pursuit problem  by solving a
sequence of augmented lagrangian subproblems, and  each subproblem is solved
using a variant of the infinite memory proximal gradient algorithm
(Algorithm 3)~\cite{Tseng08}. We prove
that FAL iterates converge to the optimal solution of the basis pursuit
problem whenever it is unique, which is true with overwhelming probability for compressed sensing problems (In~\cite{Can05_1J} Cand{\'e}s
and Tao have shown that for random measurement $A$ matrices the resulting
basis pursuit problem has a unique solution with very high
probability). We are able to prove FAL needs at most
$\cO(\epsilon^{-1})$ matrix-vector multiplies to compute an
$\epsilon$-feasible and $\epsilon$-optimal solution. However, in our numerical
experiments we observe that we only need $\cO(\log(\epsilon^{-1}))$
matrix-vector multiplies!
% We give a
% continuation scheme on penalty parameter $\lambda$
% used in FAL that guarantees the convergence of the iterate sequence to the
% target signal and we are also
% able to compute a convergence rate.
We found that for a fixed measurement ratio $m/n$, sparsity ratios $s/n$,
and solution accuracy $\gamma$, the number of matrix-vector multiplies
computed by FAL
were effectively independent of the dimension $n$. This allows us to
tune the algorithm parameters on small problems
and then use these parameters for all larger problems with the same
measurement and sparsity ratios. The numerical results reported in this
paper clearly show that FAL solves both the noise-less and noisy versions of the
compressive sensing problems very efficiently.
% required very
% few iterations to accurately recover the target signal.
% In our numerical
% studies, we also observed that FAL does $\cO(\log(\epsilon^{-1}))$ matrix
% multiplications as opposed to the $\cO(\epsilon^{-1})$ worst case
% theoretical bound shown.
\section{Acknowledgments}
We thank the anonymous referees for their insightful
comments that significantly improved both the algorithm and the paper.
% nd for pointing out to include numerical experiments which
% demonstrate the performance of FAL when there is measurement error.
We also thank Professor Y. Zhang for helping us better understand the
capabilities of YALL1.
%%%%%%%%%%%%%%%%%%%%%%%%%%%%%%%%%%%%%%%%%%%%%%%%%%%%%%%%%%%%%%%%%%%%%%%%%%%%%%%%
\newpage
\bibliographystyle{siam}
\bibliography{thesis}
\newpage
\appendix
\section{Auxiliary results}
\label{ch2_app:sec2}
$\mbox{}$\\
\begin{theorem}
\label{ch2_thm:grad_convergence}
Let $f:\reals^n\rightarrow\reals$ be a convex function. Suppose the $\grad
f$ is Lipschitz continuous with the Lipschitz constant $L$. % Let $x^*$
% denote any unconstrained minimizer of the function
% $\lambda\norm{.}_1+f(.)$.
Fix $\epsilon > 0$. Suppose $\bar{x}\in\reals^n$ satisfies
$\lambda\norm{\bar{x}}_1+f(\bar{x})-(\lambda\norm{x^*}_1+f(x^*))\leq\epsilon$,
where $x^\ast \in \argmin\{\lambda \norm{x}_1 + f(x): x \in \reals^n\}$.
Then
\begin{align}
\frac{1}{2L}\sum_{i:|\grad f_i(\bar{x})|>\lambda}(|\grad
f_i(\bar{x})|-\lambda)^2\leq \epsilon. \label{ch2_eq:Lipschitz_ineq_L1}
\end{align}
The bound \eqref{ch2_eq:Lipschitz_ineq_L1} implies $\norm{\grad
  f(\bar{x})}_\infty \leq \sqrt{2L\epsilon}+\lambda$.
\end{theorem}
\begin{proof}
% Given $\epsilon>0$, fix $\bar{x}\in\reals^n$ such that
% $\lambda\norm{\bar{x}}_1+f(\bar{x})-(\lambda\norm{x^*}_1+f(x^*))\leq\epsilon$.
Triangular inequality for $\norm{.}_1$ and Lipschitz continuity of $\grad f$
implies that for all $y \in \reals^n$
\eq
\lambda\norm{y}_1+f(y) \leq \lambda\norm{\bar{x}}_1+ f(\bar{x})+\grad
f(\bar{x})^T(y-\bar{x})+\frac{L}{2}\norm{y-\bar{x}}_2^2+\lambda\norm{y-\bar{x}}_1.
\en
% where $x^T y\in\Re$ denotes the usual Euclidean inner product of
% $x\in\Re^n$ and $y\in\Re^n$.
Taking the minimum with respect to $y$, we get
\begin{align}
\lambda\norm{x^*}_1+f(x^*)
&\leq\lambda\norm{\bar{x}}_1+f(\bar{x})+\min_{y\in\reals^n}\left\{\grad
  f(\bar{x})^T(y-\bar{x})+\frac{L}{2}\norm{y-\bar{x}}_2^2+\lambda\norm{y-\bar{x}}_1\right\}.
\label{ch2_eq:combined_upper_bound}
\end{align}
Let $w\equiv\grad f(\bar{x})$. Then
\begin{align}
y^* &=
\argmin_{y\in\reals^n} \left\{
  w^T(y-\bar{x})+\frac{L}{2}\norm{y-\bar{x}}_2^2 +
  \lambda\norm{y-\bar{x}}_1\right\},\\
    & =\argmin_{y\in\reals^n}\left\{
      \frac{1}{2}\norm{y-\bar{x}+\frac{w}{L}}_2^2 +
      \frac{\lambda}{L}\norm{y-\bar{x}}_1\right\},\\
    &= \bar{x} +\mbox{sign}\left( \frac{-w}{L}\right) \odot
    \max\left\{\left|\frac{-w}{L}\right|
      -\frac{\lambda}{L},0\right\}, \label{ch2_eq:shrinkage_solution}\\
    &= \bar{x} +
    \frac{-\mbox{sign}(w)}{L}\odot\max\{|w|-\lambda,0\},
    \label{ch2_eq:simplified_shrinkage_solution}
\end{align}
where \eqref{ch2_eq:shrinkage_solution} follows from the fact that
$\argmin_{z\in\reals^n}\{\nu\norm{z}_1+\frac{1}{2}\norm{z-\zeta}_2^2\} =
\mbox{sign}(\zeta)\odot\max\{|\zeta|-\nu,0\}$,
where $\odot$ is
component-wise multiplication operator~\cite{Yin08_1J}, and all
other vector operators such as $|\cdot|$, $\mbox{sign}(\cdot)$ and
$\max\{\cdot,\cdot\}$ are
defined to operate component-wise. Substituting $y^*$ in
\eqref{ch2_eq:combined_upper_bound}, we get
\begin{align}
     &\min_{y\in\reals^n}\left\{w^T(y-\bar{x})+
       \frac{L}{2}\norm{y-\bar{x}}_2^2+\lambda\norm{y-\bar{x}}_1\right\},
     \nonumber \\
    =&-\sum_i\frac{|w_i|}{L}\max\{|w_i|-\lambda,0\} +
    \frac{1}{2L}\sum_i\max\{|w_i|-\lambda,0\}^2 +
    \frac{\lambda}{L}\sum_i\max\{|w_i|-\lambda,0\},
    \nonumber \\
    =&\frac{1}{L}\sum_{i:|w_i|>\lambda}\left(-|w_i| +
      \frac{1}{2}(|w_i|-\lambda)+\lambda\right)(|w_i|-\lambda),
    \nonumber\\
    =& \mbox{}
    -\frac{1}{2L}\sum_{i:|w_i|>\lambda}(|w_i|-\lambda)^2.
    \label{ch2_eq:shrinkage_optvalue}
\end{align}
% Thus, \eqref{ch2_eq:shrinkage_optvalue} and
% \eqref{ch2_eq:combined_upper_bound} together give:
% \begin{align}
% \lambda\norm{x^*}_1+f(x^*) &\leq
% \lambda\norm{\bar{x}}_1+f(\bar{x})-\frac{1}{2L}\sum_{i:|w_i|>\lambda}(|w_i|-\lambda)^2.
% \end{align}
The bound~\eqref{ch2_eq:Lipschitz_ineq_L1} follows from the fact
$\lambda\norm{\bar{x}}_1+f(\bar{x})-\left(\lambda\norm{x^*}_1+f(x^*)\right)\leq\epsilon$.
%  we have
% \begin{align}
% \frac{1}{2L}\sum_{i:|w_i|>\lambda}(|w_i|-\lambda)^2\leq\epsilon.
% \end{align}
The bound~\eqref{ch2_eq:Lipschitz_ineq_L1} clearly implies that
$|w_i|\leq\sqrt{2L\epsilon}+\lambda$ for all $i$, i.e. $\norm{\grad
  f(\bar{x})}_\infty \leq \sqrt{2L\epsilon}+\lambda$.
\end{proof}

\begin{corollary}
\label{ch2_cor:grad_norm_bound}
Suppose $A \in \reals^{m\times n}$ with $m \leq n$ and full rank.
Let $P(x) = \lambda \norm{x}_1 + \frac{1}{2}
\norm{Ax-b-\lambda \theta}_2^2$.
Suppose $\bar{x}$ is $\epsilon$-optimal for
$\min_{x\in\reals^n}P(x)$, i.e. $0\leq P(\bar{x})- \min_{x \in
  \reals^n}P(x) \leq \epsilon$.
Then
\begin{equation}
  \label{eq:grad_norm_bound}
  \begin{array}{rcl}
    \norm{A^T(A\bar{x}-b-\lambda\theta)}_\infty & \leq & \sqrt{2\epsilon}\
    \sigma_{max}(A)+\lambda, \\
    \norm{A\bar{x}-b-\lambda\theta}_2 & \leq &
    \frac{\sqrt{n}}{\sigma_{min}(A)}\left(\sqrt{2\epsilon}\
      \sigma_{max}(A)+\lambda\right),
  \end{array}
\end{equation}
where $\sigma_{max}(A)$ denotes the maximum singular value of $A$.
\end{corollary}
\begin{proof}
Let $f(x)=\frac{1}{2}\norm{Ax-b-\lambda\theta}_2^2$,
then $\grad f(x)=A^T(Ax-b-\lambda\theta)$. For any $x, y \in\reals^n$, we have
\eq
\norm{\grad f(x)-\grad f(y)}_2 = \norm{A^TA(x-y)}_2 \leq \sigma_{max}^2(A) \norm{x-y}_2,
\en
where $\sigma_{max}(A)$ is the maximum singular-value of $A$. Thus,
$f:\Re^n\rightarrow \Re$ is a convex function and $\grad
f$ is Lipschitz continuous with the constant $L=\sigma_{max}^2(A)$.

Since  $\bar{x}$ is an $\epsilon$-optimal solution to
$\min_{x\in\reals^n}P(x) = \min_{x\in\reals^n}\{\lambda
\norm{x}_1+f(x)\}$, Theorem~\ref{ch2_thm:grad_convergence} %  guarantees that
% \eq
% \norm{\grad f(\bar{x})}_\infty=\norm{A^T(A\bar{x}-b-\lambda\theta)}_\infty\leq \sqrt{2L\epsilon}+\lambda=\sqrt{2\epsilon}~\sigma_{max}(A)+\lambda,
% \en
immediately implies the first bound in \eqref{eq:grad_norm_bound}. The
second bound follows from the fact that
\eq
\norm{A\bar{x}-b-\lambda\theta}_2  \leq
\frac{\norm{A^T(A\bar{x}-b-\lambda\theta)}_2}{\sigma_{min}(A)} \leq
\frac{\sqrt{n}}{\sigma_{min}(A)}\norm{A^T(A\bar{x}-b-\lambda\theta)}_\infty,
\en
where the first inequality follows the definition of $\sigma_{\min}(A)$
and the second from the bound $\norm{y}_2 \leq \sqrt{n}\norm{y}_{\infty}$
for all $y \in \reals^n$.
\end{proof}

\begin{lemma}
\label{ch2_lem:boundry_solution}
Let $f:\reals^n\rightarrow \reals$ be a strictly convex function and
$S\subset\reals^n$ be a closed, convex set. Let $x^\ast_S=\argmin_{x\in
  S}f(x)$ and $x^*=\argmin_{x\in\reals^n}f(x)$. Suppose the unconstrained
optimum $x^*\not\in S$, then
$x^\ast_S \in \partial S$, where $\partial S$ denotes the boundary of the
set $S$.
\end{lemma}
\begin{proof}
We will establish the result by contradiction. Suppose $x^\ast_S\in
\intr(S)$. Then, there exists an $\epsilon>0$ such that
$B_{\epsilon} = \{x\in\reals^n\;:\;\norm{x-x^\ast_{S}}_2<\epsilon\}\subset
S$. Since $f$ is strictly convex and $x^* \neq x^\ast_S$,
$f(x^*) < f(x^\ast_S)$.

Fix  $0<\lambda<
\frac{\epsilon}{\norm{\bar{x}-x^*}_2}<1$. Then $x_{\lambda} = \lambda x^* +
(1-\lambda) x^\ast_S \in B_{\epsilon} \subset S$. Since $f$ is
strictly convex,
\begin{align}
f(x_{\lambda})< \lambda f(x^*) + (1-\lambda) f(\bar{x}) < f(x_S^\ast).
\end{align}
% Since $x_{\lambda} \in B_{\epsilon} \subset X$
% and $f(x_{\lambda})<f(x^\ast_S)$
% contradicts the fact that $f(\bar{x})<f(x)$ for all $x\in X$.
This contradicts the fact that $x^\ast_S = \argmin_{x\in S}\{f(x)\}$.
Thus, $x^\ast_S \in S \backslash \intr(S) = \partial S$.
%  Since $\bar{x}\in X$, then it must be true that
% $\bar{x}\in \partial X$.
\end{proof}

\begin{lemma}
\label{ch2_lem:constrained_shrinkage}
Fix $y \in \reals^n$, $\lambda > 0$ and $\eta > 0$. Let $P(x) =
\lambda\norm{x}_1+\frac{1}{2}\norm{x-y}_2^2$ and
\begin{equation}
  \label{ch2_eq:constrained_shrinkage_problem}
  x^*=\argmin\{P(x):\norm{x}_1\leq\eta\}.
\end{equation}
Then the deterministic
complexity of computing $x^*$ is $\cO(n\log(n))$, and the randomized
complexity is~$\cO(n)$.
\end{lemma}
\begin{proof}
% Since $P(x)$ is strongly convex and
Since $P(x)$ is strongly convex,
\eqref{ch2_eq:constrained_shrinkage_problem} has a unique primal
optimal solution.
Also, since the optimization problem
\eqref{ch2_eq:constrained_shrinkage_problem}
satisfies Slater's constraint qualification, strong duality
holds, and since the primal optimal
value bounded, the dual optimal value is attained. % Let
% $x^*\in\reals^n$ and $\alpha^*\in\reals$ denote, respectively, the  primal
% and dual optimal solutions.

Let
\begin{align}
\mathcal{L}(x,\alpha)& =
\lambda\norm{x}_1+\frac{1}{2}\norm{x-y}_2^2+\alpha(\norm{x}_1-\eta),\\
&= (\lambda+\alpha)\norm{x}_1+\frac{1}{2}\norm{x-y}_2^2-\alpha\eta,
\end{align}
denote the Lagrangian function. Since strong duality holds, $x^*$ is a
minimizer of $\mathcal{L}(x,\alpha^*)$, where $\alpha^\ast$ denote the
optimal dual solution.
Since $\mathcal{L}(x,\alpha^*)$ is a strictly convex function of $x$,
$x^*$ is
the unique minimizer of $\mathcal{L}(x,\alpha^*)$.
Let
\begin{align}
\label{eq:x_alpha}
x^*(\alpha)& = \argmin_{x\in\reals^n}\mathcal{L}(x,\alpha)=
\mbox{sign}(y)\odot\max\{|y|-(\lambda+\alpha), 0\}.
\end{align}
It is clear that $x^\ast = x^*(\alpha^*)$. In the rest of this proof, we
show how to efficiently compute $\alpha^\ast$.

Note that $x^*(0)=\mbox{sign}(y)\odot\max\{|y|-\lambda, 0\}$ is the unique
unconstrained minimizer of $P(x)$. When $\norm{x^*(0)}_1\leq\eta$, then
trivially $x^*=x^*(0)$. However, when $\norm{x^*(0)}_1>\eta$,
Lemma~\ref{ch2_lem:boundry_solution} implies that $x^* \in \partial
\{x\in\reals^n\;:\; \norm{x}_1\leq
\eta\}$, i.e. $\norm{x^*}_1=\eta$. Therefore,
\begin{align}
\alpha^*\in\{\alpha>0:\;\norm{x^*(\alpha)}_1=\eta\}. \label{ch2_eq:optimal_alpha_set}
\end{align}
From (\ref{eq:x_alpha}), it follows that
\begin{align}
\norm{x^*(\alpha)}_1 =
\sum_{i:|y_i|- \lambda\geq\alpha}
((|y_i|-\lambda)-\alpha)=\sum_{i=1}^n(|x^*(0)|-\alpha)^+. \label{ch2_eq:euclidean_projection}
\end{align}
Note that $\norm{x^*(\alpha)}_1$ is a strictly decreasing continuous
function of $\alpha$. Since $\norm{x^*(0)}_1>\eta$, there exists a
unique $\hat{\alpha}>0$ such that
$\norm{x^*(\hat{\alpha})}_1=\eta$. From \eqref{ch2_eq:optimal_alpha_set}, we
can conclude that $\alpha^*=\hat{\alpha}$.

%For $\alpha\geq 0$,
%\begin{align}
%\norm{x^*(\alpha)}_1=\sum_{i:|x_i^*(0)|\geq\alpha}(|x_i^*(0)|-\alpha). \label{ch2_eq:euclidean_projection}
%\end{align}
To compute $\hat{\alpha}$ such that $\norm{x^*(\hat{\alpha})}_1=\eta$,
sort $z = |x^*(0)|$ in decreasing order. Let $z_{[i]}$ denote the
$i$-th largest component of $z$. It is clear that
$\norm{x^*(w_{[n]})}_1>\eta>0$ and $\norm{x^*(\alpha)}_1=0$ for all
$\alpha > w_{[1]}$. Hence, there exists an index  $1\leq k < n$ such that
$\norm{x^*(w_{[k]})}_1\leq\eta$ and $\norm{x^*(w_{[k+1]})}_1>\eta$, and it
follows that
\begin{align}
\alpha^* = \frac{1}{k}\left(\sum_{j=1}^{k}w_{[j]}-\eta\right).
\end{align}
Thus, $x^*=x^*(\alpha^*)$ can be computed in $O(n\log(n))$
operations. Singer et al~\cite{Duc08_1C} show that $\alpha^*$ with a
$\cO(n)$ randomized complexity
%such that
%\begin{align}
%\norm{x^*(\alpha^*)}_1=\sum_{i:|x_i^*(0)|\geq\alpha^*}(|x_i^*(0)|-\alpha^*)=\eta,
%\end{align}
using a slightly modified version of the
randomized median finding algorithm.
\end{proof}
\end{document}

%% file: defs_garud_edit.tex
\def\grad{\nabla}

\def\cK{\mathcal{K}}

\def\cN{\mathcal{N}}
\def\cO{\mathcal{O}}

\def\cX{\mathcal{X}}

\def\mE{\mathbb{E}}

\def\eq{\[}
\def\en{\]}

\def\smskip{\smallskip}

\def\texitem#1{\par\smskip\noindent\hangindent 25pt
               \hbox to 25pt {\hss #1 ~}\ignorespaces}

\def\abs#1{\left|#1\right|}
\def\norm#1{\|#1\|}

\newcommand{\BEAS}{\begin{eqnarray*}}
\newcommand{\EEAS}{\end{eqnarray*}}
\newcommand{\BEA}{\begin{eqnarray}}
\newcommand{\EEA}{\end{eqnarray}}
\newcommand{\BEQ}{\begin{eqnarray}}
\newcommand{\EEQ}{\end{eqnarray}}
\newcommand{\BIT}{\begin{itemize}}
\newcommand{\EIT}{\end{itemize}}
\newcommand{\BNUM}{\begin{enumerate}}
\newcommand{\ENUM}{\end{enumerate}}

\newcommand{\BA}{\begin{array}}
\newcommand{\EA}{\end{array}}

\newcommand{\reals}{\mathbb{R}}
\newcommand{\integers}{\mathbb{Z}}

\newcommand{\argmin}{\mathop{\rm argmin}}

\newcommand{\dom}{\mathop{\bf dom}}

\newcommand{\intr}{\mathop{\bf int}}

\newif\ifpagenumbering
\pagenumberingtrue

\pagenumberingfalse

\newsavebox{\remarkbox}
\savebox{\remarkbox}{\noindent\bf Remark}
\newcommand{\proc}[1]{\textnormal{\scshape#1}}

\def\etak{\eta^{(k)}}
\def\xbp{x_{\ast}}
\def\xk{x^{(k)}}
\def\xkopt{x^{(k)}_{\ast}}
\def\lk{\lambda^{(k)}}
\def\tk{\tau^{(k)}}
\def\epsk{\epsilon^{(k)}}
\def\thetak{\theta^{(k)}}

\def\Nk{N^{(k)}}

\def\tkp{\tau^{(k+1)}}
\def\epskp{\epsilon^{(k+1)}}
\def\xik{\xi^{(k)}}

\def\pmk{p_\mu^{(k)}}
\def\fnk{f_\nu^{(k)}}
\def\Pk{P^{(k)}}

\def\sk{s^{(k)}}

\def\skopt{s^{(k)}_{\ast}}

\def\xik{\xi^{(k)}}

\renewcommand{\argmin}{\operatornamewithlimits{argmin}} 